\documentclass[final,onefignum,onetabnum]{siamart190516}

\usepackage{lipsum}
\usepackage{amsmath,amssymb,amsfonts,latexsym,stmaryrd}

\usepackage{mathrsfs} 
\usepackage{color}
\usepackage{epstopdf}
\usepackage{algorithmic}

\usepackage{multirow}
\ifpdf
  \DeclareGraphicsExtensions{.eps,.pdf,.png,.jpg}
\else
  \DeclareGraphicsExtensions{.eps}
\fi

\usepackage{pst-node}

\usepackage{graphicx,float,epsfig}

\def\O{\Omega}

\def\mV{\mathcal{V}}

\renewcommand\sp{\mathop{\mathrm{Sp}}\nolimits}

\newcommand{\set}[1]{\lbrace #1 \rbrace}

\newcommand{\jump}[1]{\llbracket #1 \rrbracket}
\newcommand{\mean}[1]{\{#1\}}

\newcommand{\norm}[1]{\lVert#1\rVert}

\newcommand{\n}{\boldsymbol{n}}
\usepackage{booktabs}
\usepackage{float}

\newcommand{\bP}{{\bf P}}

\newcommand{\bT}{{\bf T}}

\def\R{\mathbb R}

\def\H{\mathrm H}

\def\L{\mathrm L}

\def\div{\mathop{\mathrm{div}}\nolimits}

\def\bu{\boldsymbol{u}}

\def\bv{\boldsymbol{v}}

\def\bf{\boldsymbol{f}}

\def\bT{\boldsymbol{T}}

\def\bP{\boldsymbol{P}}

\def\b0{\boldsymbol{0}}

\def\btau{\boldsymbol{\tau}}

\def\bw{\boldsymbol{w}}

\def\mG{\mathcal{G}}

\def\mK{\mathcal{K}}
\def\mV{\mathcal{V}}

\renewcommand\sp{\mathop{\mathrm{Sp}}\nolimits}

\def\div{\mathop{\mathrm{div}}\nolimits}

\newtheorem{remark}{Remark}[section]

\renewcommand{\theequation}{\thesection.\arabic{equation}}
\def\HsO{{\H^{s}(\O)}}
\def\LO{\L^2(\O)}

\newcommand\cT{\mathcal{T}}
\newcommand{\cF}{\mathcal{F}}

\newcommand\cP{\mathcal{P}}
\newcommand\DO{\partial\O}

\newcommand{\vertiii}[1]{{\left\vert\kern-0.25ex\left\vert\kern-0.25ex\left\vert #1 
    \right\vert\kern-0.25ex\right\vert\kern-0.25ex\right\vert}}

\renewcommand\H{\mathrm{H}}
\renewcommand\L{\mathrm{L}}

\renewcommand\div{\mathop{\mathrm{div}}\nolimits}

\renewcommand\sp{\mathop{\mathrm{sp}}\nolimits}

\newcommand\gap{\widehat{\delta}}

\newsiamremark{hypothesis}{Hypothesis}
\crefname{hypothesis}{Hypothesis}{Hypotheses}
\newsiamthm{claim}{Claim}

\headers{DG methods for the  acoustic vibration problem}{Felipe Lepe, David Mora, Jesus Vellojin}

\title{Discontinuous Galerkin methods for the acoustic vibration problem\thanks{Submitted to the editors DATE.
\funding{The first   author  was partially supported by  DIUBB through project 2120173 GI/C Universidad del B\'io-B\'io and by 
	ANID-Chile through FONDECYT project 11200529 (Chile). The second author was partially supported by DIUBB
	through project 2120173 GI/C Universidad del B\'io-B\'io and  project Centro
	de Modelamiento Matemático (CMM), FB210005, BASAL funds for centers of excellence. The second and third author were partially supported by ANID-Chile through project Anillo of Computational Mathematics for Desalination Processes ACT210087. }}}

\author{Felipe Lepe\thanks{GIMNAP-Departamento de Matem\'atica, Universidad del B\'io - B\'io, Casilla 5-C, Concepci\'on, Chile. \email{flepe@ubiobio.cl}}
\and David Mora\thanks{GIMNAP-Departamento de Matem\'atica, Universidad del B\'io - B\'io, Casilla 5-C, Concepci\'on, Chile and C$\textrm{I}^2$MA, Universidad de Concepci\'on, Chile. \email{dmora@ubiobio.cl}}
\and Jesus Vellojin\thanks{GIMNAP-Departamento de Matem\'atica, Universidad del B\'io - B\'io, Casilla 5-C, Concepci\'on, Chile. \email{jvellojin@ubiobio.cl}}}

\usepackage{amsopn}

\ifpdf
\hypersetup{
  pdftitle={DG methods for the  acoustic vibration problem},
  pdfauthor={Felipe Lepe}
}
\fi

\usepackage{diagbox}
\begin{document}

\maketitle

\begin{abstract}
In two and three dimension we analyze discontinuous Galerkin methods for the acoustic problem. The acoustic fluid  that we consider on this paper is inviscid, leading to a linear eigenvalue problem. The acoustic problem is written, in first place,  in terms of the displacement. Under the approach of the non-compact operators theory, we prove convergence and error estimates for the method when the displacement formulation is considered. We analyze the influence of the stabilization parameter on the computation of the spectrum, where spurious eigenmodes arise when this parameter is not correctly chosen. Alternatively we present the formulation depending only on the pressure,  comparing the performance of the DG methods  with the pure displacement formulation. Computationally, we study the influence of the stabilization parameter on the arising of spurious eigenvalues when the spectrum is computed. Also, we report tests in two and three dimensions where convergence rates are reported, together with a comparison between the displacement and pressure formulations for the proposed DG methods.
\end{abstract}

\maketitle

\section{Introduction}  
In this paper we are interested in the analysis of a discontinuous Galerkin method (DG) to
approximate the natural frequencies of the acoustic problem. The method is 
based in a interior penalization strategy, in which we focus our attention 
since the aforementioned parameter may introduce spurious eigenvalues on the computed spectrum.

The DG discretization for eigenvalues is a ongoing research topic where the pioneer work
related to this subject is \cite{MR2220929}, where under suitable norms that depend on jumps and averages, the authors 
 proved spectral correctness and convergence of the DG method for the Laplace eigenvalue problem for both, the primal and mixed formulations.
 Also, this paper consider symmetric and nonsymmetric DG methods. This study inspired a number of works where DG methods have been considered for other spectral problems \cite{MR2324460, MR2263045, MR2511729, MR3962898, MR4077220}.
 
The contribution of the present work is the development of DG methods for the acoustic problem. This problem is one of the most important on engineering and physics, since the correct knowledge of some acoustic fluids, allows to design different structures as containers, motors, vehicles, aircrafts,  etc., or noise reduction devices, necessary on the industry. These applications have motivated the development of different numerical methods and approaches  such as those of \cite{MR3660301,MR2264552, MR2086168,MR2045461, MR3660301, MR4364794, MR4253807, MR2265217}. Some specific applications that deserve to be  mentioned are those related  to space programs, where  mechanical effects of shock, acceleration, and acoustic vibrations are taken into account on the design of durable structures \cite{osiander2017wireless}. Also, there is the ultrasonic welding, where ultrasonic acoustic vibrations are used to generate heat and pressure to weld two work pieces together \cite{kirstein2013multidisciplinary}. Another use is found in the acoustic emission detectors, which deals with leaking on pipelines by using the energy signal produced by acoustic vibrations when the fluid escapes \cite{bai2014subsea}.
 
 In particular, we are interested in the interior penalization discontinuous Galerkin method (IPDG) in order to obtain a suitable alternative to approximate accurately the eigenfunctions and eigenvalues of the acoustic vibration problem.  In recent years, the IPDG has been applied to different eigenvalue problems as for instance, the Maxwell spectral problem \cite{MR2263045, MR2324460}, the Stokes eigenvalue problem \cite{MR4077220}, or the elasticity eigenvalue problem \cite{MR3962898}, where the accuracy on the approximation of the spectrum is remarkable. The IPDG method results to be easy to handle from the computational point of view. This allows the consideration of arbitrary polynomial degrees in two dimensional (resp. three dimensional) convex (resp. non-convex) domain, together with single or mixed boundary conditions. However, the main drawback of the IPDG methods for eigenvalue problems is the correct choice of the penalization parameter, since an incorrect selection of the parameter may introduce spurious eigenvalues. This is a inherent concern that comes with numerical methods that depend on some stabilizations, as for instance the virtual element method (VEM), where in \cite{MR3660301,MR3340705, MR4229296}, the authors show in their numerical tests  that spurious eigenvalues arise when the stabilization parameter changes. The same is observed with IPDG methods such as those of \cite{MR3854050,MR4077220}, where an analysis of spurious eigenvalues is performed. This is important to take into account, since in real applications, the correct frequencies must be well approximated to avoid undesirable physical phenomenons, as for instance, the resonance of certain structures, or the incorrect development of noise insulators.
 
The available  literature for IPDG methods for eigenvalue problems indicates that the correct choice of the  aforementioned stabilization parameters depend strongly on the configuration of the problem of interest. This implies that, as the development of the theoretical well posedness of the problem is analyzed, the domain, physical parameters and boundary conditions, need to be taken into account. Also, the polynomial degree for the approximation plays an important role when the stabilization is chosen, since when the polynomial approximation increases, the pollution of the spectrum vanishes, as the computational results available on, for instance, \cite{MR3854050,MR4077220} suggests. 
 
 With these considerations at hand, our aim is to prove that a IPDG approach is a suitable technique to approximate the natural frequencies of the acoustic eigenvalue problem, where the numerical method, whether symmetric or nonsymmetric, approximates accurately the spectrum, and the stabilization parameter can be chosen with no major problems, according to the needs of the problem. In this sense, we will apply our IPDG methods for the two classic formulations of the acoustic problem: the displacement and pressure formulations, in order to compare the performance of the discontinuous  numerical schemes when different unknowns are considered in the mathematical formulation.

The paper is organized as follows: in Section~\ref{MAIN_PROBLEM}, we
present the acoustic eigenvalue problem, the displacement formulation and summarize some
important results such as the solution operator,  regularity properties and the corresponding spectral characterization.
Section~\ref{sec:DG} contains the development of the DG methods, presenting definitions associated to the mesh elements, norms
and technical results that are needed for the analysis. Also, we introduce the DG discretization for the acoustic problem and
the discrete solution operator that will approximate the continuous one. Also, the spectral analysis is presented, where approximation results and error estimates for the eigenvalues and eigenfunctions are presented.
 In Section \ref{sec:pressure} we introduce the pressure formulation and its DG discretizations, which also are symmetric and nonsymmetric.  Finally, in Section~\ref{sec:numerics}, we perform a number of numerical tests in two and three dimensions where we analyze, in one hand, the influence of the stabilization parameters on the computation of the spectrum, performing a spurious analysis and, on the other hand, computation of orders of convergence for the eigenvalues. In this section we also compare the performance of the displacement and pressure DG formulations with a benchmark test, where the computational cost is mesaured trough the cpu-time, sparsity and problem size.
 
 \subsection{Notations}
Throughout this work, $\O$ is a generic Lipschitz bounded domain of $\R^\texttt{d}$, where $\texttt{d}\in\{2,3\}$. For $s\geq 0$,
$\norm{\cdot}_{s,\O}$ stands indistinctly for the norm of the Hilbertian
Sobolev spaces $\HsO$ or $\HsO^\texttt{d}$ with the convention
$\H^0(\O):=\LO$.  If ${X}$ and ${Y}$ are normed vector spaces, we write ${X} \hookrightarrow {Y}$ to denote that ${X}$ is continuously embedded in ${Y}$. We denote by ${X}'$ and $\|\cdot\|_{{X}}$ the dual and the norm of ${X}$, respectively.
Finally,
we employ $\boldsymbol{0}$ to denote a generic null vector and
the relation $\texttt{a} \lesssim \texttt{b}$ indicates that $\texttt{a} \leq C \texttt{b}$, with a positive constant $C$ which is independent of $\texttt{a}$, $\texttt{b}$, and the size of the elements in the mesh. The value of $C$ might change at each occurrence. We remark that we will write the constant $C$ only when is needed.
 %
%

\section{The model problem}
\label{MAIN_PROBLEM}
 The acoustic problem reads
as follows: Find $\omega\in\mathbb{R}$, the displacement $\bu$ and the pressure $p$ on a domain $\Omega\subset\mathbb{R}^{\texttt{d}}$, such that 
\begin{equation}\label{def:acustica}
\left\{
\begin{array}{rcll}
\nabla p-\omega^2\rho(x)\bu & = & \boldsymbol{0}&\text{in}\,\O\\
p+\rho(x) c^2\div\bu & = & 0 &\text{in}\,\O\\
\bu\cdot\boldsymbol{n}&=&0&\text{on}\,\partial\O,
\end{array}
\right.
\end{equation}
where $\rho(x)$ is the density defined in $x\in\O$, $c$ is the sound speed and $\boldsymbol{n}$ is the outward unitary vector. As an assumption on the density, we consider two constants $\overline{\rho}$ and $\underline{\rho}$, both positives,  such that,  $\underline{\rho}\leq\rho(x)\leq\overline{\rho}$, for $x\in\O$.

Let us introduce  the displacement formulation for \eqref{def:acustica}.
Let us define the space
$$
\mathcal{V}:=\{\boldsymbol{v}\in \H(\div,\O)\,:\, \boldsymbol{v}\cdot\boldsymbol{n}=0\,\,\text{on}\,\,\partial\O\}.
$$

With this space at hand, a variational formulation for the acoustic problem is the following: Find $(\omega,\bu)\in\mathbb{R}\times\mV$, with $\bu\neq \boldsymbol{0}$, such that
\begin{equation}
\label{eq:acustic1}
\displaystyle\int_{\O}\rho(x) c^2\div\bu\div\boldsymbol{v}=\omega^2\int_{\O}\rho(x)\bu\cdot\bv\quad\forall\bv\in\mV.
\end{equation}
Defining $\widehat{\lambda
}:=\omega^2$, we immediately deduce that the eigenspace associated to $\widehat{\lambda}=0$ is
$$
\mK:=\{\bv\in\mV\,:\,\div\bv=0\,\,\text{in}\,\O\}.
$$
We observe that \eqref{eq:acustic1} is not an elliptic problem. Hence, with a shift argument we obtain the following problem: Find $(\lambda,\bu)\in\mathbb{R}\times\mV$, with $\bu\neq \boldsymbol{0}$, such that
$$
\displaystyle\int_{\O}\rho(x) c^2\div\bu\div\boldsymbol{v}+\int_{\O}\rho(x)\bu\cdot\bv=\lambda\int_{\O}\rho(x)\bu\cdot\bv\quad\forall\bv\in\mV,
$$
where $\lambda:=1+\widehat{\lambda}$.

Let us  introduce the bilinear forms $a:\mV\times\mV\rightarrow\mathbb{R}$ and $b:\mV\times\mV\rightarrow\mathbb{R}$ defined by
$$
\displaystyle a(\bv,\bw):=\int_{\O}\rho(x) c^2\div\bv\div\bw+\int_{\O}\rho(x)\bv\cdot\bw\quad\forall\bv,\bw\in\mV.
$$
and
$$
b(\bv,\bw):=\int_{\O}\rho(x) \bv\cdot\bw,\quad\forall\bv,\bw\in\mV.
$$
With these bilinear forms at hand, problem \eqref{eq:acustic1} is rewritten as follows: Find $(\lambda,\bu)\in\mathbb{R}\times\mV$, with $\bu\neq \boldsymbol{0}$, such that
\begin{equation}
\label{eq:acustic21}
a(\bu,\bv)=\lambda b(\bu,\bv)\quad\forall\bv\in \mV.
\end{equation}

Clearly $a(\cdot,\cdot)$ is continuous, in the sense that there exists a positive constant $M:=\|\rho\|_{\infty,\O}\min\{c^2,1\}$ such that $|a(\bu,\bv)|\leq M\|\bu\|_{\div,\O}\|\bv\|_{\div,\O}$.
Observe that for any $\bv\in\mV$ there holds
\begin{equation}
\label{elipticCont}
a(\bv,\bv)\geq \underline{\rho} c^2\|\div\bv\|_{0,\O}^2+\underline{\rho}\|\bv\|_{0,\O}^2\geq\underbrace{\min\{\underline{\rho} c^2,\underline{\rho}\}}_{\alpha^c}\|\bv\|_{\div,\O}^2,
\end{equation}
implying the coercivity of $a(\cdot,\cdot)$ with a constant $\alpha^c:=\min\{\underline{\rho} c^2,\underline{\rho}\}$.
This allows us to introduce  the so called solution operator associated to \eqref{eq:acustic21}, defined by $\bT:\mV\rightarrow\mV$ where  for a given  $\bf\in\mV$, we have $\bT\bf=\widehat{\bu}$, where $\widehat{\bu}$ is the solution of the source problem 
$$
a(\widehat{\bu},\bv)= b(\bf,\bv)\quad\forall\bv\in \mV.
$$
Note that operator $\bT$ is well defined and  bounded thanks to Lax-Milgram's lemma, non-compact, and selfadjoint respect to $a(\cdot,\cdot)$ and $b(\cdot,\cdot)$. Moreover, $(\lambda,\bu)$ is a solution of \eqref{eq:acustic1} if and only if $(1/\lambda,\bu)$ is an eigenpair of $\bT$.

Let us remark that  since we have the space $\mK$ at hand, there exists an operator $\bP:\mV\rightarrow\mV$ such that the space $\mV$ can be decomposed as a direct sum given by $\mV=\mK\oplus \bP(\mV)$, where $\bP(\mV)$ is the orthogonal complement of $\mK$ with respect to the inner product defined in $\mV$.

The following technical results is available in \cite{MR3660301, MR1342293}.
\begin{lemma}[Orthogonal complement]
\label{lmm:G}
The space $\bP(\mV)$ is defined by $\bP(\mV):=\{\nabla\,q\,:\, q\in \H^1(\O)\}$ and satisfies
$\mV=\mK\oplus \bP(\mV)$, which is a orthogonal decomposition in $\L^2(\O)^\texttt{d}$ and $\H(\div,\O)$. Moreover, there exists $s\in (1/2,1]$ such that for all $\bv\in\mV$, if $\bv=\boldsymbol{\varphi}+\nabla q$, with $\boldsymbol{\varphi}\in\mK$ and $\nabla q\in \bP(\mV)$, then $\nabla q\in\H^s(\O)^\texttt{d}$ and $\|\nabla q\|_{s,\O}\lesssim\|\div\bv\|_{0,\O}.$
\end{lemma}

Now we recall some important properties that $\bT$ satisfies.
\begin{lemma}[Invariance of $\bT$ on $\bP(\mV)$]
\label{lmm:invariant}
There holds $\bT(\bP(\mV))\subset \bP(\mV).$

Also, there holds
$$
\bT(P(\mV))\subset\left\{\bv\in \H^s(\O)^\texttt{d}\,:\,\div\bv\in \H^1(\O)\right\},
$$
and   for all $\bf\in\mG\cap\mV$, if $\bu=\bT\bf$, the following estimate holds
$$
\|\bu\|_{s,\O}+\|\div\bu\|_{1+s,\O}\lesssim \|\bf\|_{\div,\O},
$$
where this estimate is obtained with a bootstrap argument. Moreover, we have that 
$\bT|_{\bP(\mV)}:\bP(\mV)\rightarrow \bP(\mV)$ is compact.
\end{lemma}

Finally, we present the following spectral characterization of $\bT$.
\begin{theorem}
The spectrum of $\bT$ decomposes as $\sp(\bT)=\{0,1\}\cup\{\mu_k\}_{k\in\mathbb{N}}$ where
\begin{enumerate}
\item $\mu=1$ is an infinite-multiplicity eigenvalue of $\bT$ and its associated 
eigenspace is $\mK$,
\item $\{\mu_k\}_{k\in\mathbb{N}}\subset (0,1)$ is a sequence of finite-multiplicity eigenvalues if $\bT$ which converge to zero and if $\bu$ is an eigenfunction of $\bT$ associated with such an eigenvalue, then there exists $\widehat{s}>1/2$, depending on the eigenvalue, such that
$$
\|\bu\|_{\widehat{s},\O}+\|\div\bu\|_{1+\widehat{s},\O}\lesssim\|\bu\|_{\div,\O},
$$
where the hidden constant also depends on the eigenvalue.
\item $\mu=0$ is not an eigenvalue of $\bT$.
\end{enumerate}
\end{theorem}

With this result at hand, we are in a position to introduce the numerical method of our interest.

\section{The DG method}
\label{sec:DG}
 To begin, we need to set some definitions and preliminaries.
For $\Omega\subset\mathbb{R}^\texttt{d}$, for $\texttt{d}\in\{2,3\}$,  let $\mathcal{T}_h$ be a shape regular family of meshes which subdivide the domain $\bar \Omega$ into  
triangles/tetrahedra $K$.  We denote by  $h_K$ 
the diameter of the element $K\in\cT_h$ and $h$ the maximum of the diameters of all the
elements of the mesh, i.e. $h:= \max_{K\in \cT_h} \{h_K\}$. Let $F$ be a closed set such  that  $F\subset \overline{\Omega}$
is an interior edge/face if $F$ has a positive $(n-1)$-dimensional 
measure and if there are distinct elements $K$ and $K'$
such that $F =\bar K\cap \bar K'$. A closed 
subset $F\subset \overline{\Omega}$ is a boundary edge/face if
there exists $K\in \cT_h$ such that $F$ is an edge/face
of $K$ and $F =  \bar K\cap \partial \Omega$. 
Let $\cF_h^*$ and $\cF_h^\partial$ be  the sets  of interior edges/faces
and  boundary edges/face, respectively.

For any $t\geq 0$, we define the following broken Sobolev space 
\[
 \H^t(\cT_h)^\texttt{d}:=
 \set{\bv \in \L^2(\O)^\texttt{d}: \quad \bv|_K\in \H^t(K)^\texttt{d}\quad \forall K\in \cT_h}.
\]
%
Also, the space of the skeletons of the triangulations
$\cT_h$ is defined  by $ \L^2(\cF_h):= \prod_{F\in \mathcal{F}_h} \L^2(F),$ where $\cF_h:=\cF_h^*\cup\cF_h^{\partial}$.

In the forthcoming analysis, $h_\cF\in \L^2(\cF_h)$
will represent the piecewise constant function 
defined by $h_\cF|_F := h_F$ for all $F \in \cF_h$,
where $h_F$ denotes the diameter of edge/face $F$.

Let  $\cP_m(\cT_h)$ be  the space of piecewise polynomials respect with to
 $\cT_h$  of degree at most $m\geq 0$;
namely,  
\[
 \cP_m(\cT_h) :=\left\{ v\in \L^2(\O):\, v|_K \in \cP_m(K)\,\,\, \,\,\forall K\in \cT_h \right\}. 
\]

For any $k\geq 1$, we define the finite element spaces
$\mathcal{V}_h := \cP_k(\cT_h)^{\texttt{d}}$ and $\mathcal{V}_h^c := \mathcal{V}_h \cap \mathcal{V}$.
We observe that the space $\mathcal{V}_h^c$ is the Brezzi-Douglas-Marini (BDM)
finite element space. Now, we recall some well-known properties
of the space $\mathcal{V}_h^c$.

We define averages $\mean{\bv}\in \L^2(\cF_h)^{\texttt{d}}$
and jumps $\jump{\bv}\in \L^2(\cF_h)$ 
as follows
\[
 \mean{\bv}_F := (\bv_K + \bv_{K'})/2 \quad \text{and} \quad \jump{\bv}_F := \bv_K \cdot\n_K + \bv_{K'}\cdot\n_{K'} 
 \quad \forall F \in \cF(K)\cap \cF(K'),
\]
where $\n_K$ is the outward unit normal vector to $\partial K$.
Also, on the boundary  $\partial\O$ and for all  $F \in \cF(K)\cap \DO$,
the averages and jumps are defined by $ \mean{\bv}_F := \bv_K $
and  $\jump{\bv}_F := \bv_K \cdot\boldsymbol{n} $, respectively.

Before concluding this section, we shall note some results required for the analysis that follows. For a vector field $\bv \in \mathcal{V}_h$ we define $\div_h \bv \in \L^2(\O)$ by 
$
\div_h \bv|_{K} := \div (\bv|_K)$ for all $K\in \cT_h
$ and  we endow $\mathcal{V}(h) := \mathcal{V} + \mathcal{V}_h$ with the following seminorm
\[
 |\bv|^2_{\mathcal{V}(h)} := \|\div_h \bv\|^2_{0,\O} + \|h_{\cF}^{-1/2} \jump{\bv}\|^2_{0,\cF^*_h},
\]
which is well defined in $\mathcal{V}(h)$ and the norm
$
 \|\bv\|^2_{\textrm{DG}} := |\bv|^2_{\mathcal{V}(h)} + \|\bv\|^2_{0,\O}.
$
Another important ingredient for the analysis is the following discrete trace inequality \cite{MR2882148}:
\begin{equation}\label{discTrace}
\|h^{1/2}\mean{v}\|_{0,\cF}\lesssim \|v\|_{0,\O}\qquad\forall v\in\cP_k(\mathcal{T}_h).
\end{equation}
\subsection{The DG acoustic problem}
Now we  introduce the DG discretization for the acoustic problem \eqref{eq:acustic1}, which reads as follows: Find $\lambda_h:=1+\widehat{\lambda}_h\in\mathbb{C}$
and $\boldsymbol{0}\neq\bu_h\in\mathcal{V}_h$ such that
\begin{equation}\label{spectdisc}
a_h(\bu_h,\btau_h)=\lambda_h b(\bu_h,\btau_h)\quad\forall\btau_h\in\mathcal{V}_h,
\end{equation}
where the bilinear form $a_h:\mathcal{V}_h\times\mathcal{V}_h\rightarrow\mathbb{C}$ is defined by
\begin{multline}
\label{ah}
a_h(\bu_h,\bv_h):=\int_{\O}\rho(x) c^2\div_h\bu_h\div_h\bv_h+\int_{\O}\rho(x)\bu_h\cdot\bv_h\\
+ \int_{\cF^*_h} \frac{\texttt{a}_S}{ h_{\cF}}\, \jump{\bu_h}\cdot \jump{\bv_h}-\int_{\cF^*_h} 
\mean{\rho(x) c^2\div_h\bu_h} \cdot \jump{\bv_h} 
-\varepsilon\int_{\cF^*_h} \mean{\rho(x) c^2\div_h\bv_h} \cdot \jump{\bu_h},
\end{multline}
where $\texttt{a}_S>0$ is the so called stabilization parameter
and $\varepsilon\in\{-1,0,1\}$. These parameters are relevant for our DG methods. In one hand, the stabilization 
$\texttt{a}_S$ is a parameter that, if is not correctly chosen, could yield to instabilities for the method. In general, this parameter  depends on theoretical and physical constants and, in the numerical performance, also on the polynomial degrees. (cf. Section \ref{sec:numerics}). On the other hand, $\varepsilon$ establishes if the DG method is symmetric or nonsymmetric. More precisely,   if $\varepsilon=1$ we obtain
the classic symmetric interior penalty method (SIP), 
if $\varepsilon=-1$ we obtain the non-symmetric interior penalty method (NIP)
and if $\varepsilon=0$ the incomplete interior penalty method (IIP) (see \cite{MR2220929, MR4077220, MR3962898} for instance).

For the analysis, we introduce the following norm
$$
\|\bu\|^{*}_{\textrm{DG}}:=\big(\|\bu\|^2_{\textrm{DG}}
+\| h_{\cF_h}^{1/2}\mean{\div_h\bu}\|^2_{0,\cF_h^{*}}\big)^{1/2}.
$$
It is easy to check that  for all $\bu,\bv\in\mV(h)$,
and $\div\bu,\div\bv\in\H^t(\O)$
with $t>1/2$, the bilinear form $a_h(\cdot,\cdot)$
is bounded. In fact, there exists a constant $C>0$,
independent of $h$, such that $\big|a_h(\bu,\bv)\big|\leq C\|\bu\|_{\textrm{DG}}^{*}\|\bv\|_{\textrm{DG}}^{*}$.
Moreover, by means of \eqref{discTrace} it is possible to prove
that for all $\bv_h\in\mathcal{V}_h$ there exists a positive constant $M_{\textrm{DG}}$ such that $\vert a_h (\bu,\bv_h)\vert\leq M_{\textrm{DG}}\|\bu\|_{\textrm{DG}}^*\|\bv_h\|_{\textrm{DG}}$.

In order to analyze the discrete eigenvalue problem \eqref{spectdisc},
we need to decompose the space $\mV^c_h$.
With this aim, we consider the following subspace of $\mathcal{K}$
\[
\mathcal{K}_h : = \left\{\bv_h \in \mV^c_h:\,\, \div \bv_h    = 0  \right\}.
\]
With $\mK_h$ at hand, we are interested in its orthogonal complement at the discrete level. With this in mind, is necessary to introduce the discrete counterpart of the operator $\bP$, which we denote by $\bP_h$ and is defined by $\bP_h:\mV_h^c\rightarrow\mV_h^c$. With this operator at hand, the following approximation result holds, an its proof is identical to \cite[Lemma 3.1]{MR4077220}.
\begin{lemma}
The operator $\bP_h:\mV_h^c\rightarrow\mV_h^c$ is a projection with kernel $\mK_h$ and satisfies the following estimate
$$
\|(\bP-\bP_h)\bv_h\|_{\div,\O}\lesssim h^s\|\div\bv_h\|_{0,\Omega}\quad\forall\bv_h\in\mV_h^c,
$$
where the hidden constant is independent of $h$ and $s$ is as in Lemma \ref{lmm:G}.
\end{lemma}

The following technical result proved in \cite[Proposition~5.2]{MR3376135} for tensorial fields, and that is easily adapted for vectorial fields, 
will be useful in the forthcoming analysis.
\begin{proposition}\label{propC}
There exist a projection $\mathcal{I}_h:\, \mV_h \to \mV_h^c$
and two constants $\underbar{C},\bar{C} >0$,
independent of $h$, such that 
\begin{equation}\label{equivN}
\underbar{C}\,\,  \|\bv_h\|_{\mathrm{DG}}  \leq \Big( \|\mathcal{I}_h \bv_h\|^2_{\div,\O} + 
 \|h_{\cF}^{-1/2} \jump{\bv_h}\|^2_{0,\cF^*_h}
 \Big)^{1/2} \leq \bar{C}\,\, \|\bv_h\|_{\mathrm{DG}} \quad\forall\bv_h\in\mV_h.
\end{equation}
Moreover, we have that 
\begin{equation}\label{L2Ph}
 \|\div_h (\bv_h- \mathcal{I}_h \bv_h)\|^2_{0,\O}
 + \sum_{K\in \cT_h} h_K^{-2} \|\bv_h- \mathcal{I}_h \bv_h\|^2_{0,K}
 \leq C\,\, \|h_{\cF}^{-1/2} \jump{\bv_h}\|^2_{0,\cF^*_h},
\end{equation}
 with $C>0$ independent of $h$.
\end{proposition}

Now, with this result at hand, we will prove that $a_h(\cdot,\cdot)$ is coercive in $\mV_h$, independent of the DG method that we are considering. 
\begin{lemma}\label{discElip}
For any $\varepsilon\in\{-1,0,1\}$, there exists a parameter $\textup{\texttt{a}}^*>0$
such that for all $\textup{\texttt{a}}_S\geq\textup{\texttt{a}}^*$  there holds
$$
a_h(\bv_h,\bv_h)\geq\alpha_{\mathrm{DG}}\|\bv_h\|_{\mathrm{DG}}^2\qquad\forall\bv_h\in\mV_h,
$$
where  $\alpha_{\mathrm{DG}}>0$ is independent of $h$.
\end{lemma}
\begin{proof}
The proof follows the same arguments of those in \cite[Lemma 3.2]{MR4077220}, taking into account for our case, the presence of density and sound speed of the fluid.
\end{proof}
\begin{remark}
Notice that the coercivity constant of $a_h(\cdot,\cdot)$ depends strongly on the physical parameters, namely the density and sound speed. Hence, the stability of the method will depend on the configuration of the problem where these parameters play an important role, together with the geometry of domain, since the constants that appear on the discrete trace inequality \eqref{discTrace} and Proposition \ref{propC} depend on $\O$. This is important to have in mind, since as we will analyze in the numerical section (cf. Section \ref{sec:numerics}), the appearance of spurious modes depend on the choice of $\texttt{a}$ and hence, on $\texttt{a}^*$.
\end{remark} 

We notice that the discrete coercivity of $a_h(\cdot,\cdot)$ holds true 
for any $\varepsilon\in\{-1,0,1\}$. Implying that symmetric and nonsymmetric methods
are well posed with our DG approach. Moreover, we introduce the discrete solution operator associated
to \eqref{spectdisc}:
$$
\bT^{\varepsilon}_h:\mV\rightarrow\mV_h,\qquad
\bf\mapsto \bT^{\varepsilon}_h\bf:=\widetilde{\bu}^{\varepsilon}_h,
$$
where $\widetilde{\bu}^{\varepsilon}_h\in\mV_h$ is the unique solution
of the following discrete source problem:
$$
a_h(\widetilde{\bu}^{\varepsilon}_h,\bv_h)=b(\bf,\bv_h)\qquad\forall\bv_h\in\mV_h.
$$
Observe that the superindex $\varepsilon$ denotes if the operator is associated to a symmetric or nonsymmetric method.  This is important to have in mind since
the choice of $\varepsilon$ leads to different solutions that lie on real or complex spaces.  

Clearly  $\bT^{\varepsilon}_h$ is well defined thanks to Lax-Milgram's lemma.
Moreover, there exists a constant $C>0$ independent of $h$ such that
$
\norm{\bT^{\varepsilon}_h\bf}_{\mathrm{DG}}\leq C\norm{\bf}_{\div,\O}
$, for all $\bf\in\mV$. Also, It is easy to check that $(\lambda_h,\bu_h)\in\mathbb{C}\times\mV_h$
is a solution of problem \eqref{spectdisc} if and only if
$(\mu_h,\bu_h)\in\mathbb{C}\times\mV_h$ with $\mu_h=1/\lambda_h$
is an eigenpair of $\bT^{\varepsilon}_h$, i.e.,
$
\bT^{\varepsilon}_h\bu_h=\frac{1}{\mu_h}\bu_h.
$

\begin{proposition}
Let $\bf\in \bP(\mV)$ be such that $\bT\bf:=\widetilde{\bu}$. Then, for any $\varepsilon\in\{-1,0,1\}$ the following best approximation result holds
$$
\|(\bT-\bT^{\varepsilon}_h)\bf\|_{\mathrm{DG}}\leq\frac{M_{DG}}{\alpha_{DG}}\inf_{\bv_h\in\mV_h}\|\bT\bf-\bv_h\|_{\mathrm{DG}}^*.
$$
Moreover, for $s$ as in Lemma \ref{lmm:G}, the following estimate holds
$$
\|(\bT-\bT^{\varepsilon}_h)\bf\|_{\mathrm{DG}}\lesssim h^s(\|\widetilde{\bu}\|_{s,\O}+\|\div\widetilde{\bu}\|_{1+s,\O}),
$$
where the hidden constant is independent of $h$.
\end{proposition}

Finally we derive the following result which holds for symmetric and non symmetric methods.
\begin{lemma}
For any $\bv_h\in\mV_h$, the following estimate holds
$$
\|(\bT-\bT^{\varepsilon}_h)\bv_h\|_{\mathrm{DG}}\lesssim h^s\|\bv_h\|_{\mathrm{DG}},
$$
where $s$ is the exponent provided by  Lemma \ref{lmm:G}
\end{lemma}


\subsection{Convergence analysis and error estimates}
\label{subsec:error}
Since the solution operator $\bT$ is non compact, the analysis of the spectral convergence and error estimates will be performed under the approach of the theory of \cite{MR483400,MR483401}. Let us recall the following definitions, needed to perform the analysis. 
\begin{itemize}
\item P1. $\norm{\bT-\bT_h^{\varepsilon}}_{\mathcal{L}(\mV_h, \mV(h))}\to0$ as $h\to0$.

\item P2. $\forall\bv\in\mV$, there holds
$$\inf_{\bv_h\in \mV_h}\norm{\bv-\bv_h}_{\mathrm{DG}}\,\,\to0\quad\text{as}\quad h\to0.$$
\end{itemize}

In order to prove convergence between eigenspaces,
we introduce the following definitions: let
$\boldsymbol{x}\in\mV(h)$ and $\mathbb E$ and $\mathbb F$
be closed subspaces of $\mV(h)$. We define 
$$\delta(\boldsymbol{x},\mathbb E):=\inf_{\boldsymbol{y}\in\mathbb E}\norm{\boldsymbol{x}-\boldsymbol{y}}_{\mathrm{DG}},\quad\delta(\mathbb E,\mathbb F):=\sup_{\boldsymbol{y}\in\mathbb E:\,\norm{\boldsymbol{y}}_{\mathrm{DG}}=1}\delta(\boldsymbol{y},\mathbb F).$$
Hence, the gap between two closed subspaces is defined by 
$$\gap(\mathbb E,\mathbb F):=\max\set{\delta(\mathbb E,\mathbb F),\delta(\mathbb F,\mathbb E)}.$$

Let $\lambda\in(0,1)$ be an isolated eigenvalue of $\bT$
and let $D$ an open disk in the complex plane with boundary $\gamma$
such that $\lambda$ is the only eigenvalue of $\bT$ lying in $D$ and $\gamma\cap\sp(\bT)=\emptyset$.
We introduce the spectral projector corresponding to the continuous
and discrete solution operators $\bT$ and $\bT_h^{\varepsilon}$, respectively
$$
\begin{aligned}
	\mathcal{E}&:=\frac{1}{2\pi i}
\int_{\gamma}\left(z\boldsymbol{I}-\boldsymbol{T}\right)^{-1}\, dz:\mV(h)\longrightarrow \mV(h),\\
\mathcal{E}_h&:=\frac{1}{2\pi i}
\int_{\gamma}\left(z\boldsymbol{I}-\bT^{\varepsilon}_h\right)^{-1}\, dz:\mV(h)\longrightarrow \mV(h)
\end{aligned}
$$
where $\mathcal{E}_h$ is well-defined and bounded uniformly in 
$h$. 
Moreover, $\mathcal{E}|_{\mV}$ is a
spectral projection in $\mV$ onto the (finite dimensional)
eigenspace $\mathcal{E}(\mV)$ corresponding to the eigenvalue
$\lambda$ of $\bT$. In fact, we have that $\mathcal{E}(\mV(h)) = \mathcal{E}(\mV)$ (see \cite{MR3962898} for further details).
Moreover, $\mathcal{E}_h|_{\mV_h}$ is a projector in $\mV_h$
onto the eigenspace $\mathcal{E}_h(\mV_h)$  corresponding to
the eigenvalues of $\bT^{\varepsilon}_h:\, \mV_h \to \mV_h$ contained in $\gamma$.
We also have that $\mathcal{E}_h(\mV(h)) = \mathcal{E}_h(\mV_h)$.

Now, we will compare $\mathcal{E}_h(\mV_h)$ to $\mathcal{E}(\mV)$ in terms of the gap $\gap$.
The proof of the next auxiliary result follows 
from the definition of $\mathcal{E}$ and $\mathcal{E}_h$.

\begin{lemma}
\label{lot}
The following estimate holds 
\begin{equation*}\label{E-Eh}
\displaystyle \norm{\mathcal{E} - \mathcal{E}_h}_{\mathcal{L}(\mV_h, \mV(h))} \leq C\norm{\boldsymbol{T} - \boldsymbol{T}^{\varepsilon}_h}_{\mathcal{L}(\mV_h, \mV(h))},
\end{equation*}
where the hidden constant is independent of $h$.
\end{lemma}

The following result will be used to establish
the approximation properties of the eigenfunctions of
the continuous problem by means of those of the discrete DG discretization (see \cite[Lemma 4.5]{MR4077220}).

\begin{lemma}
\label{conv}
The following estimate holds 
\[
\gap(\mathcal{E}(\mV), \mathcal{E}_h(\mV_h)) \leq C \Big( \norm{\bT - \boldsymbol{T}^{\varepsilon}_h}_{\mathcal{L}(\mV_h, \mV(h))}+
\delta(\mathcal{E}(\mV), \mV_h)\Big),
\]
where the hidden constant is independent of $h$.
\end{lemma}

Now, we state the convergence properties of the DG methods.

\begin{theorem}
Let $\lambda\in(0,1)$ be an eigenvalue of $\bT$ of algebraic multiplicity $m$  and 
let $D_\lambda$ be a closed disk in the complex plane centered at $\kappa$ with  boundary $\gamma$ such that 
$D_\lambda \cap \sp(\boldsymbol{T}) = \set{\kappa}$. Let $\lambda_{1, h}, \ldots, \lambda_{m(h), h}$ be the eigenvalues of 
$\bT_h^{\varepsilon}$  lying in  $D_\lambda$ and repeated according to their algebraic multiplicity. 
Then, for any DG method defined by $\varepsilon\in\{-1,0,1\}$,
we have that $m(h) = m$ for $h$ sufficiently small and 
\[
\lim_{h\to 0} \max_{1\leq i \leq m} |\lambda -  \lambda_{i, h}| =0.
\]
Moreover, if $\mathcal{E}(\mV)$ is the eigenspace corresponding to $\lambda$ and $\mathcal{E}_h(\mV_h)$ is the
$\bT_h^{\varepsilon}$-invariant subspace of $\mV_h$ spanned by the eigenspaces  corresponding to
$\set{\lambda_{i, h},\hspace{0.1cm} i = 1,\ldots, m}$ then
\[
\lim_{h \to 0} \gap(\mathcal{E}(\mV), \mathcal{E}_h(\mV_h)) = 0.
\]
\end{theorem}
\begin{proof}
See proof of  \cite[Theorem 5.2]{MR3962898}.
\end{proof}


Let us introduce the following distance
\[
\delta^* (\mathcal{E}(\mV) , \mV_h):= \sup_{{\bv\in \mathcal{E}(\mV)}\atop{\norm{\bv}_{\mathrm{DG}}=1}}\inf_{\bv_h\in \mV_h} \norm{\bv - \bv_h}^*_{\mathrm{DG}}.
\]

With this distance at hand, we derive the following results, which  have been established in  \cite[Theorem 6.1]{MR3962898}
for a fixed eigenvalue $\lambda\in(0,1)$ of $\bT$.

\begin{theorem}
\label{hatgap}
For $h$ small enough, the following estimate holds
$$
\gap\big(\mathcal{E}(\mV), \mathcal{E}_h(\mV_h) \big)\lesssim \delta^* (\mathcal{E}(\mV) , \mV_h),
$$
where the hidden constant is independent of $h$.
 \end{theorem}

Finally, with the aid of Proposition~\ref{lmm:G},
we present the rates of convergence of the proposed DG methods. The proof of the following result is identical to the one contained in \cite[Theorem 4.4]{MR4077220} and holds for both, symmetric and nonsymmetric DG methods.
\begin{theorem}
\label{errorE}
Let $r>0$ be such that
$\mathcal{E}(\mV) \subset \{\bv\in\H^r(\O):\,\,\div\bv\in\H^{1+r}(\O)^n\}$
(cf. Proposition~\ref{lmm:G}). Then, for $\varepsilon\in\{-1,0,1\}$ we have
$$
\displaystyle\gap(\mathcal{E}_h(\mV_h), \mathcal{E}(\mV))\lesssim h^{\min\{r, k\}}\quad \text{ and }\quad
\max_{1\leq i\leq m} |\lambda - \lambda_{i, h}|\lesssim \, h^{\sigma\min\{r, k\}},
$$
where $\sigma:=\frac{(\varepsilon+2)^{\varepsilon}-1}{2}+1$ and the hidden constants are independent of $h$.
\end{theorem}

 \section{The  pressure formulation}
\label{sec:pressure}
An alternative approach to analyze the acoustic problem 	\eqref{def:acustica} is to consider the  pressure formulation which is obtained eliminating the displacement in \eqref{def:acustica}, using the second equation, in order to write
\begin{equation}\label{def:acustica_pressure}
\left\{
\begin{aligned}
 c^2\div \left(\frac{1}{\rho(x)}\nabla p\right)+\frac{\omega^2}{\rho(x)} p& = &0&\text{ in }\,\O,\\
\frac{\partial p}{\partial\boldsymbol{n}}&=&0&\text{ on }\,\partial\O,
\end{aligned}
\right.
\end{equation}
where \eqref{def:acustica_pressure} is now a scalar problem. Clearly the displacement is recovered with a postprocess of the pressure, using the first equation on \eqref{def:acustica}.  Let us remark that the computed solutions, namely the eigenvalues and eigenfunctions, are exactly the same when the pure displacement formulation is considered.  
 A variational formulation for this problem is: Find $\omega\in\mathbb{R}$ and $0\neq p\in \H^1(\O)$ such that
$$
\displaystyle\int_{\O}\frac{1}{\rho(x)}\nabla p\cdot\nabla v=\omega^2\int_{\O}\frac{1}{\rho(x)}pv \quad\forall v\in \H^1(\O).
$$

With a shift argument we arrive to the following problem: Find $\omega\in\mathbb{R}$ and $0\neq p\in \H^1(\O)$ such that
\begin{equation}
\label{eq:pression}
a^p(p,v)=(\omega^2+1)b^p(p,v)\quad\forall v\in \H^1(\O),
\end{equation}
where the bilinear forms $a:\H^1(\O)\times \H^1(\O)\rightarrow\mathbb{R}$ and $b:\H^1(\O)\times \H^1(\O)\rightarrow\mathbb{R}$
are defined for all $q,v\in \H^1(\O)$ by 
$$ a^p(q,v):=c^2\int_{\O}\frac{1}{\rho(x)}\nabla q\cdot\nabla v+\int_{\O}\frac{1}{\rho(x)}qv\quad\text{and}\quad b^p(q,v):=\int_{\O}\frac{1}{\rho(x)}qv.
$$
Note that $a^p(\cdot,\cdot)$ is coercive in $\H^1(\O)$. Also, the solution operator $T^p: \L^2(\O)\rightarrow \H^1(\O)\hookrightarrow \L^2(\O)$ is defined by 
$T^pf=\widetilde{p}$ where $\widetilde{p}\in \H^1(\O)$ is the solution of the corresponding  associated source problem. This solution operator results to be compact due the compact inclusion of $\H^1(\O)$ onto $\L^2(\O)$. On the other hand, 
from Lemma \ref{lmm:invariant}, we have that $\bu\in\H^s(\O)$ with $s\in(1/2,1]$. Hence, from the first equation of \eqref{def:acustica} we have that $\nabla p\in \H^s(\O)$ and hence $p\in\H^{1+s}(\O)$, providing the regularity requirement for the eigenfunctions when problem \eqref{eq:pression} is considered.
For this case, and for any $k\geq 1$,  the finite element spaces are 
$\mathcal{V}_h^p := \cP_k(\cT_h)$ and $\mathcal{V}_h^{p,c} := \mathcal{V}_h^p \cap \H^1(\O)$, which now are 
scalar.
Inspired by the analysis of \cite{MR2220929}, let us define $\mathcal{V}^p(h):=\H^1(\O)+\mathcal{V}^p_h$ which we endow with the following seminorm
$
\| v\|_{\mathcal{V}(h)}^2=\|\nabla_h v\|_{0,\O}^2+\|h_{\mathcal{F}}^{-1/2}\jump{ v}\|_{0,\mathcal{F}_h}^2.
$
Then, we define the norm $\|v_h\|_{DG}^2=\|v_h\|_{0,\O}^2+\|v_h\|_{\mathcal{V}(h)}^2$.

The DG discretization for problem \eqref{eq:pression} is as follows: Find $\omega_h\in\mathbb{C}$ and $0\neq p_h\in\mV_h^p$ such that 
\begin{equation}
\label{eq:pression_discreto}
a_h^p(p_h,v_h)=(\omega_h^2+1)b^p(p_h,v_h)\quad\forall v_h\in \mV_h^p,
\end{equation}
where 
\begin{multline*}
\label{eq:Ahp}
a_h^p(p_h, v_h):=c^2\int_{\O}\frac{1}{\rho(x)}\nabla_h p_h\cdot\nabla_h v_h+ \int_{\O}\frac{1}{\rho(x)}p_hv_h+\int_{\cF^*_h}\frac{\texttt{a}_{S,p}}{h_{\cF}}\jump{p_h}\cdot\jump{v_h}\\
- \int_{\cF^*_h}\frac{1}{\rho(x)}\mean{c^2\nabla_h p_h}\cdot\jump{v_h}
-\varepsilon \int_{\cF^*_h}\frac{1}{\rho(x)}\mean{c^2\nabla_h v_h}\cdot\jump{p_h}.
\end{multline*}

Let us prove that $a_h^p(\cdot,\cdot)$ is coercive in $\mV_h^p$.
\begin{lemma}[ellipticity of $a_h^p(\cdot,\cdot)$]
For any $\varepsilon\in\{-1,0,1\}$, there exists a positive parameter $\texttt{a}_p^*$ such that for all $\texttt{a}_{S,p}\geq \texttt{a}_p^*$ there holds
$$a_h^p(v_h,v_h)\geq \alpha_p\| v_h\|_{DG}^2\quad\forall v_h\in\mathcal{V}_h^p,$$
where $\alpha_p>0$ is independent of $h$.
\end{lemma}
\begin{proof}
Let $v_h\in\mathcal{V}_h^p$. Then, from the definition of $a_h^p(\cdot,\cdot)$ we have
\begin{multline*}
a_h^p( v_h, v_h)=c^2\int_{\O}\frac{1}{\rho(x)}\nabla_h v_h\cdot\nabla_h v_h+ \int_{\O}\frac{1}{\rho(x)}v_hv_h+\int_{\cF^*_h}\frac{\texttt{a}_{S,p}}{h_{\cF}}\jump{ v_h}\cdot\jump{v_h}\\
-\int_{\cF^*_h}\frac{1}{\rho(x)}\mean{c^2\nabla_h v_h}\cdot\jump{ v_h}
-\varepsilon\int_{\cF^*_h}\frac{1}{\rho(x)}\mean{c^2\nabla_h v_h}\cdot\jump{ v_h}\\
\geq \frac{c^2}{\overline{\rho}}\|\nabla_h v_h\|_{0,\Omega}^2+\frac{1}{\overline{\rho}}\|v_h\|_{0,\O}^2+\texttt{a}_{S,p}c^2\|h_{\cF}^{-1/2}\jump{ v_h}\|_{0,\cF_h^*}^2+\frac{c^2(1-\varepsilon)}{\overline{\rho}}\int_{\cF^*_h}\mean{\nabla_h v_h}\cdot\jump{ v_h}\\
\geq \frac{c^2}{\overline{\rho}}\|\nabla_h v_h\|_{0,\Omega}^2+\frac{1}{\overline{\rho}}\|v_h\|_{0,\O}^2+\texttt{a}_{S,p}c^2\|h_{\cF}^{-1/2}\jump{v_h}\|_{0,\cF_h^*}^2\\
+\frac{c^2(1-\varepsilon)}{2\overline{\rho}}(-\|h_e^{1/2}\mean{\nabla v_h }\|_{0,\cF_h^*}^2-\|h_{\cF}^{-1/2}\jump{v_h}\|_{0,\cF_h^*}^2)\\
\geq\underbrace{\min\left\{\frac{1}{\overline{\rho}}\left(c^2-C\left(\frac{1+\varepsilon}{2}\right) \right),\frac{1}{\overline{\rho}}\right\}}_{C_1}(\|\nabla_h v_h\|_{0,\O}^2+\|v_h\|_{0,\O}^2)\\
+c^2\underbrace{\left(\texttt{a}_{S,p}-\left(\frac{1+\varepsilon}{2\overline{\rho}}\right) \right)}_{C_2}\|h_{\cF}^{-1/2}\jump{v_h}\|_{0,\mathcal{F}_h^*}^2
\geq \alpha_p\| v_h\|_{DG}^2,
\end{multline*}
where the constant $C$ is the one provided by  \eqref{discTrace}. Observe that, since $\nu>0$, the ellipticity holds for $\nu> C(1+\varepsilon)/2$ and $\texttt{a}_{S,p}>(1+\varepsilon)/2\overline{\rho}$, with $\varepsilon\in\{-1,0,1\}$. Hence,   defining $\alpha_p:=\min\{C_1,C_2\}$ and choosing $\texttt{a}_S$ such that $\texttt{a}_{S,p}>\texttt{a}_p^*:=(1+\varepsilon)/2\overline{\rho}$ we  conclude the proof.
\end{proof}

Now,  the discrete solution operator  $T_h^{p,\varepsilon}: \H^1(\O)\rightarrow\mV_h^p$ is such that $T_h^{p,\varepsilon} f:=\widetilde{p}_h$ where  $\widetilde{p}_h$ is the solution of the source problem $a_h(\widetilde{p}_h,v_h)= b^p(f,v_h)$, for all $v_h\in\mV_h^p$. In this formulation,  $\varepsilon\in\{-1,0,1\}$ plays the same role as  in the displacement formulation, providing symmetric or nonsymmetric DG methods.

Let us remark that, despite to the fact that the solution operator $T^{p,\varepsilon}$ is continuous and compact, the classic theory of \cite{MR1115235} is not enough to our IPDG methods, since this theory requires that the numerical method must be conforming, implying that the noncompact theory of operators must be used for the pressure formulation.  Then, properties P1 and P2 in this context reads as follows
\begin{itemize}
\item P1. $\norm{T^p-T_h^{p,\varepsilon}}_{\mathcal{L}(\mathcal{V}^p_h, \mathcal{V}^p(h))}\to0$ as $h\to0$.
\item P2. $\forall\tau\in\mathcal{V}^p_h$, there holds
$$\inf_{\tau\in \mathcal{V}^p_h}\norm{\tau-\tau_h}_{\mathcal{V}^p(h)}\,\,\to0\quad\text{as}\quad h\to0.$$
\end{itemize}

The following convergence result holds for the continuous and discrete solution operators.
\begin{lemma}
\label{lmm:TTh_pressure}
For all $f\in \mathcal{V}^p$, the following estimate holds
$$
\|(T^p-T^{p,\varepsilon}_h)f\|_{\mathcal{V}^p(h)}\lesssim  h^s\|f\|_{0,\O},
$$
with $s\in(1/2,1]$  and the hidden constant is independent of $h$.
\end{lemma}
\begin{proof}
From the definition of the continuous and discrete solutions operators,  we have $\widehat{p}:=T^p f$ and $\widehat{p}_h:=T^{p,\varepsilon}_h f$. Now, since $\widehat{p}_h$ is the solution of the corresponding source problem, then $\|\widehat{p}-\widehat{p}_h\|_{\mathcal{V}^p(h)}$ is precisely the error of the DG method applied on the source problem. Then, the result follow with $s\in(1/2,1]$.
%
\end{proof}
If we consider   discrete sources, we obtain the analogous of the previous lemma, in the sense that 
For all $f_h\in \mathcal{V}_h$, the following estimate holds
$$
\|(T^p-T^{p,\varepsilon}_h)f_h\|_{\mathcal{V}^p(h)}\lesssim h^s\|f_h\|_{\mathcal{V}^p(h)},
$$
with $s\in(1/2,1]$  and the hidden constant is independent of $h$.
Now we are in position to establish P1.
\begin{lemma}
\label{lmm:P1_p}
There following estimate holds 
$$
\norm{T^p-T^{p,\varepsilon}_h}_{\mathcal{L}(\mathcal{V}_h,\mathcal{V}^p(h))}\lesssim h^s,
$$
where the hidden constant is independent of $h$.
\begin{proof}
Given $f_h\in\mathcal{V}^p_h$, we have
$$
\norm{T-T^{p,\varepsilon}_h}_{\mathcal{L}(\mathcal{V}^p_h, \mathcal{V}^p(h)}:=\sup_{0\neq f_h\in\mathcal{V}^p_h}
\frac{\norm{(T^p-T^{p,\varepsilon}_h)f_h}_{\mathcal{V}^p(h)}}{{\norm{f_h}_{\mathcal{V}^p(h)}}} \lesssim h^s.
$$
This concludes the proof.
\end{proof}
\end{lemma}

If $\mathcal{E}(\mathcal{V}^p)$ represents the invariant space for the eigenfunctions of the pressure formulation, such that if $p\in\mathcal{E}(\mathcal{V}^p)$ then $p\in\H^{1+r}(\O)$ for $r>0$ and $\mathcal{E}(\mathcal{V}^p_h)$ is its discrete counterpart, we are able to  provide an error estimate for the eigenfunctions. To do this task, and taking into account that the pressure formulation is nothing else that the Laplacian operator, the arguments of \cite{MR2220929} can be easily followed for our purposes.
\begin{lemma}
\label{lmm:gap_gorro_ph}
For $h$ small enough it holds
$$
\widehat{\delta}_h(\mathcal{E}(\mathcal{V}^p),\mathcal{E}_h(\mathcal{V}^p_h))\lesssim h^{\min\{r,k\}},
$$
where the hidden constant is independent of $h$.
\end{lemma}

Finally we present error estimates for the eigenvalues 
for the  symmetric and nonsymmetric methods, where optimal and suboptimal order of convergence are attained, respectively.
\begin{theorem}
	There exists a strictly positive constant $h_0$ such that, for $h<h_0$ there holds
	\begin{enumerate}
	\item If the symmetric IPDG method is considered $(\varepsilon=1)$, then there holds
$$
		|\lambda-\lambda_h|\lesssim  h^{2\min\{r,k\}},
$$
	\item If any of the nonsymmetric IPDG methods are  considered $(\varepsilon\in\{-1,0\})$, then there holds
$$
		|\lambda-\lambda_h|\lesssim h^{\min\{r,k\}},
$$
	\end{enumerate}
	where in each estimate the hidden constant is independent of $h$.
\end{theorem}





\section{Numerical experiments}
\label{sec:numerics}
Now we report a series of  numerical experiments in order to corroborate the robustness of the proposed method with $\varepsilon \in \{ - 1, 0, 1\}$. All the results have been obtained with a FEniCS script \cite{MR3618064}. The computed  order of convergence for the eigenvalues were obtained by means of a standard least square fitting. When no analytical solution is available, we compare the computed eigenvalues for each mesh with extrapolated values.
Let us denote by $N$ the refinement level of the mesh and $\lambda$ the eigenvalues.  For the experiments we consider three cases. In the first case, we study the influence of the stabilization parameter $\texttt{a}$ when the spectrum is computed, in order to observe the arising of spurious modes. This is an important analysis since, according to Lemma \ref{discElip}, the method is stable for a certain threshold of $\texttt{a}$. Let us remark that this threshold changes when the configuration of the method changes, namely, physical constants, geometry of the domain, boundary conditions, etc.
 The second situation studies the convergence of the methods, where for $\varepsilon=1$ the optimal order is expected, whereas for $\varepsilon\in\{-1,0\}$ suboptimal order of convergence must be attained according to, for instance,  \cite{MR2324460, MR4077220}.  Finally, we consider domains with singularities and non-polygonal domains to observe the behavior of the method in these geometries.

Throughout this section, the stabilization parameter will be chosen proportional to the square of the polynomial degree $k$ as $\texttt{a}_S=\texttt{a}k^2$, with $\texttt{a}>0$. We begin our report of numerical results considering first the discrete  displacement formulation \eqref{spectdisc}. Then, we will study the pure pressure numerical scheme \eqref{eq:pression_discreto}.

From the results reported in \cite{MR3962898, MR4077220} for the elasticity and Stokes equations, respectively it follows that for $\texttt{a}<1$, the spurious eigenvalues appear more strongly compared with those computed when $\texttt{a}>1$. Therefore, considering that the coercivity constant depends on the density, in the spurious analysis experiments, we will present different combinations of stabilization parameters that depend on the lower and upper bounds of the density.

\subsection{Test 1: Influence of the stabilization}
 In this experiment we test the effect of the stabilization parameter on each discontinuous scheme, namely SIP ($\varepsilon=1$), IIP ($\varepsilon=0$) and NIP ($\varepsilon=-1$). The importance of this study lies in determining which is the safest parameter to avoid pollution on the  spectrum. To detect spurious modes, we compare eigenvalues and eigenfunctions. 

 The domain is the rectangle $\Omega:=(0,a)\times(0,b)$.  We choose $\rho_1(x,y)=\frac{1}{x^2+y^2+1}$ and $\rho_2(x,y)=e^{xy+1}$, which correspond to  density functions such that $|\rho_1(x,y)|\leq 1$ and $|\rho_2(x,y)|>1$, respectively. For the experiment, we choose $a=1$ and $b=1.1$, whereas $c=1$.
 
 \subsubsection{SIP method} In this test, we take  $\varepsilon=1$ in \eqref{ah}. The first goal is to study the influence of the stabilization parameter  $\texttt{a}_S$  on the computation of the spectrum. More precisely, the intention is to analyze the appearance of spurious eigenvalues for certain values and localize a threshold in which the method is safe to compute the eigenvalues with physical meaning.  In the forthcoming results, we report the computed eigenvalues fixing the refinement parameter $N$ and taking different values of $\texttt{a}$, and polynomial degrees. The numbers inside boxes represent spurious eigenvalues.
 \begin{table}[!h]
 	\centering
 	{\setlength{\tabcolsep}{3.8pt}\footnotesize
 		\caption{Test 1. Computed eigenvalues $\lambda_{h,i}$ for polynomial degrees $k=1,2,3$, mesh level $N=8$ and different values of $\texttt{a}$ and $\rho(x)$, for the SIP method ($\varepsilon=1$).}
 		\begin{tabular}{c|cccc|cccc}
 			&\multicolumn{4}{c}{$\rho_1(x,y)$}&\multicolumn{4}{c}{$\rho_2(x,y)$}\\
 			\toprule
 			\diagbox{k}{a} &$2(\overline{\rho}+\underline{\rho})$ & $4\overline{\rho}$    & $4(\overline{\rho}+\underline{\rho})$    & $8\overline{\rho}$      & $2(\overline{\rho}+\underline{\rho})$ & $4\overline{\rho}$       & $4(\overline{\rho}+\underline{\rho})$       & $8\overline{\rho}$            \\\midrule
 			\multirow{10}{0.2cm}{1}
 				& 7.83005 & 7.85010 & 7.86139 & 7.87395 & 7.92006 & 7.93951 & 7.94676 & 7.95428 \\
 			& 9.58484 & 9.28119 & 9.64013 & 9.65956 & 10.29239 & 10.32611 & 10.33702 & 10.35267 \\
 			& 17.38317 & 10.02391 & 17.53053 & 17.60191 & 18.16105 & 18.28682 & 18.326917 & 18.36967 \\
 			& \fbox{26.83626} & 17.65229 & 32.78925 & 33.04550 & 32.91697 & 33.22771 & 33.41198 & 33.58185 \\
 			& 32.20629 & 32.68393 & 39.74527 & 40.03675 & 39.80118 & 40.23012 & 40.43040 & 40.62000 \\
 			& 39.05536 & 39.66805 & 42.43583 & 42.88156 & \fbox{41.68459} & 43.10124 & 43.38814 & 43.63388 \\
 			& 41.58004 & 42.41955 & 47.72058 & 48.27411 & 42.68574 & 48.74642 & 49.15382 & 49.52661 \\
 			& 46.71721 & 47.77799 & 72.84119 & 74.20859 & 48.66419 & 73.83369 & 74.61207 & 75.31965 \\
 			& \fbox{49.80943} & 72.54681 & 75.13436 & 76.53154 & 72.59414 & 75.71586 & 76.64642 & 77.55341 \\
 			& 70.17544 & 73.49931 & 84.38202 & 86.64709 & 74.59481 & 85.64420 & 87.01155 & 88.42254 \\
 			\midrule
 			\multirow{10}{0.2cm}{2}
 			& 7.83258 & 7.83261 & 7.83262 & 7.83263 & 7.91306 & 7.913092 & 7.913102 & 7.91312 \\
 			& 9.59294 & 9.59299 & 9.59300 & 9.59302 & 10.27004 & 10.27012 & 10.27016 & 10.27018 \\
 			& 17.42091 & 17.42146 & 17.42162 & 17.42179 & 18.16579 & 18.16631 & 18.16649 & 18.16666 \\
 			& 32.34693 & 32.34999 & 32.35098 & 32.35200 & 32.75789 & 32.76038 & 32.76132 & 32.76219 \\
 			& 39.15838 & 39.16262 & 39.16390 & 39.16518 & 39.60587 & 39.60917 & 39.61035 & 39.61145 \\
 			& 41.87547 & 41.88456 & 41.88704 & 41.88956 & 42.53659 & 42.54257 & 42.54459 & 42.54646 \\
 			& 47.03714 & 47.04933 & 47.05232 & 47.05533 & 47.98522 & 47.99580 & 47.99878 & 48.00148 \\
 			& 71.48122 & 71.53121 & 71.54302 & 71.55498 & 72.23229 & 72.26248 & 72.27162 & 72.28008 \\
 			& 73.16055 & 73.19808 & 73.20927 & 73.22068 & 73.57347 & 73.59884 & 73.60805 & 73.61669 \\
 			& 82.66712 & 82.75860 & 82.77748 & 82.79638 & 83.59049 & 83.65245 & 83.66898 & 83.68400 \\	
 			\midrule
 			\multirow{10}{0.2cm}{3}
			& 7.83253 & 7.83253 & 7.83253 & 7.83253 & 7.91295 & 7.91295 & 7.91295 & 7.91295 \\
			& 9.59288 & 9.59288 & 9.59288 & 9.59288 & 10.26991 & 10.26991 & 10.26991 & 10.26992 \\
			& 17.42060 & 17.42060 & 17.42060 & 17.42060 & 18.16508 & 18.16508 & 18.16508 & 18.16509 \\
			& 32.34437 & 32.34437 & 32.34437 & 32.34438 & 32.75262 & 32.75263 & 32.75263 & 32.75264 \\
			& 39.15584 & 39.15584 & 9.155855 & 39.15585 & 39.59998 & 39.59998 & 39.59998 & 39.59998 \\
			& 41.87258 & 41.87260 & 41.87261 & 41.87262 & 42.52780 & 42.52781 & 42.52782 & 42.52784 \\
			& 47.03552 & 47.03554 & 47.03555 & 47.03556 & 47.97646 & 47.97648 & 47.97649 & 47.97650 \\
			& 71.47553 & 71.47572 & 71.47580 & 71.47587 & 72.19433 & 72.19446 & 72.19452 & 72.19457 \\
			& 73.13334 & 73.13345 & 73.13350 & 73.13354 & 73.51695 & 73.51703 & 73.51707 & 73.51710 \\
			& 82.66821 & 82.66852 & 82.66863 & 82.66875 & 83.53586 & 83.53609 & 83.53618 & 83.53628 \\
 			\bottomrule
 	\end{tabular}}
 	\label{tabla:square-SIP-k123}
 \end{table}

 \begin{table}[!h]
 	\centering
 	{\setlength{\tabcolsep}{3.8pt}\footnotesize
 		\caption{Test 1. Computed eigenvalues $\lambda_{h,i}$ for polynomial degrees $k=1,2,3$, mesh level $N=8$ and different values of $\texttt{a}$ and $\rho(x)$, for the SIP method ($\varepsilon=1$).}
 		\begin{tabular}{c|cccc|cccc}
 			&\multicolumn{4}{c}{$\rho_1(x,y)$}&\multicolumn{4}{c}{$\rho_2(x,y)$}\\
 			\toprule
 			\diagbox{k}{a} &4 & $4(\overline{\rho}+1)$    & 8    & $8(\overline{\rho}+1)$      & 4 & $4(\overline{\rho}+1)$       & 8       & $8(\overline{\rho}+1)$            \\\midrule
 			\multirow{10}{0.2cm}{1}
 			&7.85010&7.87395 & 7.87395 & 7.88547 & \fbox{4.94719}& 7.94189 & 7.85203 & 7.95588 \\
 			&9.28119&9.65956 & 9.65956 & 9.67768 & 7.73864 & 10.32434 & 10.15531 & 10.35597 \\
 			&\fbox{10.02391}&17.60191 & 17.60191 & 17.66392 & 9.91112& 18.29696 & 17.80452 & 18.37871 \\
 			&17.65229&33.04550 & 33.04550 & 33.27389 & \fbox{15.58674} & 33.31167 & \fbox{29.18053} & 33.61829 \\
 			&32.68393&40.03675 & 40.03675 & 40.29402 & 17.28392 & 40.31934 & 32.13420 & 40.66076 \\
 			&39.66805&42.88156 & 42.88156 & 43.24050 & \fbox{27.92958} & 43.21682 & 37.94840 & 43.68507 \\
 			&42.41955&48.27411 & 48.27411 & 48.73874 &28.51991 & 48.86380 & 40.57725 & 49.60571 \\
 			&47.77799&74.20859 & 74.20859 & 75.17274 & \fbox{33.30603}& \fbox{59.44931} & 44.65770 & 75.46235 \\
 			&72.54681&76.53154 & 76.53154 & 77.73133 & 34.96146& 74.52412 & \fbox{56.35742} & 77.74420 \\
 			&73.49931&86.64709 & 86.64709 & 88.35366 & 36.65087 & 76.40702 & 65.81340 & 88.70701 \\
 			\midrule
 			\multirow{10}{0.2cm}{2}
 			&7.83261&7.83263 & 7.83263 & 7.83264 & 7.91487 & 7.91309 & 7.91191 & 7.91312 \\
 			&9.59299&9.59302 & 9.59302 & 9.59304 & 10.27009 & 10.27014 & \fbox{8.87486} & 10.27019 \\
 			&17.42146&17.42179 & 17.42179 & 17.42194 & 18.15392 & 18.16639 & 10.26904 & 18.16669 \\
 			&32.34999&32.35200 & 32.35200 & 32.35292 & \fbox{26.39878} & 32.76080 & \fbox{13.21076} & 32.76238 \\
 			&39.16262&39.16518 & 39.16518 & 39.16633 & 32.75376 & 39.60970 & \fbox{17.28802} & 39.61168 \\
 			&41.88456&41.88956 & 41.88956 & 41.89181 & 39.61667 & 42.54348 & 18.17024 & 42.54685 \\
 			&47.04933&47.05533 & 47.05533 & 47.05800 & 42.52010 & 47.99715 & 32.75270 & 48.00205 \\
 			&71.53121&71.55498 & 71.55498 & 71.56563 & 47.97542 & 72.26659 & 39.60443 & 72.28185 \\
 			&73.19808&73.22068 & 73.22068 & 73.23088 & \fbox{53.28165} & 73.60296 & 42.53533 & 73.61851 \\
 			&82.75860&82.79638 & 82.79638 & 82.81314 & 72.17575 & 83.65992 & 47.97276 & 83.68714 \\			
 			\midrule
 			\multirow{10}{0.2cm}{3}
 			&7.83253&7.83253 & 7.83253 & 7.83253 & 7.91295 & 7.91295   & 7.91295   & 7.91295 \\
 			&9.59288&9.59288  & 9.59288  & 9.59288  & 10.26991 & 10.26991  & 10.26991  & 10.26991  \\
 			&17.42060&17.42060 & 17.42060 & 17.42060 & 18.16508 & 18.16508  & 18.16508  & 18.16508  \\
 			&32.34437&32.34438 & 32.34438 & 32.34438 & \fbox{31.51368} & 32.75263  & 32.75260  & 32.75263  \\
 			&39.15584&39.15585 & 39.15585 & 39.15585 & 32.75255 & 39.59998  & 39.59995  & 39.59998  \\
 			&41.87260&41.87262 & 41.87262 & 41.87263 & 39.59988 & 42.52782  & 42.52768  & 42.52783  \\
 			&47.03554&47.03556 & 47.03556 & 47.03557 & 42.52740 & 47.97649  & 47.97635  & 47.97651  \\
 			&71.47572&71.47587 & 71.47587 & 71.47594 & 47.97424 & 72.19448  & 72.19350  & 72.19458  \\
 			&73.13345&73.13354 & 73.13354 & 73.13359 & \fbox{50.98591} & 73.51705  & 73.51650  & 73.51711  \\
 			&82.66852&82.66875 & 82.66875 & 82.66885 & 72.19381 & 83.53613  & 83.53511  & 83.53630  \\
 			\bottomrule
 	\end{tabular}}
 	\label{tabla:square-SIP-k123_2}
 \end{table}
\begin{table}[!h]
	{\footnotesize
		\begin{center}
			\caption{Test 1. Convergence analysis for $\texttt{a}=10(\overline{\rho}+1)$, $\rho(x,y)=\frac{1}{x^2+y^2+1}$, and $k=1,2,3$ for the SIP method ($\varepsilon=1$). }
			\begin{tabular}{c |c c c c |c| c}
				\toprule
				k        & $N=10$             &  $N=20$         &   $N=30$         & $N=40$ & Order & $\lambda_{i}$ \\ 
				\midrule\multirow{4}{0.2cm}{1}
				&7.86806 & 7.84149 & 7.83652 & 7.83478 & 1.98 & 7.83248 \\
				&9.64947 & 9.60705 & 9.59919 & 9.59643 & 2.00 & 9.59290 \\
				&17.58538 & 17.46235 & 17.43923 & 17.43109 & 1.98 & 17.42052 \\
				&32.96979 & 32.50146 & 32.41428 & 32.38371 & 1.99 & 32.34399 \\
				\midrule
				\multirow{4}{0.2cm}{2} 
				&7.83258 & 7.83253 & 7.83253 & 7.83253 & 3.99 & 7.83253 \\
				&9.59295 & 9.59289 & 9.59288 & 9.59288 & 3.99 & 9.59288 \\
				&17.42116 & 17.42063 & 17.42060 & 17.42060 & 3.98 & 17.42060 \\
				&32.34797 & 32.34457 & 32.34438 & 32.34435 & 3.97 & 32.34434 \\
				\midrule
				\multirow{4}{0.2cm}{3}   
				&7.83253 & 7.83253 & 7.83253 & 7.83253 & 5.85 & 7.83253 \\
				&9.59288 & 9.59288 & 9.59288 & 9.59288 & 5.88 & 9.59288 \\
				&17.42060 & 17.42060 & 17.42060 & 17.42060 & 5.98 & 17.42060 \\
				&32.34435 & 32.34434 & 32.34434 & 32.34434 & 5.99 & 32.34434 \\				
				\bottomrule             
			\end{tabular}
	\end{center}}
	\label{tabla:square-SIP-convergencia}
\end{table}
\begin{table}[!h]
	{\footnotesize
		\begin{center}
			\caption{Test 1. Convergence analysis for $\texttt{a}=10(\overline{\rho}+1)$, $\rho(x,y)=e^{xy+1}$, and $k=1,2,3$ for the SIP method ($\varepsilon=1$). }
			\begin{tabular}{c |c c c c |c| c}
				\toprule
				k        & $N=10$             &  $N=20$         &   $N=30$         & $N=40$ & Order & $\lambda_{i}$ \\ 
				\midrule\multirow{4}{0.2cm}{1}
					& 7.94234 & 7.92041 & 7.91628 & 7.91483 & 1.97 & 7.91291 \\
				& 10.32857 & 10.28466 & 10.27648 & 10.27361 & 1.99 & 10.26989 \\
				& 18.31292 & 18.20280 & 18.18194 & 18.17459 & 1.96 & 18.16476 \\
				& 33.34798 & 32.90287 & 32.81956 & 32.79031 & 1.98 & 32.75189 \\
				\midrule
				\multirow{4}{0.2cm}{2} 
					& 7.91302 & 7.91296 & 7.912956 & 7.912956 & 3.99 & 7.91295 \\
				& 10.27003 & 10.26992 & 10.26991 & 10.26991 & 3.99 & 10.26991 \\
				& 18.16577 & 18.16512 & 18.16508 & 18.16508 & 3.98 & 18.16507 \\
				& 32.75676 & 32.75285 & 32.75264 & 32.75260 & 3.97 & 32.75258 \\
				\midrule
				\multirow{4}{0.2cm}{3}   
					& 7.91295 & 7.91295 & 7.91295 & 7.91295 & 5.92 & 7.91295 \\
				& 10.26991 & 10.26991 & 10.26991 & 10.26991 & 5.95 & 10.26991 \\
				& 18.16508 & 18.16507 & 18.16507 & 18.16507 & 5.98 & 18.16507 \\
				& 32.75260 & 32.75258 & 32.75258 & 32.75258 & 5.98 & 32.75258 \\
				\bottomrule             
			\end{tabular}
	\end{center}}
	\label{tabla:square-SIP-convergencia2}
\end{table}
In Table \ref{tabla:square-SIP-k123} we note that spurious eigenvalues appear in the lowest order approximation for $\rho_1$ since $2(\overline{\rho}+\underline{\rho})<4$, which is expected from the results in \cite{MR4077220}. However, for $\rho_2$ we have a different scenario since $2(\overline{\rho}+\underline{\rho})\approx 21.7689$, but a spurious eigenvalue appears in the spectrum. Also, for this configuration of the problem, when the polynomial degree is increased, the appearance of spurious eigenvalues is less frequent. This fact allows us to infer that when the polynomial approximation increases, the spurious eigenvalues begin to vanish, at least when the lowest frequencies are considered.

From the above, one may be tempted to infer that the safe parameter for the SIP scheme is $\texttt{a}\overline{\rho}$, for $\texttt{a}\geq 4$. However, in Table \ref{tabla:square-SIP-k123_2} we observe an alternative spurious analysis, where we consider, in particular, stabilization parameters of the references \cite{MR3962898,MR4077220}, together with a particular combination of the form $\texttt{a}(\overline{\rho}+1)$. Note the remarkable difference in the cleanliness of the spectrum between the schemes with $\rho_1$ and $\rho_2$. In particular, the scheme with $\rho_2$ and $4(\overline{\rho}+1)$, whose stabilized is larger compared to $4\overline{\rho}$, presented in the previous table, gives a spurious eigenvalue for $k=1$. Moreover, we note that the safe parameters in the references are not sufficient if $|\rho(x,y)|>1$. Therefore, we infer that the safe parameter choice for the SIP with variable density lies between the combinations $\texttt{a}(\overline{\rho}+\underline{\rho})/2$ and $\texttt{a}(\overline{\rho}+1)$, for $\texttt{a}\geq 8$, at least for the cases under study.  

On the other hand, Tables \ref{tabla:square-SIP-convergencia}--\ref{tabla:square-SIP-convergencia2} reveal  that with the SIP method the optimal order of convergence is attained. More precisely, the order is exactly $\mathcal{O}(h^{2k})$. This order is obtained since the domain in which we state the problem is convex, leading to sufficiently smooth eigenfunctions for our spectral problem.
\subsubsection{NIP and IIP method} Now our aim is to study computationally the presence of spurious eigenvalues for the nonsymmetric methods, more precisely the NIP ($\varepsilon=-1$) and SIP ($\varepsilon=0$). The combination of stabilization parameters are the same as those of SIP. In Table \ref{tabla:square-NIP-k123} we present the computed eigenvalues for different stabilizations and polynomial degrees when the NIP method is considered.
\begin{table}[!h]
	\centering
	{\setlength{\tabcolsep}{3.8pt}\footnotesize
		\caption{Test 1. Computed eigenvalues $\lambda_{h,i}$ for polynomial degrees $k=1,2,3$, mesh level $N=8$ and different values of $\texttt{a}$ and $\rho(x)$, for the NIP ($\varepsilon=-1$) method.}
		\begin{tabular}{c|cccc|cccc}
			&\multicolumn{4}{c}{$\rho_1(x,y)$}&\multicolumn{4}{c}{$\rho_2(x,y)$}\\
			\toprule
			\diagbox{k}{a} &$\overline{\rho}+\underline{\rho}$ & $2\overline{\rho}$    & $2(\overline{\rho}+\underline{\rho})$    & $4\overline{\rho}$      & $\overline{\rho}+\underline{\rho}$ & $2\overline{\rho}$       & $2(\overline{\rho}+\underline{\rho})$       & $4\overline{\rho}$            \\\midrule
			\multirow{10}{0.2cm}{1}
				& 7.74084 & 7.79239 & 7.81612 & 7.84287 & 7.89912 & 7.92612 & 7.93896 & 7.95100 \\
			& 9.42728 & 9.50518 & 9.54304 & 9.58812 & 10.16304 & 10.22912 & 10.26408 & 10.30103 \\
			& 16.84491 & 17.11841 & 17.24596 & 17.39284 & 17.90370 & 18.07698 & 18.16537 & 18.25622 \\
			& 30.37398 & 31.33856 & 31.79499 & 32.32444 & 31.60803 & 32.30451 & 32.66855 & 33.05071 \\
			& 36.67343 & 37.91183 & 38.46944 & 39.10970 & 38.24068 & 39.06090 & 39.49054 & 39.94433 \\
			& 38.21148 & 39.99929 & 40.82734 & 41.74630 & 40.64388 & 41.81255 & 42.38055 & 42.94331 \\
			& 43.61388 & 45.13923 & 45.90425 & 46.84143 & 45.58253 & 46.85679 & 47.56141 & 48.33365 \\
			& 59.21923 & 65.54764 & 68.22675 & 70.98606 & 65.44416 & 69.58263 & 71.41674 & 73.13639 \\
			& 63.44700 & 68.01181 & 70.23721 & 72.89916 & 67.67393 & 71.02756 & 72.85119 & 74.81097 \\
			& 70.50419 & 75.99282 & 78.72432 & 82.06519 & 75.21822 & 79.64399 & 82.07951 & 84.72267 \\
			\midrule
			\multirow{10}{0.2cm}{2}
			& 7.84646 & 7.84285 & 7.84091 & 7.83847 & 7.92746 & 7.92356 & 7.92134 & 7.91889 \\
			& 9.61147 & 9.60681 & 9.60426 & 9.60101 & 10.29693 & 10.29000 & 10.28596 & 10.28140 \\
			& 17.49628 & 17.47764 & 17.46744 & 17.45441 & 18.24717 & 18.22656 & 18.21449 & 18.20084 \\
			& 32.64374 & 32.56841 & 32.52784 & 32.47677 & 33.02521 & 32.95571 & 32.91572 & 32.87114 \\
			& 39.53666 & 39.44163 & 39.39020 & 39.32520 & 39.94652 & 39.85876 & 39.80801 & 39.75122 \\
			& 42.36415 & 42.24515 & 42.18022 & 42.09761 & 42.95950 & 42.85318 & 42.79122 & 42.72144 \\
			& 47.62310 & 47.48150 & 47.40405 & 47.30527 & 48.54451 & 48.40647 & 48.32566 & 48.23419 \\
			& 72.84790 & 72.52868 & 72.35369 & 72.13023 & 73.37488 & 73.09611 & 72.93325 & 72.74945 \\
			& 74.68225 & 74.30735 & 74.10584 & 73.85231 & 74.84721 & 74.52393 & 74.33815 & 74.13114 \\
			& 84.65230 & 84.19038 & 83.93873 & 83.61867 & 85.32188 & 84.90646 & 84.66393 & 84.39006 \\	
			\midrule
			\multirow{10}{0.2cm}{3}
				& 7.83256 & 7.83255 & 7.83255 & 7.83254 & 7.91298 & 7.91297 & 7.91296 & 7.91296 \\
			& 9.59293 & 9.59291 & 9.59291 & 9.59290 & 10.26997 & 10.26995 & 10.26994 & 10.26993 \\
			& 17.42105 & 17.42093 & 17.42087 & 17.42079 & 18.16550 & 18.16538 & 18.16531 & 18.16524 \\
			& 32.34726 & 32.34647 & 32.34607 & 32.34556 & 32.75509 & 32.75440 & 32.75402 & 32.75360 \\
			& 39.15956 & 39.15856 & 39.15803 & 39.15738 & 39.60313 & 39.60225 & 39.60176 & 39.60122 \\
			& 41.87933 & 41.87756 & 41.87662 & 41.87545 & 42.53310 & 42.53165 & 42.53083 & 42.52994 \\
			& 47.04380 & 47.04165 & 47.04051 & 47.03908 & 47.98422 & 47.98215 & 47.98097 & 47.97967 \\
			& 71.50754 & 71.49928 & 71.49490 & 71.48939 & 72.21964 & 72.21292 & 72.20910 & 72.20487 \\
			& 73.16741 & 73.15827 & 73.15349 & 73.14759 & 73.54462 & 73.53705 & 73.53280 & 73.52813 \\
			& 82.72109 & 82.70733 & 82.70006 & 82.69098 & 83.58124 & 83.56921 & 83.56237 & 83.55478 \\
			\bottomrule
	\end{tabular}}
	\label{tabla:square-NIP-k123}
\end{table}
\begin{table}[!h]
	\centering
	{\setlength{\tabcolsep}{3.8pt}\footnotesize
		\caption{Test 1. Computed eigenvalues $\lambda_{h,i}$ for polynomial degrees $k=1,2,3$, mesh level $N=8$ and different values of $\texttt{a}$ and $\rho(x)$, for the NIP ($\varepsilon=-1$) method.}
		\begin{tabular}{c|cccc|cccc}
			&\multicolumn{4}{c}{$\rho_1(x,y)$}&\multicolumn{4}{c}{$\rho_2(x,y)$}\\
			\toprule
			\diagbox{k}{a} &2 & $2(\overline{\rho}+1)$    & 4    & $4(\overline{\rho}+1)$      & 2 & $2(\overline{\rho}+1)$       & 4       & $4(\overline{\rho}+1)$            \\\midrule
			\multirow{10}{0.2cm}{1}
			& 7.79239 & 7.84287 & 7.84287 & 7.86916 & 7.50961 & 7.93180 & 7.751918 & 7.95346 \\
			& 9.50518 & 9.58812 & 9.58812 & 9.63647 & 9.41399 & 10.24419 & 9.85755 & 10.30940 \\
			& 17.11841 & 17.3928 & 17.3928 & 17.54436 & 15.58637 & 18.11537 & 17.02452 & 18.27644 \\
			& 31.33856 & 32.32444 & 32.32444 & 32.87046 & \fbox{20.20981} & 32.46180 & 28.30139 & 33.13710 \\
			& 37.91183 & 39.10970 & 39.10970 & 39.77174 & \fbox{22.88114} & 39.24633 & \fbox{33.68424} & 40.04751 \\
			& 39.99929 & 41.74630 & 41.74630 & 42.63479 & \fbox{24.96308} & 42.06174 & \fbox{34.82433} & 43.06573 \\
			& 45.13923 & 46.84143 & 46.84143 & 47.87919 & \fbox{26.94440} & 47.15775 & \fbox{38.22918} & 48.51337 \\
			& 65.54764 & 70.98606 & 70.98606 & 73.47962 & \fbox{28.69336} & 70.40197 & 40.38647 & 73.50271 \\
			& 68.01180 & 72.89916 & 72.89916 & 75.69995 & 29.71622 & 71.80870 & \fbox{45.80849} & 75.25708 \\
			& 75.99282 & 82.06519 & 82.06519 & 85.68356 & \fbox{30.54102} & 80.68556 & 52.33948 & 85.33071 \\
			\midrule
			\multirow{10}{0.2cm}{2}
			&7.84285 & 7.83847 & 7.83847 & 7.83579 & \fbox{7.11555} & 7.92261 & 7.94064 & 7.94064 \\
			&9.60680 & 9.60101 & 9.60101 & 9.59739 & 7.93392 & 10.28829 & 10.31911 & 10.28033 \\
			&17.47764 & 17.45441 & 17.45441 & 17.43976 & 10.27103 & 18.22145 & 18.31268 & 18.19761 \\
			&32.56841 & 32.47677 & 32.47677 & 32.42035 & 18.17500 & 32.93874 & 33.25550 & 32.86073 \\
			&39.44163 & 39.32520 & 39.32520 & 39.25305 & 32.71864 & 39.83724 & 40.23387 & 39.73792 \\
			&42.24515 & 42.09761 & 42.09761 & 42.00507 & 39.60958 & 42.82695 & 43.30082 & 42.70503 \\
			&47.48150 & 47.30527 & 47.30527 & 47.19430 & 42.28333 & 48.37230 & 48.98499 & 48.21260 \\
			&72.52868 & 72.13023 & 72.13023 & 71.87892 & 47.90767 & 73.02722 & 74.26323 & 72.70616 \\
			&74.30735 & 73.85231 & 73.85231 & 73.57203 & \fbox{50.66857} & 74.44508 & 75.92709 & 74.08280 \\
			&84.19038 & 83.61866 & 83.61866 & 83.26009 & \fbox{71.03522} & 84.80387 & 86.65687 & 84.32550 \\		
			\midrule
			\multirow{10}{0.2cm}{3}
			& 7.83255 & 7.83254 & 7.83254 & 7.83253 & 7.91308 & 7.91297 & 7.91302 & 7.91296 \\
			& 9.59291 & 9.59290 & 9.59290 & 9.59289 & 10.27009 & 10.26995 & 10.27003 & 10.26993 \\
			& 17.42093 & 17.42079 & 17.42079 & 17.42070 & 18.16644 & 18.16535 & 18.16597 & 18.16523 \\
			& 32.34647 & 32.34556 & 32.34556 & 32.34501 & 32.76030 & 32.75424 & 32.75768 & 32.75350 \\
			& 39.15855 & 39.15738 & 39.15738 & 39.15668 & 39.60958 & 39.60204 & 39.60637 & 39.60110 \\
			& 41.87755 & 41.87545 & 41.87545 & 41.87417 & 42.54384 & 42.53130 & 42.53849 & 42.52973 \\
			& 47.04165 & 47.03908 & 47.03908 & 47.03749 & 47.99840 & 47.98165 & 47.99156 & 47.97936 \\
			& 71.49928 & 71.48939 & 71.48939 & 71.48330 & 72.26535 & 72.21129 & 72.24347 & 72.20388 \\
			& 73.15827 & 73.14759 & 73.14759 & 73.14115 & 73.59971 & 73.53524 & 73.57238 & 73.52705 \\
			& 82.70733 & 82.69098 & 82.69098 & 82.68094 & 83.66459 & 83.56630 & 83.62420 & 83.55301 \\
			\bottomrule
	\end{tabular}}
	\label{tabla:square-NIP-k123_2}
\end{table}

We observe that when the NIP method is performing the approximation of the spectrum with the selected stabilization parameters,  there are no spurious eigenvalues arising for small values of $\texttt{a}$. For $\rho_1$, since $\overline{\rho}\geq 1$, we fall in the results presented in \cite{MR4077220}, whereas the scheme with density $\rho_2$ remains clean with $\overline{\rho}>2$. 

In contrast, we observe the results in Table \ref{tabla:square-NIP-k123_2}, where a completely clean spectrum is observed for $\rho_1$, while the spectrum with $\rho_2$ and $\texttt{a}=2,4,$ result in the computation of spurious eigenvalues. When we compare this behavior with the SIP method, we notice that the NIP method introduces more spurious modes compared with the SIP method, for the same stabilization parameter. However, when the polynomial degree is increased, the NIP methods behaves better than the SIP method where for $\texttt{a}\geq 4$

Let us remark that when the NIP method is considered, Table \ref{tabla:square-NIP-k123_2} shows that for $k=1$ the number of spurious eigenvalues is clearly large compared for $k=2$. Moreover, for $k=3$ there is no visible spurious eigenvalues for any stabilization parameter. This has been also observed for $k=4,5,6$. From these results, we infer that a safe parameter in which we can operate with the NIP method is for stabilizations such that $\texttt{a}\geq 4$.

Now in Tables \ref{tabla:square-NIP-convergencia}--\ref{tabla:square-NIP-convergencia_2} we present convergence rates for the NIP method when different polynomial degrees are considered. From these tables, we observe that a suboptimal order of convergence is attained. More precisely, when $k$ is odd, the order of convergence is $\mathcal{O}(h^k)$, whereas when $k$ is even, the order of convergence is $\mathcal{O}(h^{k+1})$, which clearly is not optimal compared with the SIP method.

We remark that for the IIP method ($\varepsilon =0$), the results related to the spurious analysis and order of convergence are similar to the NIP method. For the sake of brevity, we skip the results.
\begin{table}[!h]
	{\footnotesize
		\begin{center}
			\caption{Test 1. Convergence analysis for $\texttt{a}=10(\overline{\rho}+1)$, $\rho(x,y)=\frac{1}{x^2+y^2+1}$, and $k=1,2,3$ for the NIP method ($\varepsilon=-1$). }
			\begin{tabular}{c |c c c c |c| c}
				\toprule
				k        & $N=10$             &  $N=20$         &   $N=30$         & $N=40$ & Order & $\lambda_{i}$ \\ 
				\midrule\multirow{4}{0.2cm}{1}
					& 7.86660 & 7.84112 & 7.83636 & 7.83468 & 1.99 & 7.83253 \\
				& 9.64235 & 9.60531 & 9.59842 & 9.59600 & 1.99 & 9.59286 \\
				& 17.56622 & 17.45746 & 17.43704 & 17.42986 & 1.98 & 17.42050 \\
				& 32.91484 & 32.48800 & 32.40836 & 32.38040 & 1.99 & 32.34428 \\
				\midrule
				\multirow{4}{0.2cm}{2} 
					& 7.83341 & 7.83273 & 7.83262 & 7.83258 & 2.16 & 7.83253 \\
				& 9.59410 & 9.59316 & 9.59300 & 9.59295 & 2.16 & 9.59289 \\
				& 17.42594 & 17.42177 & 17.42110 & 17.42087 & 2.23 & 17.42064 \\
				& 32.36612 & 32.34895 & 32.34630 & 32.34542 & 2.29 & 32.34454 \\
				\midrule
				\multirow{4}{0.2cm}{3}   
					& 7.83253 & 7.83253 & 7.83253 & 7.83253 & 4.10 & 7.83253 \\
				& 9.59288 & 9.59288 & 9.59288 & 9.59288 & 4.08 & 9.59288 \\
				& 17.42061 & 17.42060 & 17.42060 & 17.42060 & 4.14 & 17.42060 \\
				& 32.34446 & 32.34434 & 32.34434 & 32.34434 & 4.15 & 32.34434 \\			
				\bottomrule            
			\end{tabular}
	\end{center}}
	\label{tabla:square-NIP-convergencia}
\end{table}
\begin{table}[!h]
	{\footnotesize
		\begin{center}
			\caption{Test 1. Convergence analysis for $\texttt{a}=10(\overline{\rho}+1)$, $\rho(x,y)=e^{xy+1}$, and $k=1,2,3$ for the NIP method ($\varepsilon=-1$). }
			\begin{tabular}{c |c c c c |c| c}
				\toprule
				k        & $N=10$             &  $N=20$         &   $N=30$         & $N=40$ & Order & $\lambda_{i}$ \\ 
				\midrule\multirow{4}{0.2cm}{1}
					& 7.94590 & 7.92131 & 7.91668 & 7.91505 & 1.98 & 7.91294 \\
				& 10.32299 & 10.28334 & 10.27591 & 10.27329 & 1.98 & 10.26988 \\
				& 18.30356 & 18.20050 & 18.18092 & 18.17402 & 1.96 & 18.16484 \\
				& 33.28860 & 32.88857 & 32.81330 & 32.78681 & 1.97 & 32.75168 \\
				\midrule
				\multirow{4}{0.2cm}{2} 
					& 7.91440 & 7.91328 & 7.91309 & 7.91303 & 2.16 & 7.91296 \\
				& 10.27278 & 10.27057 & 10.27020 & 10.27007 & 2.14 & 10.26993 \\
				& 18.17420 & 18.16711 & 18.16595 & 18.16556 & 2.19 & 18.16513 \\
				& 32.78360 & 32.75927 & 32.75544 & 32.75417 & 2.25 & 32.75282 \\
				\midrule
				\multirow{4}{0.2cm}{3}   
					& 7.91295 & 7.91295 & 7.91295 & 7.91295 & 4.25 & 7.91295 \\
				& 10.26992 & 10.26991 & 10.26991 & 10.26991 & 4.20 & 10.26991 \\
				& 18.16510 & 18.16508 & 18.16507 & 18.16507 & 4.22 & 18.16507 \\
				& 32.75274 & 32.75259 & 32.75258 & 32.75258 & 4.17 & 32.75258 \\		
				\bottomrule            
			\end{tabular}
	\end{center}}
	\label{tabla:square-NIP-convergencia_2}
\end{table}
\subsection{Test 2. A 2D reactor} 
This experiment aims to show the performance of the scheme when there is a non-polygonal domain and the density $\rho(x)$ is non constant. This implies that a priori we have zones within the domain where the fluid density is higher or lower. The considered domain is given by $\Omega:=\widetilde{\Omega}\backslash\bigcup_{i=1}^4\Omega_i$, where
$$
\begin{aligned}
	\widetilde{\Omega}&:=\{(x,y)\in\mathbb{R}^2\,:\, x^2+y^2\leq 1\},\\
	\Omega_i&:=\left\{(x,y)\in\mathbb{R}^2\,:\, \left(x+\frac{1}{3}a\right)^2+\left(y+\frac{1}{3}b\right)^2< \frac{1}{8}\right\},
\end{aligned}
$$
with $a=(-1)^{i}$ and $b=(-1)^{\frac{i(i-1)}{2}+1}$, for $i=1,2,3,4$. This domain resembles a cross-section of a reactor. The domain is meshed such that $h\approx 1/2^N$, with $N=3,4,5,6$. For a clean spectrum, we choose the stabilization parameter $\texttt{a}=20$. It notes that the expected order of convergence of the eigenvalues is $\mathcal{O}(h^2)$. This is due to the variational crime that is committed when a curved domain is approximated by polygons. This fact has been also observed in \cite{MR4077220}. 

In Figure \ref{fig:reactor-errores}  we  report the experimental convergence obtained by least squares fitting, solving the problem using the SIP and NIP scheme. Here, $\texttt{dof}$ denotes the number of degrees of freedom, whereas the error on the $i$-th eigenvalue is denoted by $\texttt{err}(\lambda_i)$, with 
$$
\texttt{err}(\lambda_i):=\vert \lambda_{h_i}-\lambda_{i}\vert.
$$
For the IIP method, the values obtained were similar to those obtained with NIP. Note that in all cases, the order of convergence remains as predicted theoretically, for all selected fluid densities. This is followed by Figure \ref{fig:reactor-modo1}, where we show the eigenfunctions associated to the first eigenvalue for each density function. Note that the fluid particles tend to move faster in the lower density zones, as expected. It is worth noting that, when using the NIP method, the spectrum may contain complex values, so we have only reported the real part of the spectrum. These complex contributions can appear at any place and any refinement. For example, for $\rho(x,y)=1/[\cos(\pi x)\cos(\pi y) + 2]$ and $N=3$, the first two computed eigenvalues are of the form $\lambda_{h}=3.47649\pm0.00009i$.
\begin{figure}[!h]
\centering
\begin{minipage}{0.49\linewidth}\centering
	\includegraphics[scale=0.17]{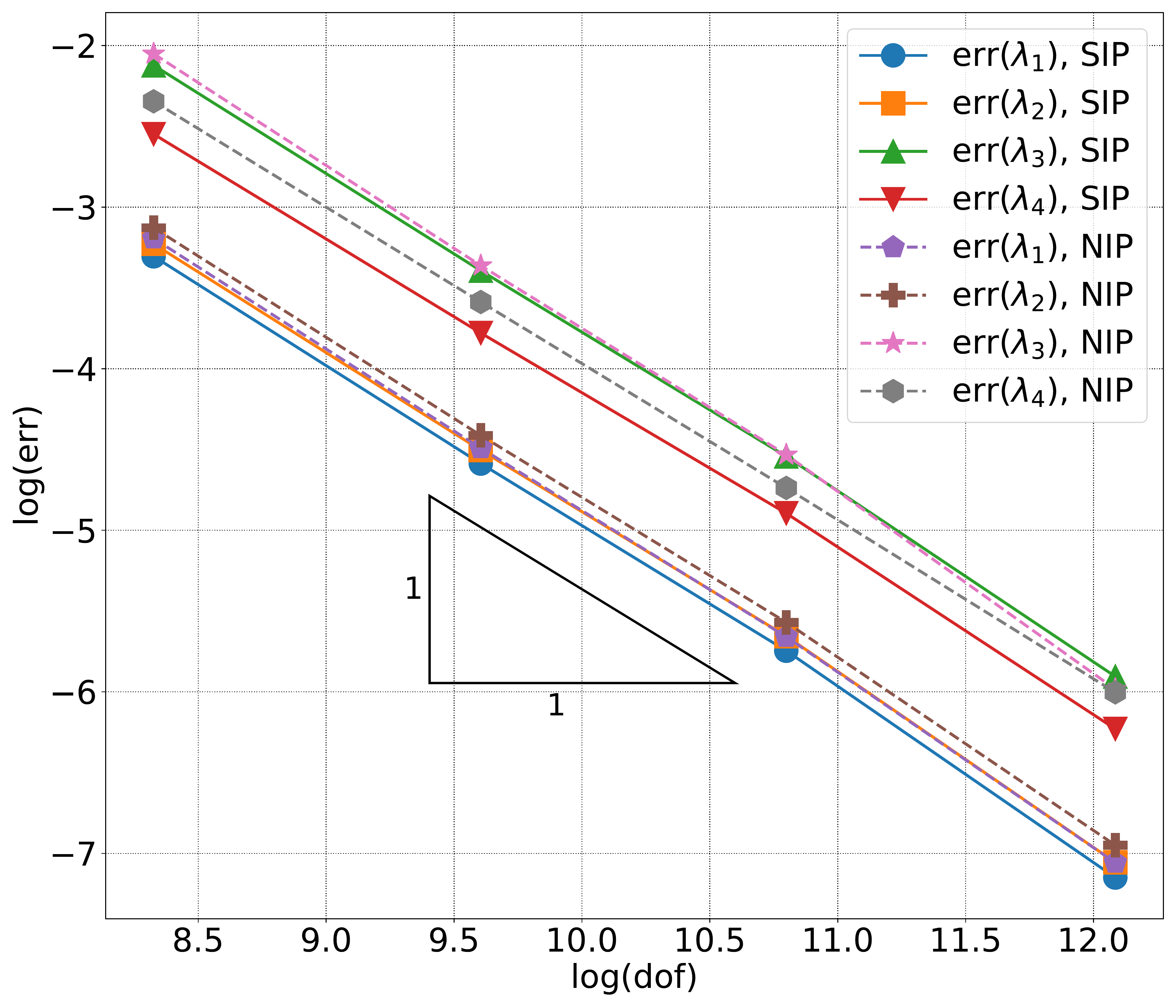}\\
	{\footnotesize $\rho(x,y)=e^{\sin\pi x\cos \pi y}$}
\end{minipage}
\begin{minipage}{0.49\linewidth}\centering
	\includegraphics[scale=0.17]{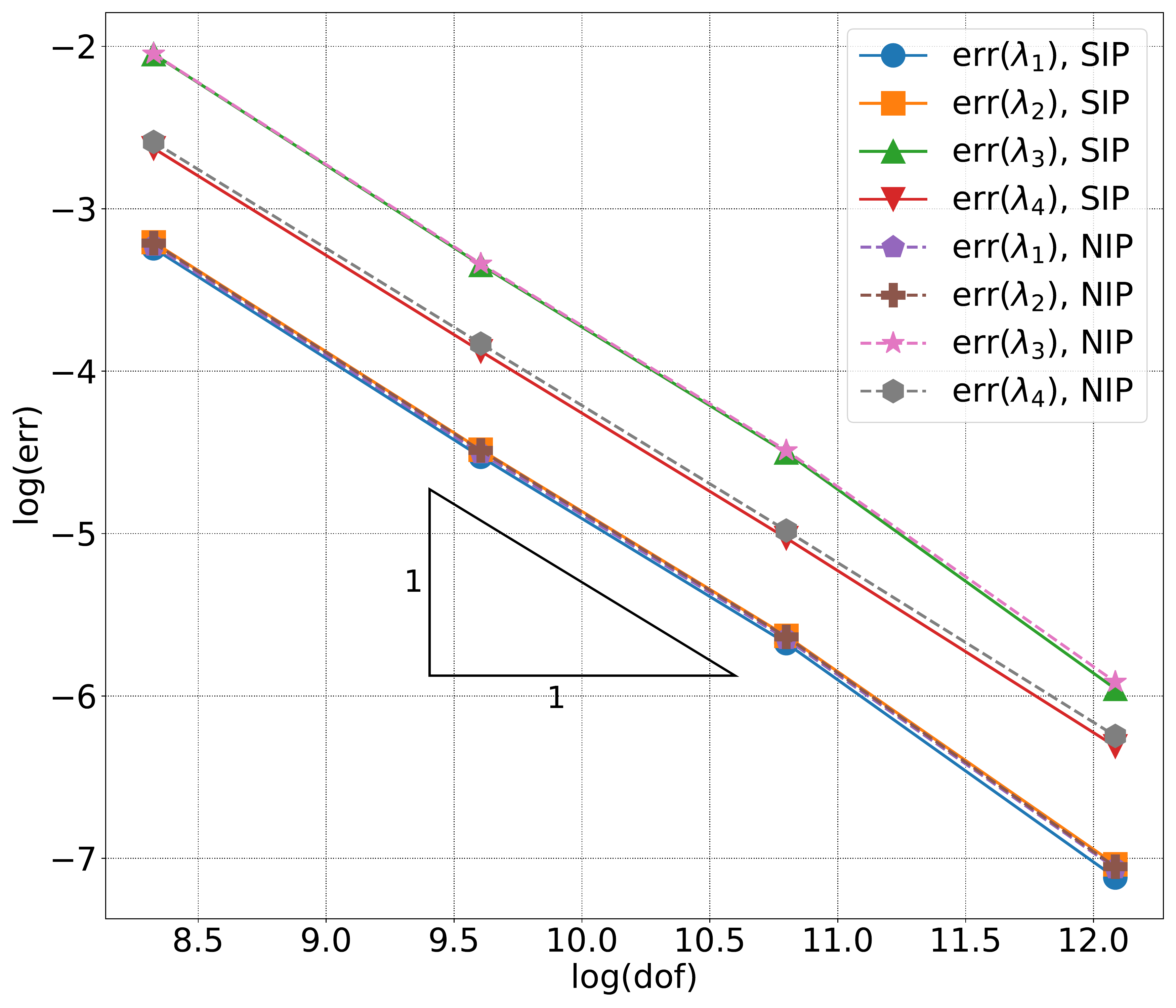}\\
	{\footnotesize $\rho(x,y)=e^{\sin\left(\frac{\pi x}{2}\right)\cos \left(\frac{\pi y}{2}\right)}$}
\end{minipage}\\
\begin{minipage}{0.49\linewidth}\centering
	\includegraphics[scale=0.17]{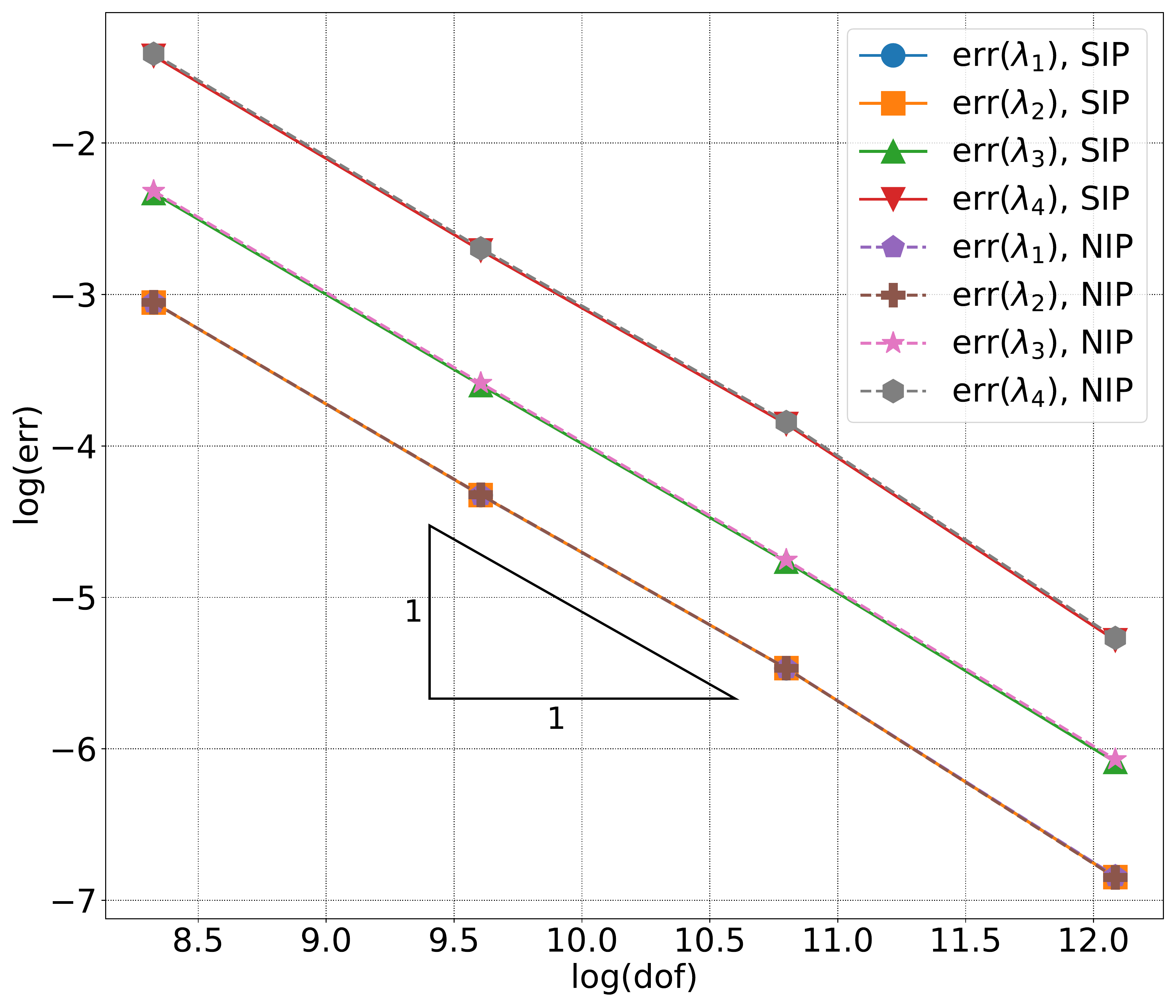}\\
	{\footnotesize $\rho(x,y)=\frac{1}{\cos(\pi x)\cos(\pi y) + 2}$}
\end{minipage}
\begin{minipage}{0.49\linewidth}\centering
	\includegraphics[scale=0.17]{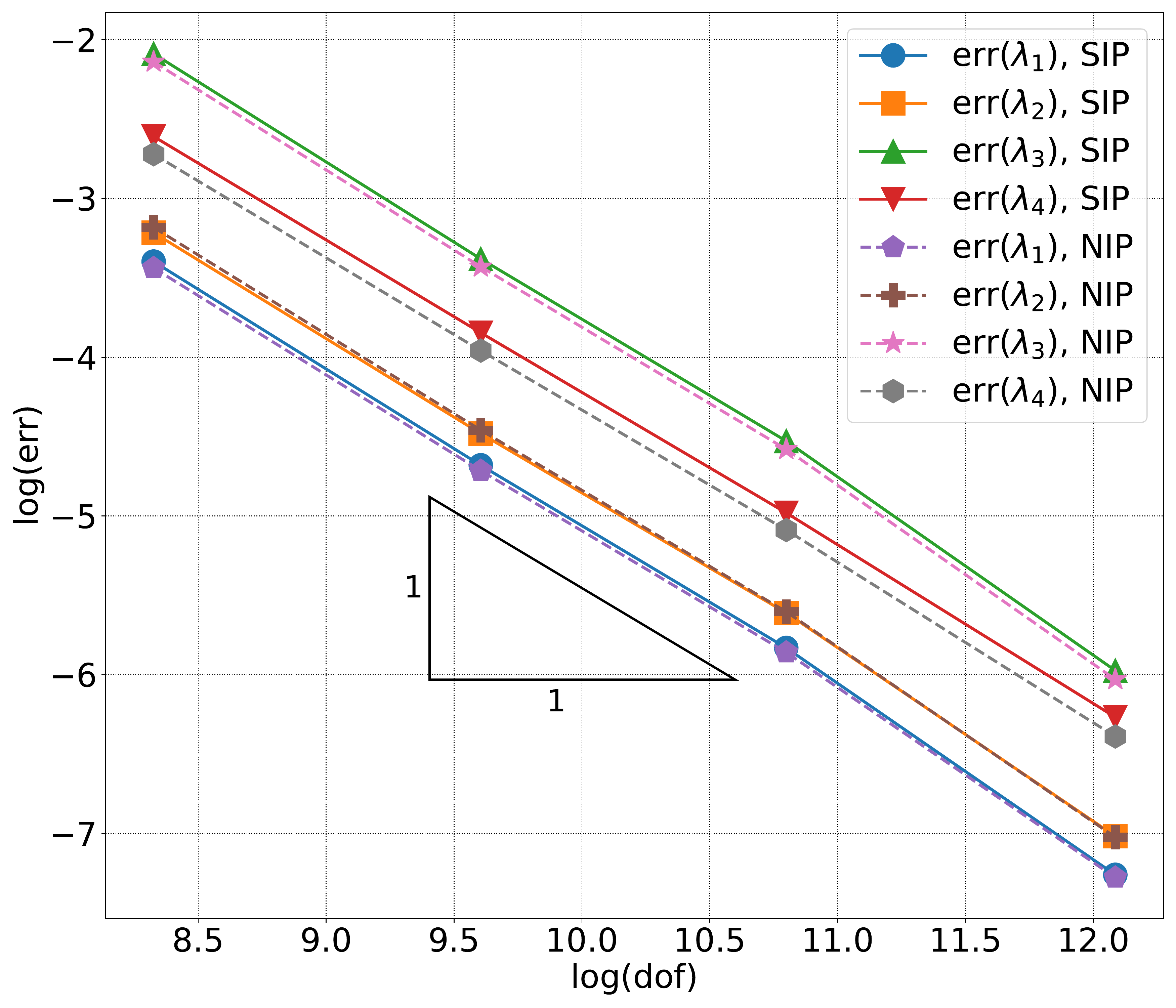}\\
	{\footnotesize $\rho(x,y)=e^{xy + 1}$}
\end{minipage}
\caption{Test 2. Error curves for the first four lowest computed eigenvalues when using different density functions.}
\label{fig:reactor-errores}
\end{figure}
\begin{figure}[!h]
	\centering
	\begin{minipage}{0.32\linewidth}\centering
		\includegraphics[scale=0.05,trim= 40cm 5cm 40cm 5cm, clip]{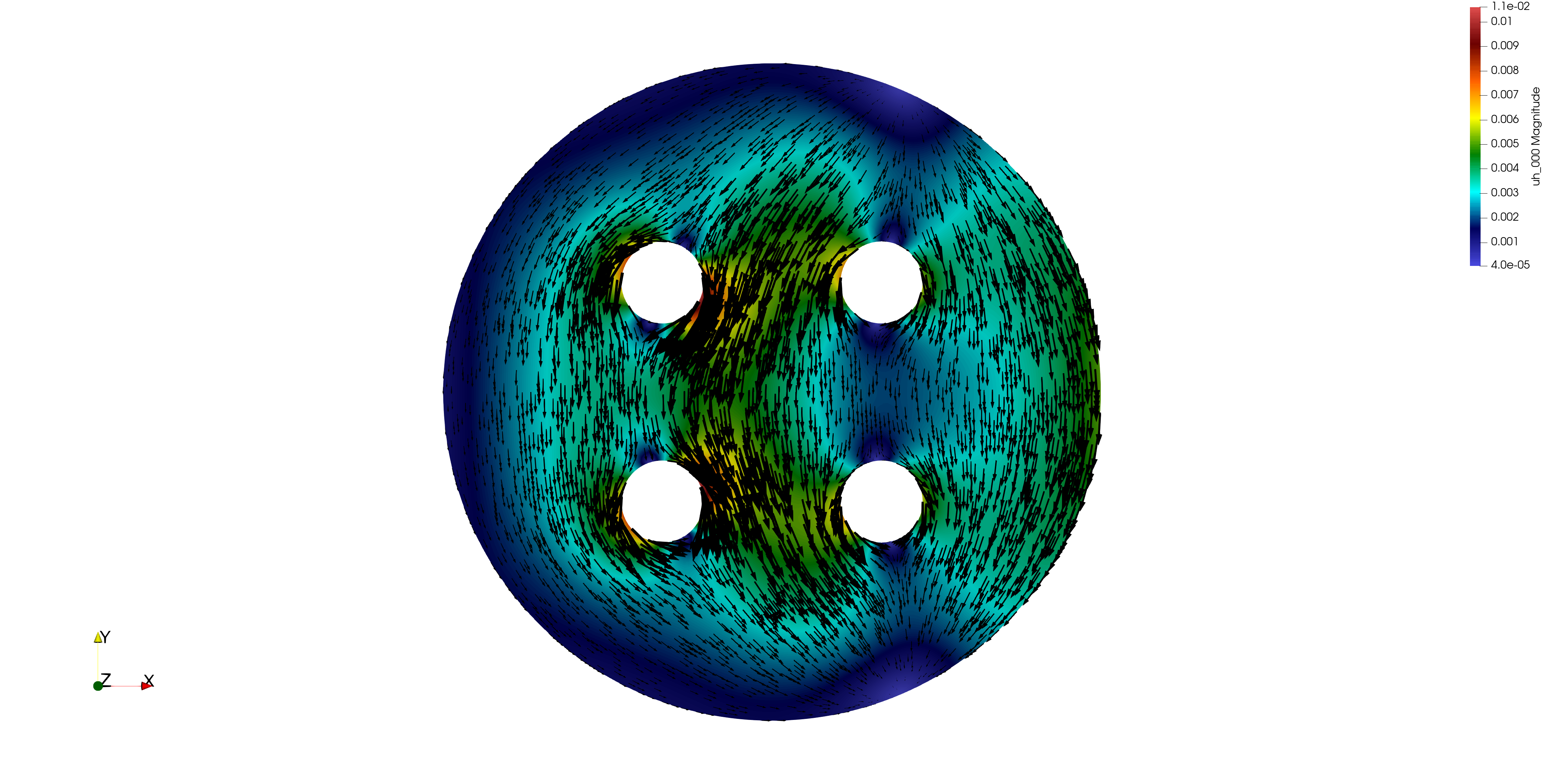}\\
		{\footnotesize $\bu_{1h}$}
	\end{minipage}
	\begin{minipage}{0.32\linewidth}\centering
		\includegraphics[scale=0.05,trim= 40cm 5cm 40cm 5cm, clip]{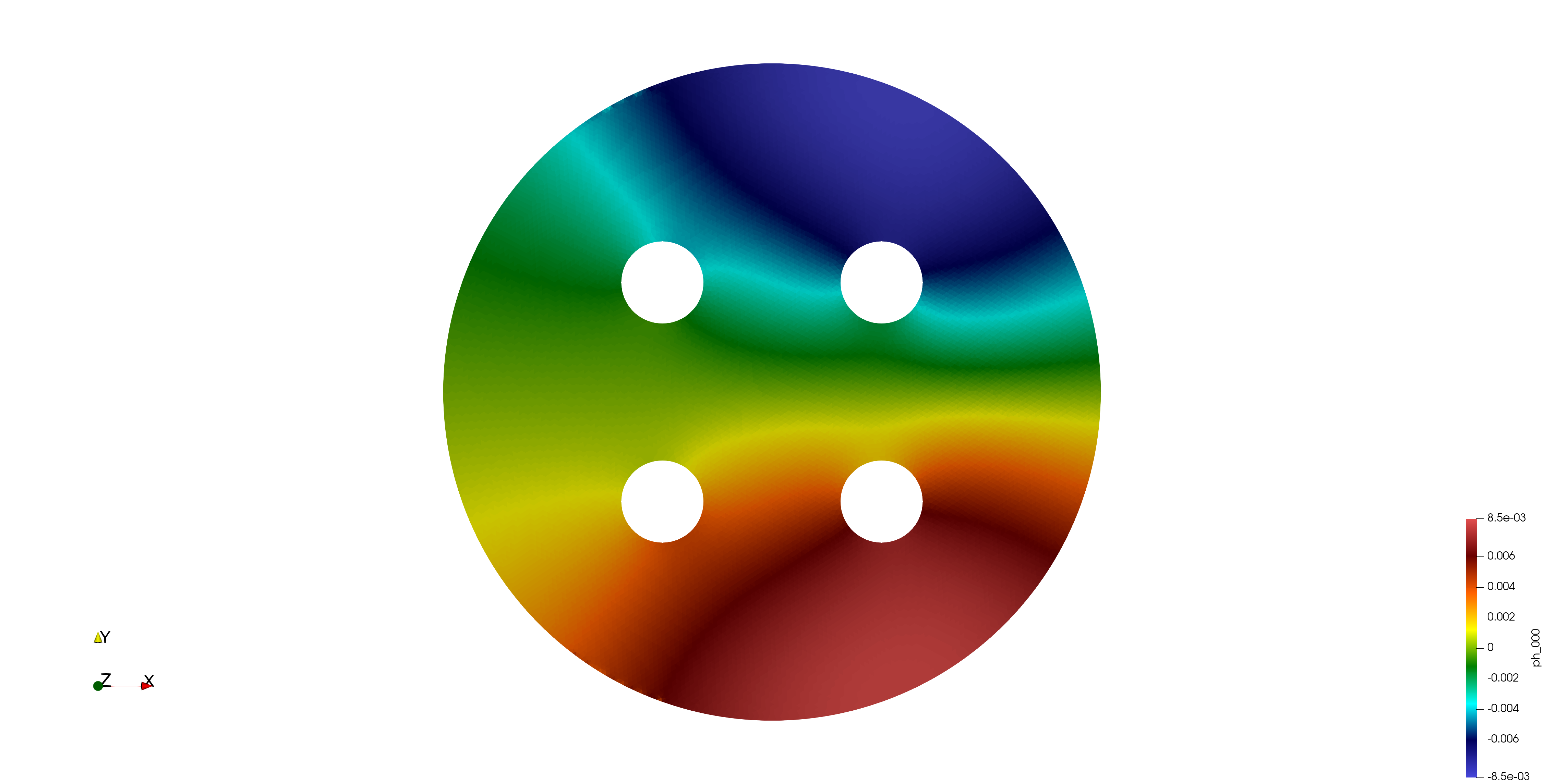}\\
		{\footnotesize $p_{1h}$}
	\end{minipage}
	\begin{minipage}{0.32\linewidth}\centering
		\includegraphics[scale=0.05,trim= 40cm 5cm 40cm 5cm, clip]{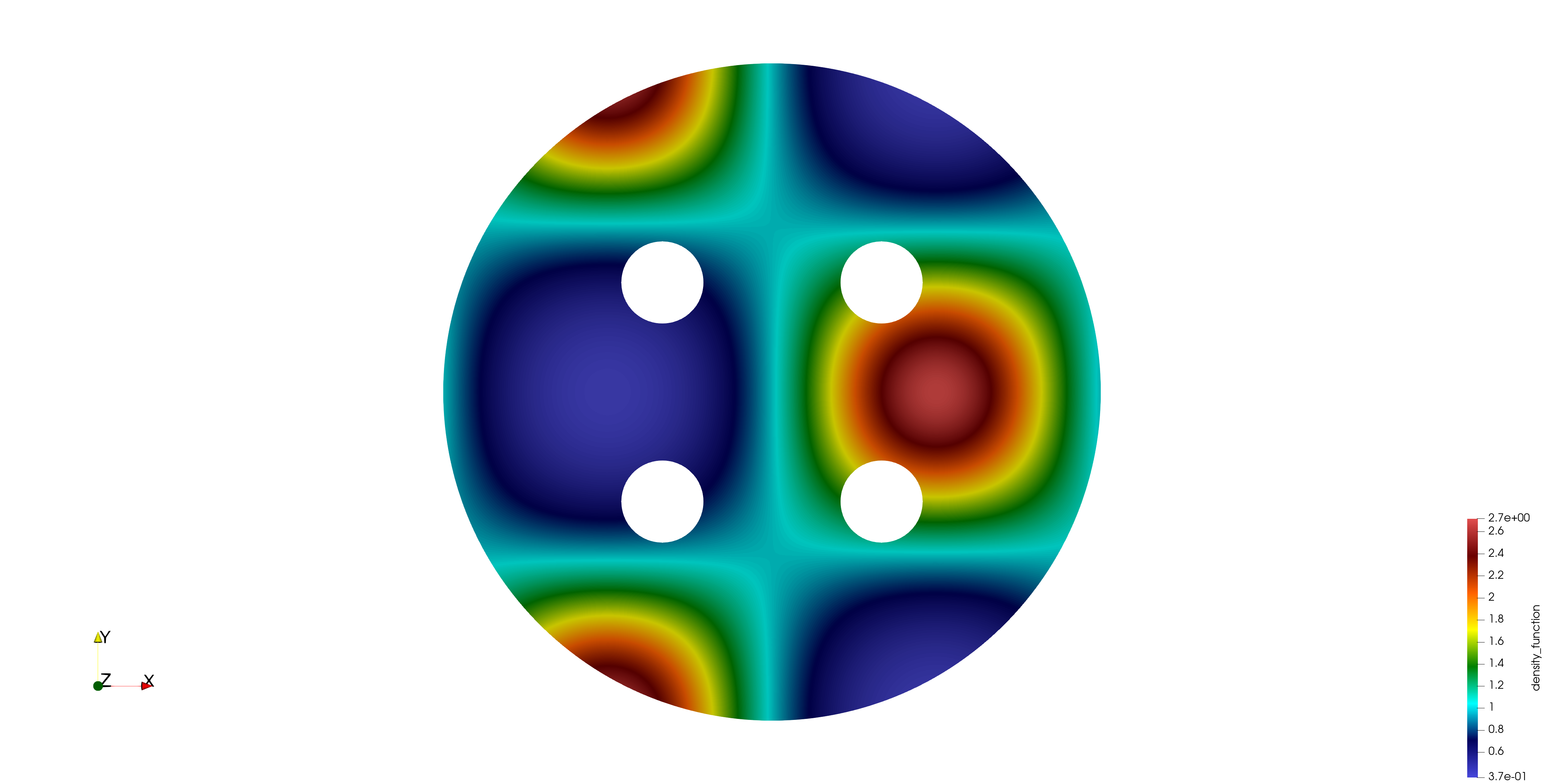}\\
		{\footnotesize $\rho(x,y)=e^{\sin\pi x\cos \pi y}$}
	\end{minipage}
	\begin{minipage}{0.32\linewidth}\centering
		\includegraphics[scale=0.05,trim= 40cm 5cm 40cm 5cm, clip]{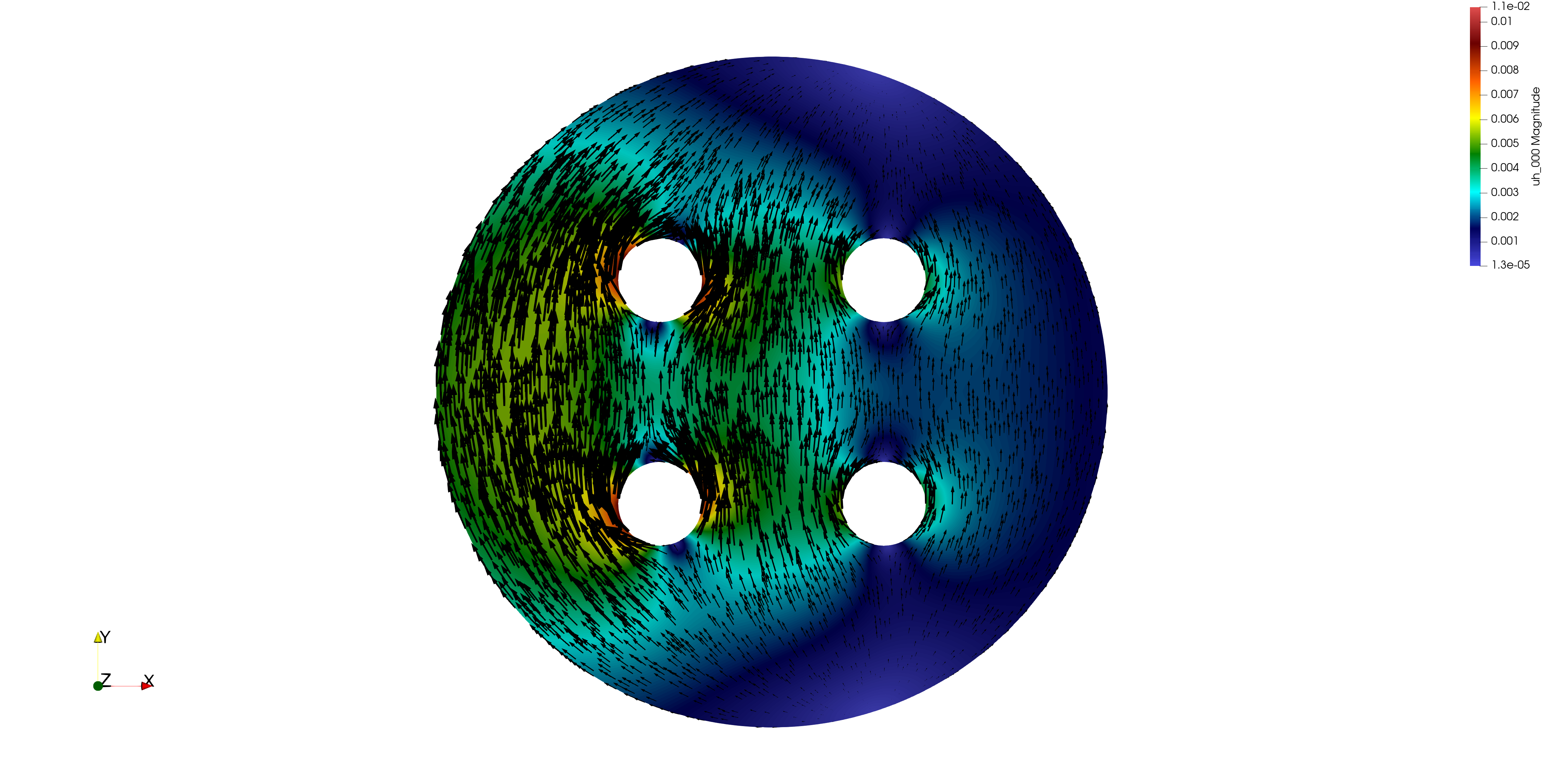}\\
		{\footnotesize $\bu_{1h}$}
	\end{minipage}
	\begin{minipage}{0.32\linewidth}\centering
		\includegraphics[scale=0.05,trim= 40cm 5cm 40cm 5cm, clip]{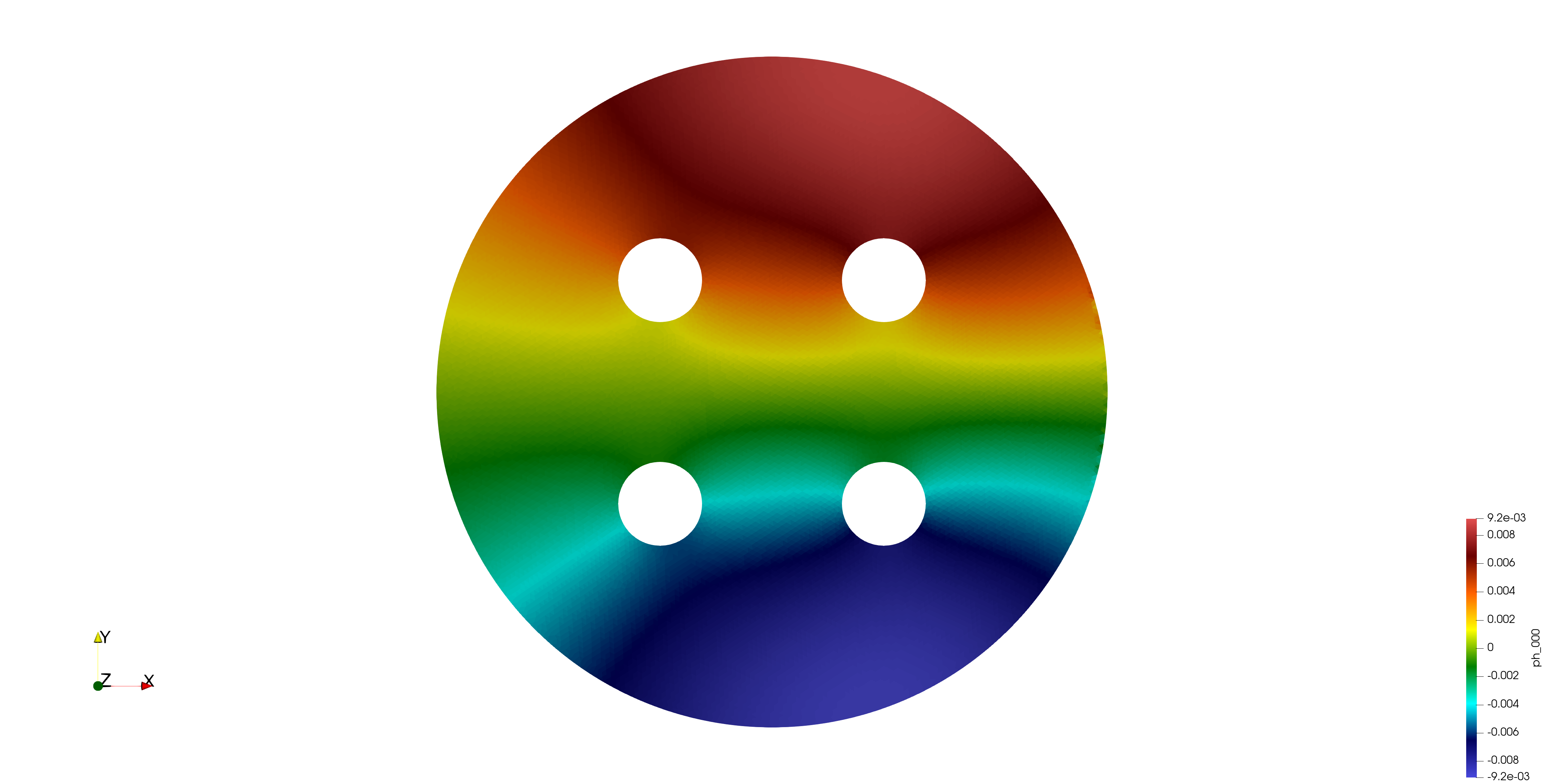}\\
		{\footnotesize $p_{1h}$}
	\end{minipage}
	\begin{minipage}{0.32\linewidth}\centering
		\includegraphics[scale=0.05,trim= 40cm 5cm 40cm 5cm, clip]{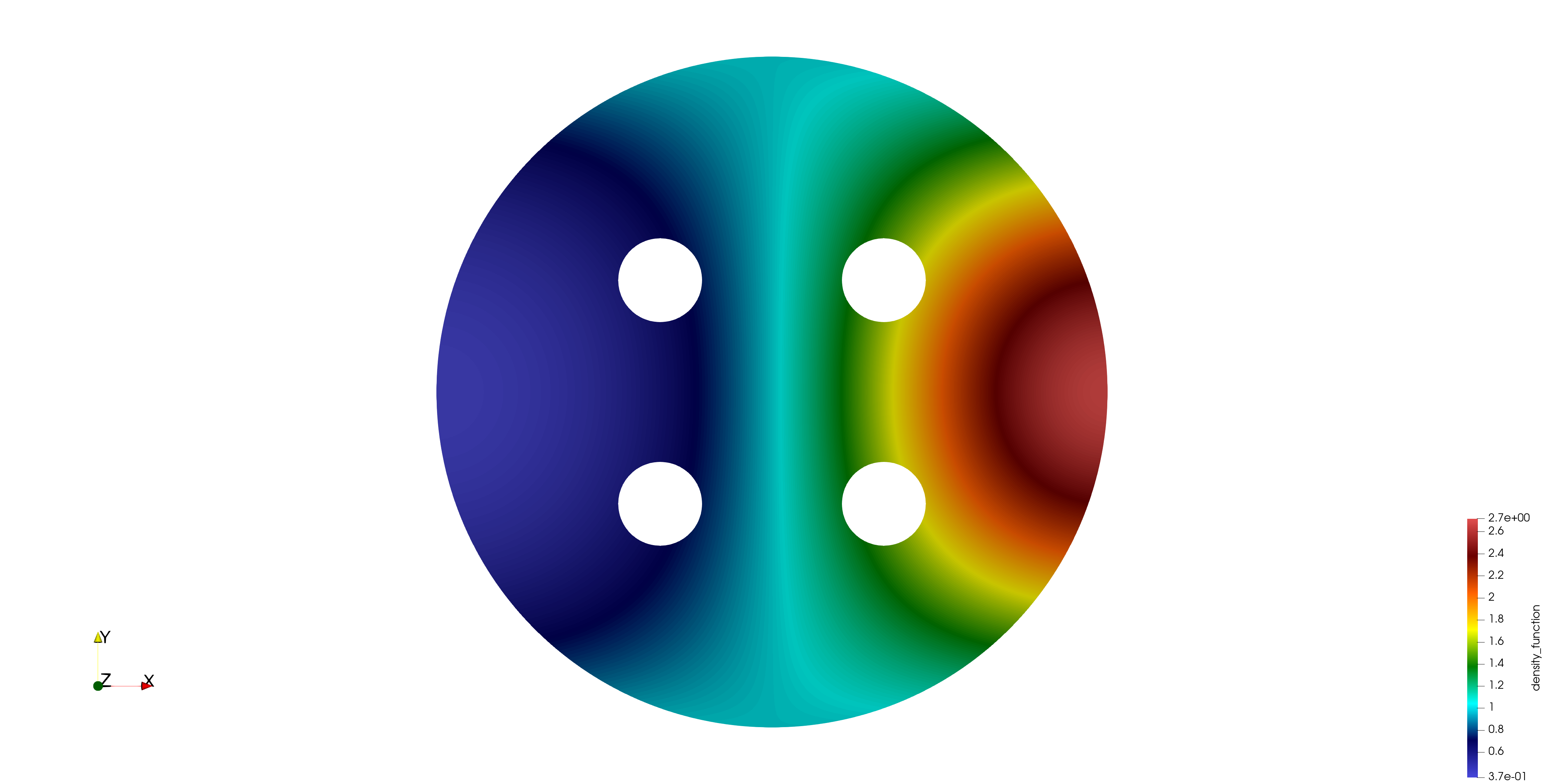}\\
		{\footnotesize $\rho(x,y)=e^{\sin\left(\frac{\pi x}{2}\right)\cos \left(\frac{\pi y}{2}\right)}$}
	\end{minipage}
	\begin{minipage}{0.32\linewidth}\centering
		\includegraphics[scale=0.05,trim= 40cm 5cm 40cm 5cm, clip]{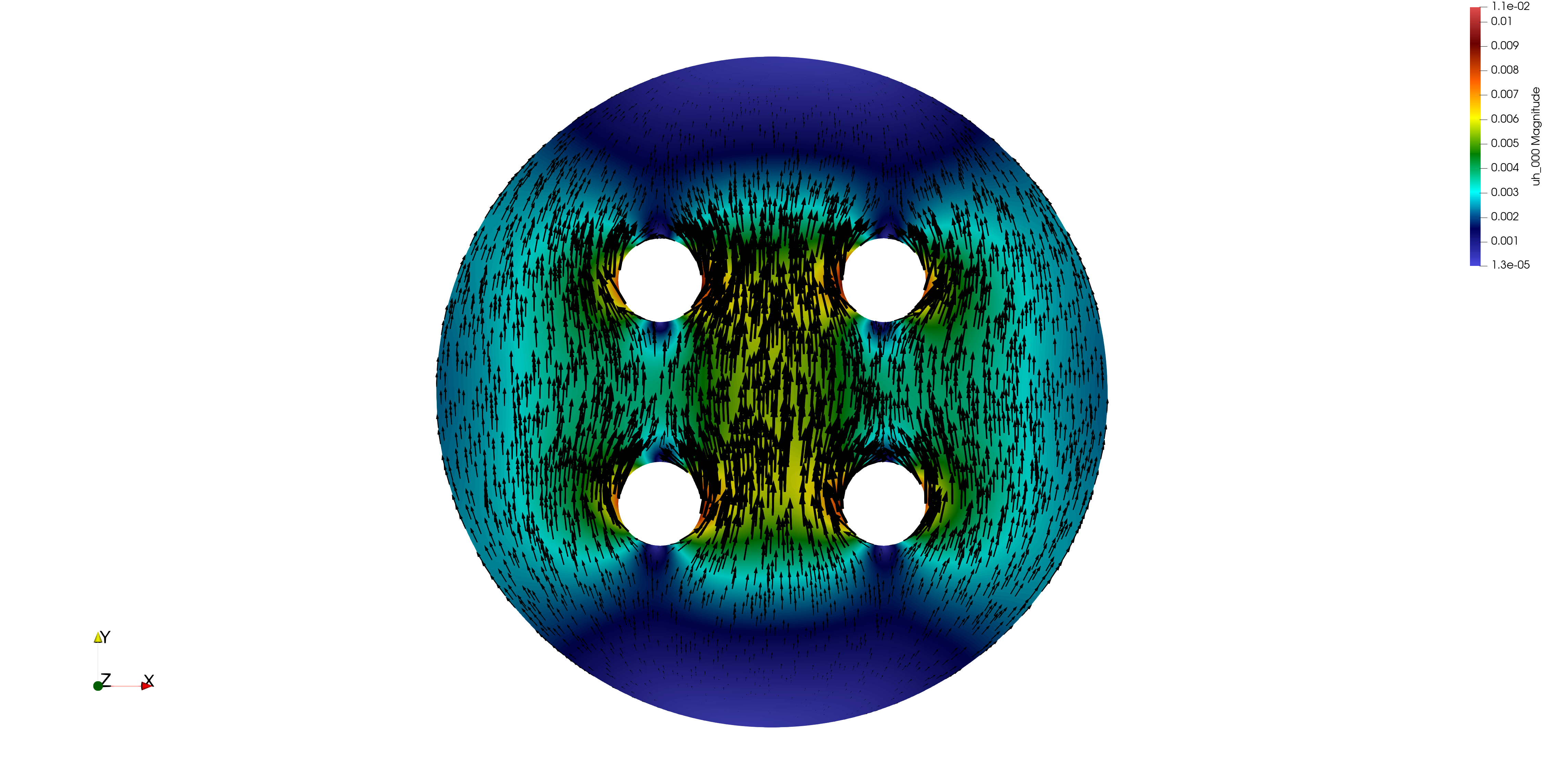}\\
		{\footnotesize $\bu_{1h}$}
	\end{minipage}
	\begin{minipage}{0.32\linewidth}\centering
		\includegraphics[scale=0.05,trim= 40cm 5cm 40cm 5cm, clip]{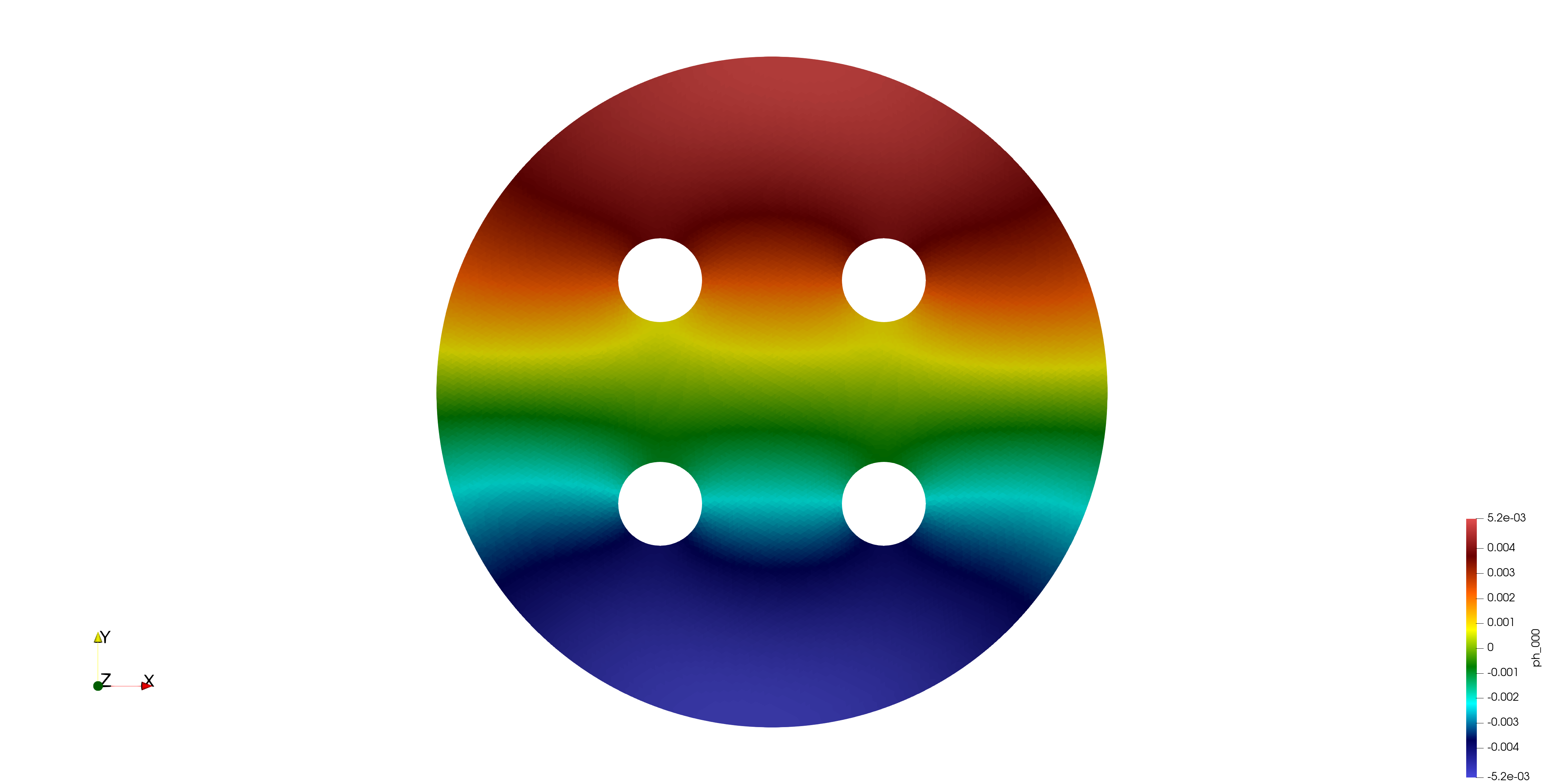}\\
		{\footnotesize $p_{1h}$}
	\end{minipage}
	\begin{minipage}{0.32\linewidth}\centering
		\includegraphics[scale=0.05,trim= 40cm 5cm 40cm 5cm, clip]{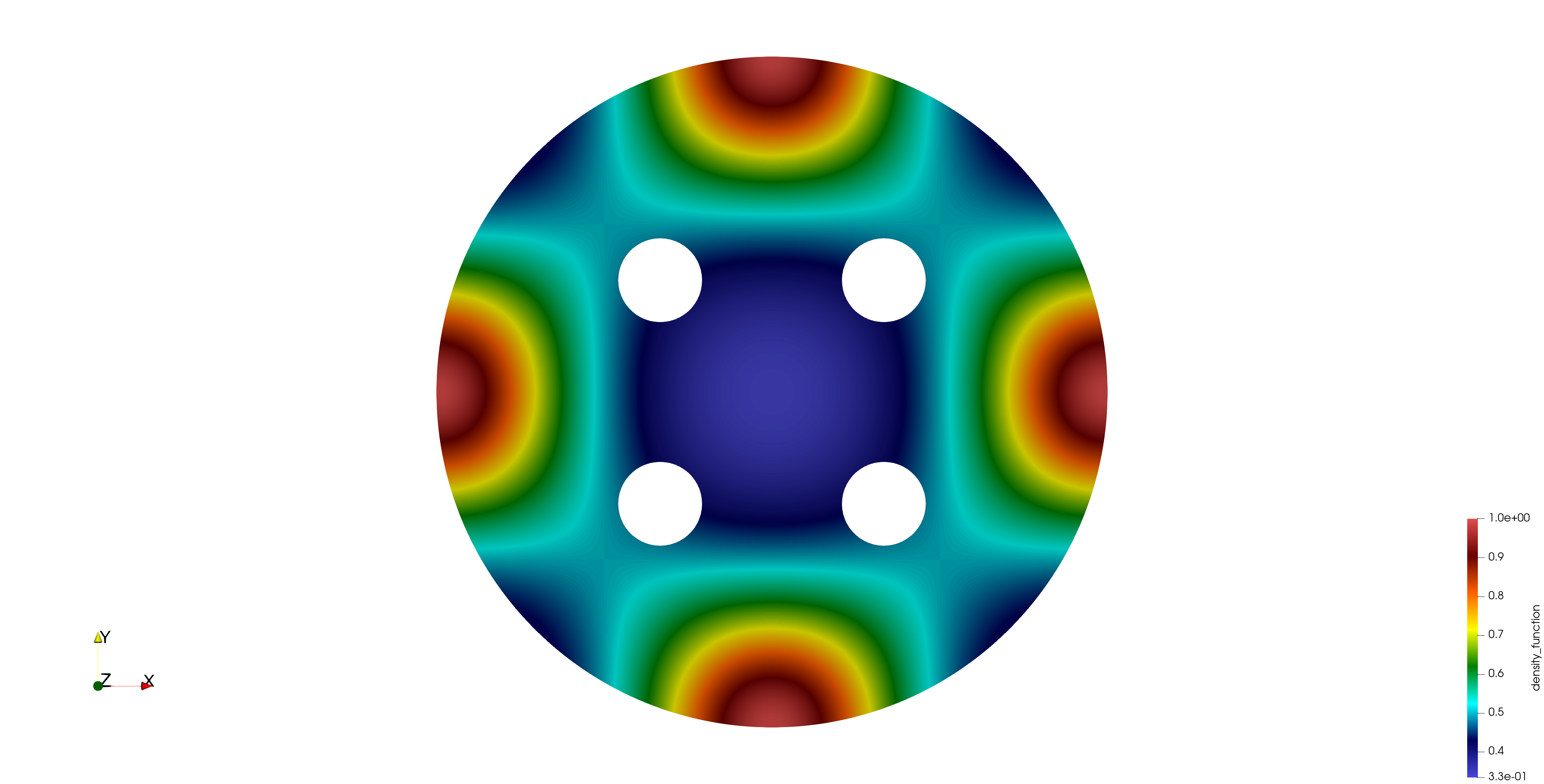}\\
		{\footnotesize $\rho(x,y)=\frac{1}{\cos(\pi x)\cos(\pi y) + 2}$}
	\end{minipage}
	\begin{minipage}{0.32\linewidth}\centering
		\includegraphics[scale=0.05,trim= 40cm 5cm 40cm 5cm, clip]{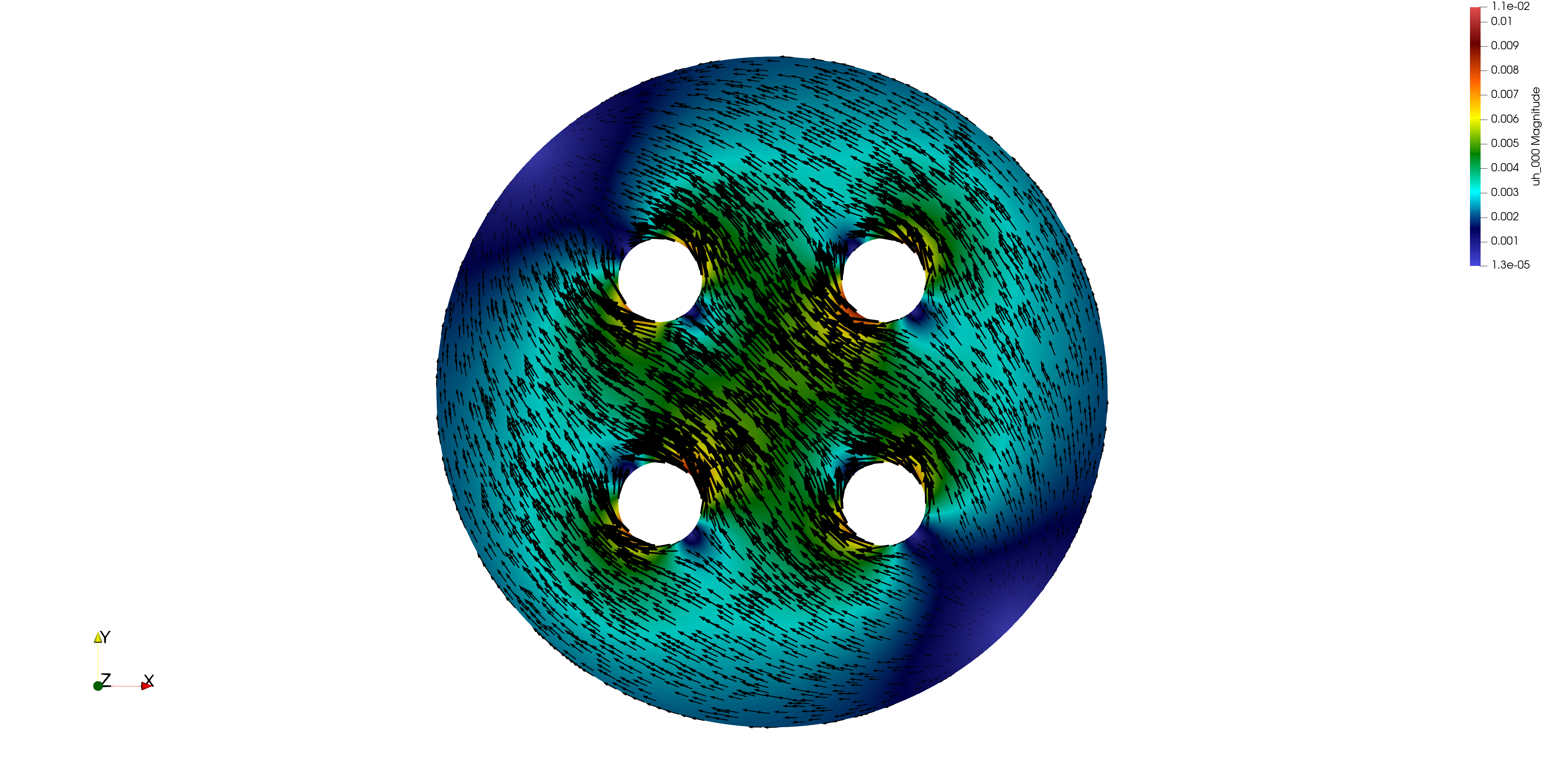}\\
		{\footnotesize $\bu_{1h}$}
	\end{minipage}
	\begin{minipage}{0.32\linewidth}\centering
		\includegraphics[scale=0.05,trim= 40cm 5cm 40cm 5cm, clip]{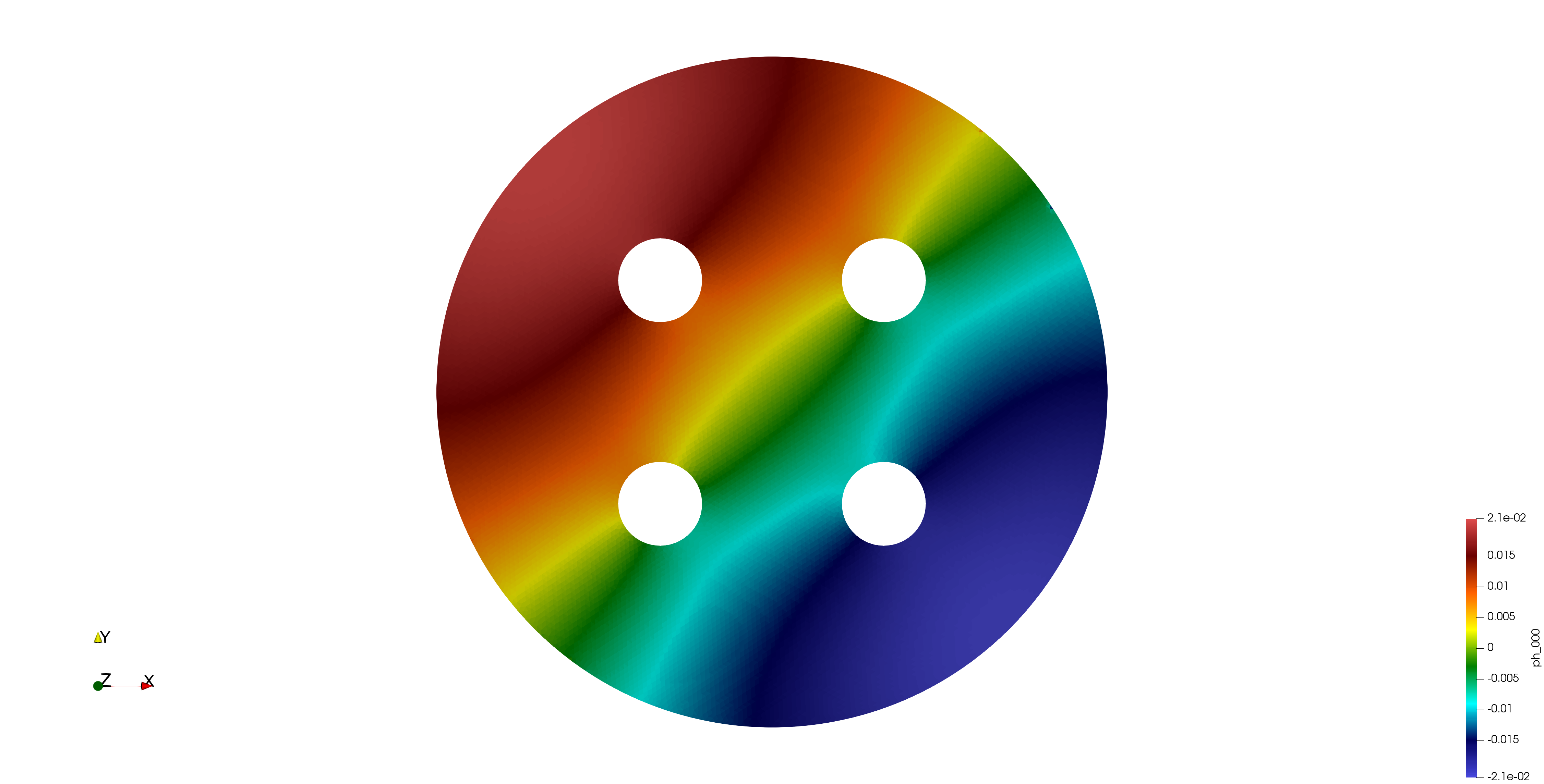}\\
		{\footnotesize $p_{1h}$}
	\end{minipage}
	\begin{minipage}{0.32\linewidth}\centering
		\includegraphics[scale=0.05,trim= 40cm 5cm 40cm 5cm, clip]{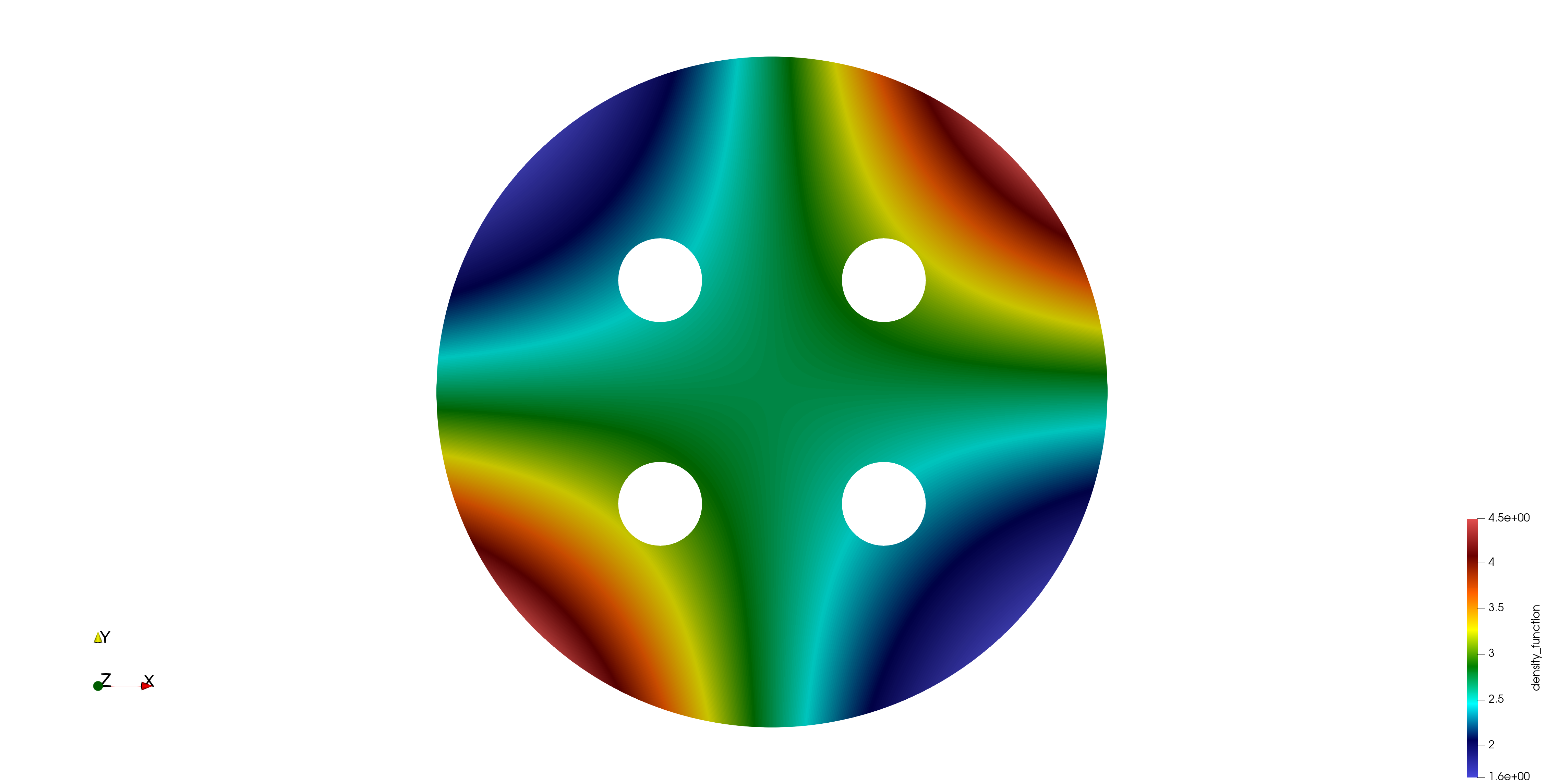}\\
		{\footnotesize $\rho(x,y)=e^{xy + 1}$}
	\end{minipage}
	\caption{Test 2. Lowest computed eigenmodes and postprocessed pressure using different density functions.}
	\label{fig:reactor-modo1}
\end{figure}

\subsection{Test 3: A 3D torus} This test aims to assess the DG method on a three dimensional non-polygonal domain. We consider a torus azimuthally symmetric about the $z-$axis, whose domain is characterized by $$\Omega:=\left\{(x,y,z)\in\mathbb{R}^3\;:\;\left(\sqrt{x^2+y^2}-R\right)^2+z^2=r^2\right\},$$ where $R=1/2$ and $r=1/4$. For simplicity, the physical parameters are taken equal to $1$.
In Table \ref{tabla:toro} we  observe that for the lowest order, the variants of our discontinuous scheme converge with quadratic order, as the theory predicts. This is due to the fact that the eigenfunctions in this domain are regular. However, a variational crime is also being committed when approximating this domain with polyhedrons, so this order will not to improve for $k>1$. Finally, in Figures \ref{fig:toro-modo3} and \ref{fig:toro-modo5} we depict the velocity field along with the post-processed pressure for the third and fifth eigenmodes.
\begin{table}[!h]
	{\footnotesize
		\begin{center}
			\caption{Test 3. Lowest computed eigenvalues for $k$$\,=\,$$1$, and $\texttt{a}=10$ for the SIP, NIP and IIP method. }
			\begin{tabular}{c |c c c c |c| c}
				\toprule
				method        & $N=15$             &  $N=20$         &   $N=25$         & $N=30$ & Order & $\lambda_{extr}$ \\ 
				\midrule
								& 4.20546& 4.20096  & 4.19911   & 4.19778 		& 1.97& 4.19504     \\
				& 4.20588& 4.20127  &   4.19923   & 4.19794 	& 1.94 &4.19497    \\
				\multirow{2}{0.6cm}{SIP}
				& 15.90982& 15.87822  &   15.86392   & 15.85563	&  2.00& 15.83625   \\
				& 15.91094& 15.87873   &   15.86493  & 15.85597	& 2.00& 15.83658   \\
				& 33.31118& 33.20592   &   33.15811  & 33.12984	&  1.96&33.06265   \\
				
				\hline
				
				&4.20528&4.20086&4.19904 &4.19774&  1.98&4.19506     \\
				&4.20571&4.20117&4.19917 &4.19789&  1.94&4.19497    \\
				\multirow{2}{0.6cm}{NIP} 
				&15.90732&15.87681&15.86295 &15.85497&  1.99&15.83604    \\
				&15.90853&15.87728&15.86398 &15.85527& 2.01& 15.83669     \\
				&33.30065&33.19987&33.15391 &33.12693&  1.96&33.06252    \\
				
				\hline
				
				& 4.20535 &4.20090  &4.19907	&4.19775  &  1.97 &4.19503   \\
				& 4.20577 &4.20121  &4.19920	&4.19791  &  1.94 &4.19497    \\
				\multirow{2}{0.6cm}{IIP}   
				& 15.90892 &15.87736  &15.86333	&15.85523 &   2.00 &15.83624     \\
				& 15.90946 &15.87784  &15.86435	&15.85554 &   2.00 &15.83653   \\
				& 33.30471 &33.20221&33.15554	&33.12806 & 1.96 &33.06259   \\

				\bottomrule             
			\end{tabular}
	\end{center}}
	\label{tabla:toro}
\end{table}
\begin{figure}[!h]
	\centering
	\begin{minipage}{0.48\linewidth}\centering
		\includegraphics[scale=0.06,trim= 30cm 2.5cm 30cm 10cm, clip]{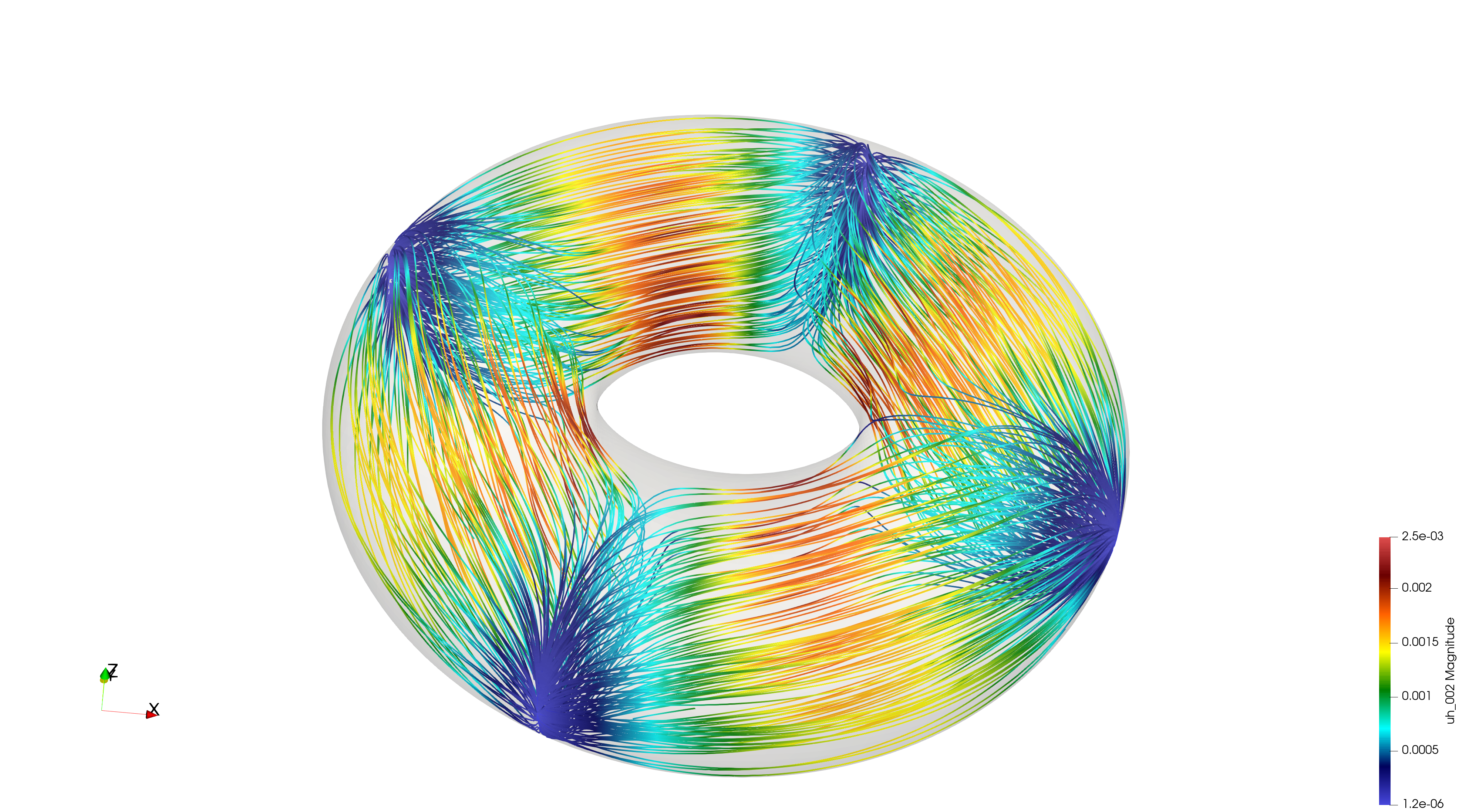}
	\end{minipage}
	\begin{minipage}{0.48\linewidth}\centering
		\includegraphics[scale=0.06,trim= 30cm 2.5cm 30cm 10cm, clip]{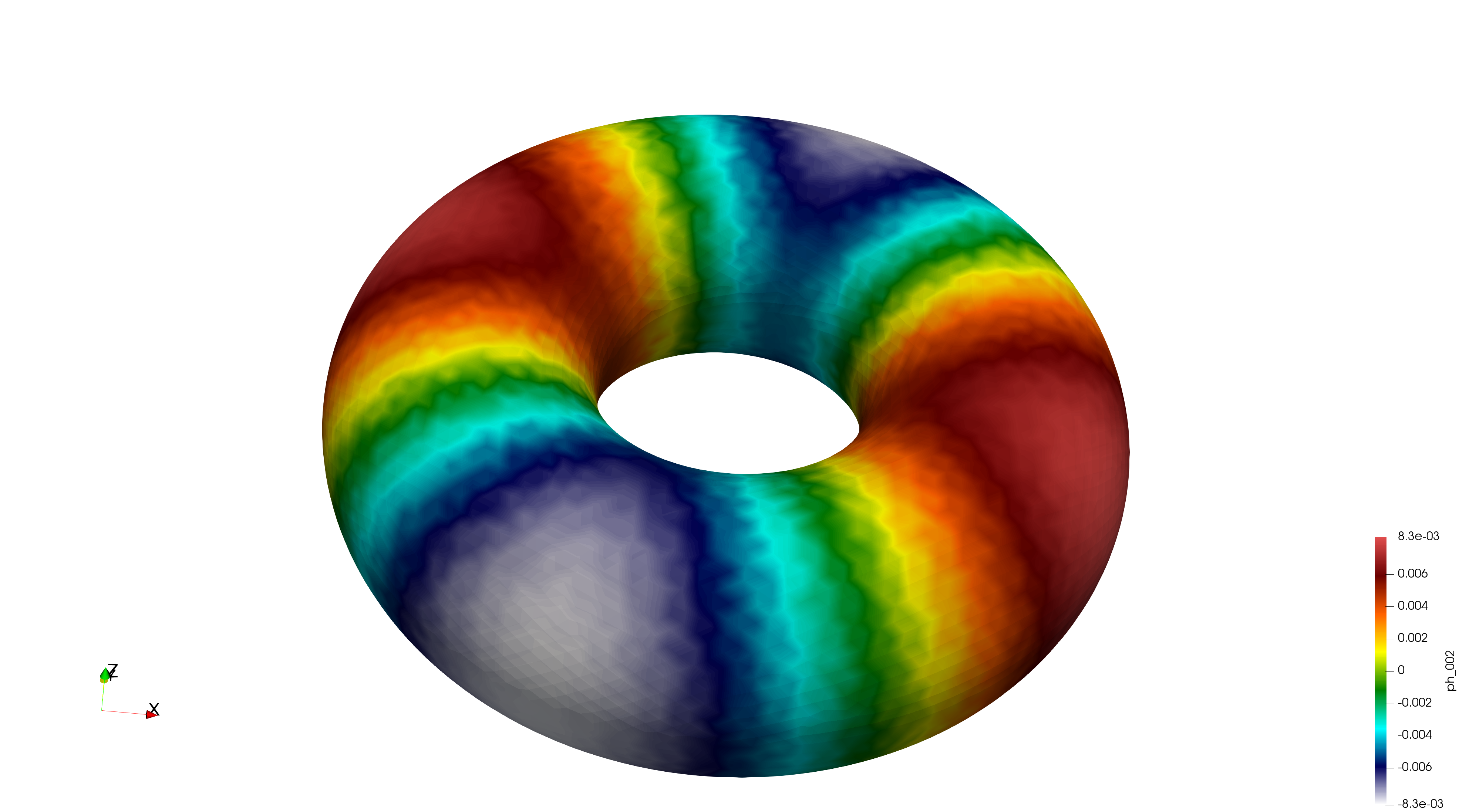}
	\end{minipage}
	\caption{Test 3. Third lowest computed eigenmode $\bu_h$ streamlines (left) and postprocessed pressure (right).}
	\label{fig:toro-modo3}
\end{figure}
\begin{figure}[!h]
	\centering
	\begin{minipage}{0.48\linewidth}\centering
		\includegraphics[scale=0.06,trim= 30cm 2.5cm 30cm 10cm, clip]{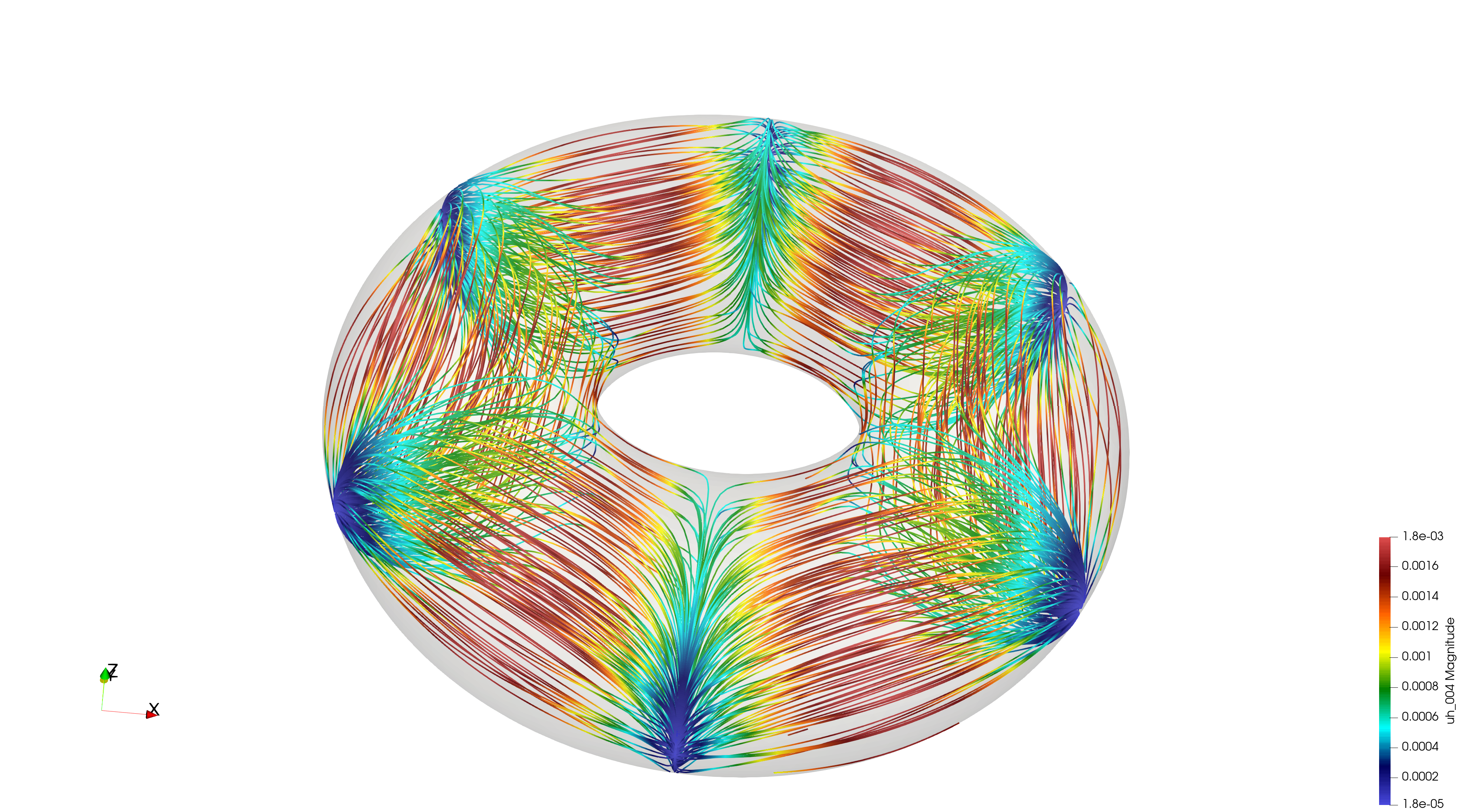}
	\end{minipage}
	\begin{minipage}{0.48\linewidth}\centering
		\includegraphics[scale=0.06,trim= 30cm 2.5cm 30cm 10cm, clip]{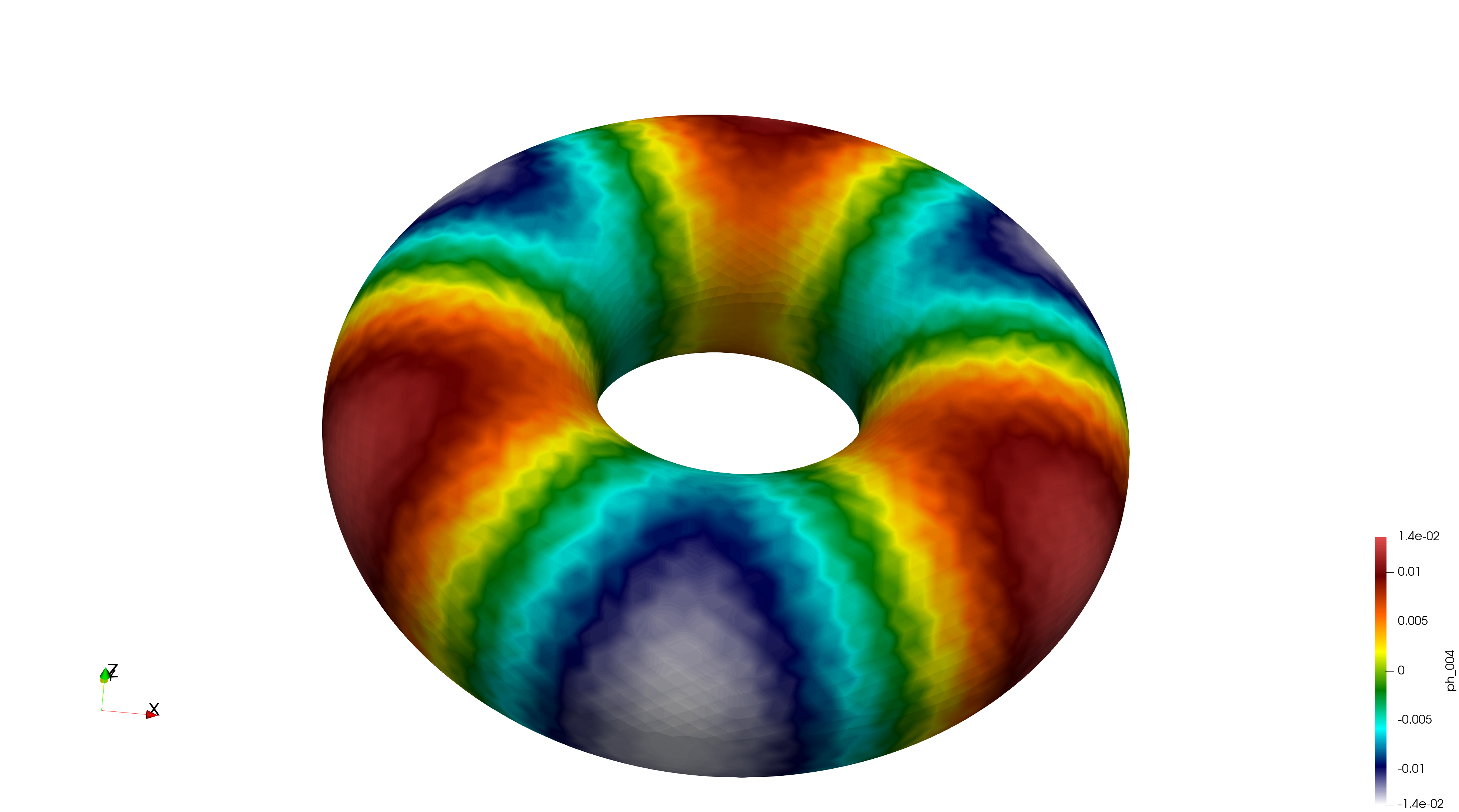}
	\end{minipage}
	\caption{Test 3. Fifth lowest computed eigenmode $\bu_h$ streamlines (left) and postprocessed pressure (right).}
	\label{fig:toro-modo5}
\end{figure}
\section{Test 4. The  pressure formulation}

Since $\varepsilon\in\{-1,0,1\}$ as in the displacement formulation, for the pressure formulation we analyze the spurious
behavior when the DG method is considered. The domain, physical parameters and eigenvalues are the same as those of Test 1. In Tables \ref{tabla:square-SIP-k123-solop}--\ref{tabla:square-SIP-k123-solop2} we report the computed eigenvalues obtained with the SIP method, fixing the mesh refinement in $N=8$ and considering different polynomial degrees and values of $\texttt{a}$. We observe that the results are similar compared with those on Tables \ref{tabla:square-SIP-k123}--\ref{tabla:square-SIP-k123_2}. However, for the smallest stabilization parameters, the pressure formulation gives more spurious eigenvalues when the polynomial degree is $k=1$. 
\begin{table}[!h]
	\centering
	{\setlength{\tabcolsep}{3.8pt}\footnotesize
		\caption{Test 4. Computed eigenvalues $\lambda_{h,i}$ for polynomial degrees $k=1,2,3$, mesh level $N=8$ and different values of $\texttt{a}$ and $\rho(x)$, for the SIP method ($\varepsilon=1$).}
		\begin{tabular}{c|cccc|cccc}
			&\multicolumn{4}{c}{$\rho_1(x,y)$}&\multicolumn{4}{c}{$\rho_2(x,y)$}\\
			\toprule
			\diagbox{k}{a} &$2(\overline{\rho}+\underline{\rho})$ & $4\overline{\rho}$    & $4(\overline{\rho}+\underline{\rho})$    & $8\overline{\rho}$      & $2(\overline{\rho}+\underline{\rho})$ & $4\overline{\rho}$       & $4(\overline{\rho}+\underline{\rho})$       & $8\overline{\rho}$            \\\midrule
			\multirow{10}{0.2cm}{1}
				& 7.85831 & 7.87585 & \fbox{6.28236} & 7.89707 & 8.00791 & 8.00878 & 8.00921 & 8.00965 \\
			& 9.61673 & 9.64631 & 7.99171 & 9.66568 & 10.39198 & 10.39333 & 10.39401 & 10.39470 \\
			& \fbox{11.12628} & \fbox{14.56096} & 9.65834 & 17.78279 & 18.81280 & 18.82297 & 18.82815 & 18.83341 \\
			& 17.58573 & 17.70166 & 17.72448 & 33.38125 & 34.38199 & 34.39806 & 34.40617 & 34.41436 \\
			& 32.74643 & 33.05492 & 33.12913 & 40.25775 & 41.54168 & 41.56437 & 41.57583 & 41.58741 \\
			& \fbox{35.66546} & \fbox{37.81773} & 40.04185 & 43.71383 & 45.60117 & 45.64630 & 45.66926 & 45.69258 \\
			& 39.50152 & 40.23103 & 42.42326 & 49.10155 & 52.15709 & 52.22394 & 52.25799 & 52.29260 \\
			& 42.56402 & 43.36561 & 46.34274 & 76.38572 & 81.30052 & 81.44241 & 81.51355 & 81.58492 \\
			& 47.59880 & 48.28572 & \fbox{51.09714} & 78.06806 & 81.90845 & 82.01209 & 82.06557 & 82.12053 \\
			& \fbox{57.83388} & \fbox{56.48285} & 72.73759 & 89.46875 & 96.68894 & 96.89340 & 96.99754 & 97.10336 \\
			\midrule
			\multirow{10}{0.2cm}{2}
			& 7.83258 & 7.83261 & 7.83264 & 7.83269 & 7.91319 & 7.91319 & 7.91320 & 7.91320 \\
			& 9.59304 & 9.59295 & 9.59300 & 9.59306 & 10.27022 & 10.27022 & 10.27023 & 10.27023 \\
			& 17.42250 & 17.42202 & 17.42253 & 17.42303 & 18.16847 & 18.16851 & 18.16853 & 18.16854 \\
			& 32.34518 & 32.34930 & 32.35180 & 32.35447 & 32.76848 & 32.76863 & 32.76870 & 32.76877 \\
			& 39.14741 & 39.15855 & 39.16265 & 39.16658 & 39.61886 & 39.61907 & 39.61917 & 39.61928 \\
			& 41.88102 & 41.89014 & 41.89664 & 41.90275 & 42.56796 & 42.56834 & 42.56853 & 42.56872 \\
			& 47.02938 & 47.05119 & 47.06061 & 47.06864 & 48.03134 & 48.03189 & 48.03216 & 48.03244 \\
			& \fbox{65.11345} & 71.55640 & 71.59582 & 71.62566 & 72.39744 & 72.39937 & 72.40035 & 72.40134 \\
			& 71.73920 & 73.19101 & 73.21856 & 73.24738 & 73.69022 & 73.69170 & 73.69245 & 73.69320 \\
			& 73.17229 & 82.78775 & 82.83178 & 82.87428 & 83.84312 & 83.84562 & 83.84688 & 83.84816 \\
			\midrule
			\multirow{10}{0.2cm}{3}
			& 7.83253 & 7.83253 & 7.83253 & 7.83253 & 7.91295 & 7.91295 & 7.91295 & 7.91295 \\
			& 9.59288 & 9.59288 & 9.59288 & 9.59288 & 10.26991 & 10.26991 & 10.26991 & 10.26991 \\
			& 17.42060 & 17.42060 & 17.42060 & 17.42060 & 18.16508 & 18.16508 & 18.16508 & 18.16508 \\
			& 32.34439 & 32.34439 & 32.34440 & 32.34440 & 32.75265 & 32.75265 & 32.75265 & 32.75265 \\
			& 39.15585 & 39.15586 & 39.15587 & 39.15587 & 39.60000 & 39.60000 & 39.60000 & 39.60000 \\
			& 41.87267 & 41.87270 & 41.87272 & 41.87274 & 42.52793 & 42.52793 & 42.52793 & 42.52793 \\
			& 47.03560 & 47.03564 & 47.03566 & 47.03568 & 47.97662 & 47.97662 & 47.97662 & 47.97662 \\
			& 71.47621 & 71.47651 & 71.47664 & 71.47679 & 72.19550 & 72.19551 & 72.19551 & 72.19551 \\
			& 73.13362 & 73.13384 & 73.13394 & 73.13406 & 73.51765 & 73.51766 & 73.51766 & 73.51766 \\
			& 82.66900 & 82.66951 & 82.66975 & 82.67002 & 83.53781 & 83.53783 & 83.53784 & 83.53784 \\
			\bottomrule
	\end{tabular}}
	\label{tabla:square-SIP-k123-solop}
\end{table}
\begin{table}[!h]
	\centering
	{\setlength{\tabcolsep}{3.8pt}\footnotesize
		\caption{Test 4. Computed eigenvalues $\lambda_{h,i}$ for polynomial degrees $k=1,2,3$, mesh level $N=8$ and different values of $\texttt{a}$ and $\rho(x)$, for the SIP method ($\varepsilon=1$).}
		\begin{tabular}{c|cccc|cccc}
			&\multicolumn{4}{c}{$\rho_1(x,y)$}&\multicolumn{4}{c}{$\rho_2(x,y)$}\\
			\toprule
			\diagbox{k}{a} &4 & $4(\overline{\rho}+1)$    & 8    & $8(\overline{\rho}+1)$      & 4 & $4(\overline{\rho}+1)$       & 8       & $8(\overline{\rho}+1)$            \\\midrule
			\multirow{10}{0.2cm}{1}
				& 7.87585 & 7.89707 & 7.89707 & 7.91503 & 7.99747 & 8.00897 & 8.00365 & 8.00975 \\
			& 9.64631 & 9.66568 & 9.66568 & 9.68920 & 10.37575 & 10.39362 & 10.38534 & 10.39485 \\
			& \fbox{14.56096} & 17.78279 & 17.78279 & 17.89578 & 18.69974 & 18.82522 & 18.76446 & 18.83457 \\
			& 17.70166 & 33.38125 & 33.38125 & 33.66847 & 34.18685 & 34.40160 & 34.30248 & 34.41615 \\
			& 33.05492 & 40.25775 & 40.25775 & 40.66849 & 41.26919 & 41.56936 & 41.43004 & 41.58995 \\
			& \fbox{37.81773} & 43.71383 & 43.71383 & 44.32293 & 45.09283 & 45.65629 & 45.38552 & 45.69770 \\
			& 40.23103 & 49.10155 & 49.10155 & 49.89565 & 51.41133 & 52.23875 & 51.83909 & 52.30021 \\
			& 43.36561 & 76.38572 & 76.38572 & 78.37028 & 79.63021 & 81.47347 & 80.59982 & 81.60048 \\
			& 48.28572 & 78.06806 & 78.06806 & 79.75233 & 80.75383 & 82.03530 & 81.42428 & 82.13270 \\
			& \fbox{56.48285} & 89.46875 & 89.46875 & 92.09170 & 94.40090 & 96.93872 & 95.71525 & 97.12663 \\
			\midrule
			\multirow{10}{0.2cm}{2}
			& 7.83261 & 7.83269 & 7.83269 & 7.83273 & 7.91316 & 7.91320 & 7.91318 & 7.91320 \\
			& 9.59295 & 9.59306 & 9.59306 & 9.59312 & 10.27018 & 10.27022 & 10.27020 & 10.27023 \\
			& 17.42202 & 17.42303 & 17.42303 & 17.42357 & 18.16806 & 18.16852 & 18.16830 & 18.16855 \\
			& 32.34930 & 32.35447 & 32.35447 & 32.35720 & 32.76667 & 32.76866 & 32.76775 & 32.76879 \\
			& 39.15855 & 39.16658 & 39.16658 & 39.17049 & 39.61629 & 39.61911 & 39.61782 & 39.61930 \\
			& 41.89014 & 41.90275 & 41.90275 & 41.90920 & 42.56350 & 42.56842 & 42.56612 & 42.56876 \\
			& 47.05119 & 47.06864 & 47.06864 & 47.07696 & 48.02489 & 48.03201 & 48.02868 & 48.03250 \\
			& 71.55640 & 71.62566 & 71.62566 & 71.65676 & 72.37484 & 72.39979 & 72.38805 & 72.40155 \\
			& 73.19101 & 73.24738 & 73.24738 & 73.27627 & 73.67168 & 73.69203 & 73.68279 & 73.69336 \\
			& 82.78775 & 82.87428 & 82.87428 & 82.91733 & 83.81269 & 83.84617 & 83.83073 & 83.84844 \\
			\midrule
			\multirow{10}{0.2cm}{3}
			& 7.83253 & 7.83253 & 7.83253 & 7.83253 & 7.91295 & 7.91295 & 7.91295 & 7.91295 \\
			& 9.59288 & 9.59288 & 9.59288 & 9.59288 & 10.26991 & 10.26991 & 10.26991 & 10.26991 \\
			& 17.42060 & 17.42060 & 17.42060 & 17.42060 & 18.16508 & 18.16508 & 18.16508 & 18.16508 \\
			& 32.34439 & 32.34440 & 32.34440 & 32.34441 & 32.75264 & 32.75265 & 32.75264 & 32.75265 \\
			& 39.15586 & 39.15587 & 39.15587 & 39.15588 & 39.60000 & 39.60000 & 39.60000 & 39.60000 \\
			& 41.87270 & 41.87274 & 41.87274 & 41.87275 & 42.52792 & 42.52793 & 42.52793 & 42.52793 \\
			& 47.03564 & 47.03568 & 47.03568 & 47.03570 & 47.97660 & 47.97662 & 47.97661 & 47.97662 \\
			& 71.47651 & 71.47679 & 71.47679 & 71.47695 & 72.19541 & 72.19551 & 72.19546 & 72.19551 \\
			& 73.13384 & 73.13406 & 73.13406 & 73.13419 & 73.51757 & 73.51766 & 73.51762 & 73.51766 \\
			& 82.66951 & 82.67002 & 82.67002 & 82.67031 & 83.53763 & 83.53783 & 83.53774 & 83.53784 \\
			\bottomrule
	\end{tabular}}
	\label{tabla:square-SIP-k123-solop2}
\end{table}

\section{Test 5. A 3D benchmark}
In order to obtain more information of the discrete formulations  \eqref{spectdisc} and \eqref{eq:pression_discreto}, now we compare some aspects of the computational cost of these schemes. To achieve this goal, we resort to the debug report delivered by FENICS through the command \texttt{set\_log\_level(LogLevel.DEBUG)}, where we can recover several debug information, among them, the assembly time, the non-zero entries of the assembled matrices and the sparsity percentage. We solve the eigenproblems in the unit cube $\Omega:=(0,1)^2$ using uniform refinements, where the number of cells scales as $6(N+1)^3$.  It is clear that for the cases studied, the generalized eigenvalue system is of the form $\boldsymbol{K_u}\bu=\lambda\boldsymbol{M_u}\bu$ for \eqref{spectdisc} and $\boldsymbol{K_p}p=\lambda\boldsymbol{M_p}p$ for \eqref{eq:pression_discreto}.  In this experiment, we denote by $K_{0,\bu}$ and $M_{0,\bu}$ the number of non-zero entries of $\boldsymbol{K_u}$ and $\boldsymbol{M_u}$ in each refinement. Similarly, we denote by $K_{0,p}$ and $M_{0,p}$ the number of non-zero entries of $\boldsymbol{K_p}$ and $\boldsymbol{M_p}$. In Figure \ref{fig:bench-tiempo} can observe the behavior of the assembly and solution times, in seconds, of the eigenvalue system in both cases, taken as the average of 5 continuous runs. It is clear that solving \eqref{eq:pression_discreto} is considerably cheaper, mainly because of the number of degrees of freedom ($\texttt{dof}$) in the problem. This is followed by the results depicted in Figure \ref{fig:bench-sparsity}, where we observe the number of non-zero inputs with respect to the number of degrees of freedom, and also the ratio between the non-zero inputs and \texttt{dof}$^2$. The $y-$axis scale is measured with the logarithm of the computed values. Note that the ratio between non-zero entries and $\texttt{dof}^2$ of the matrices associated with \eqref{eq:pression_discreto} are lower than those of \eqref{spectdisc}, which is expected since \eqref{eq:pression_discreto} is associated with fewer degrees of freedoms.
\begin{figure}[!h]
	\centering
	\begin{minipage}{0.48\linewidth}\centering
		\includegraphics[scale=0.17]{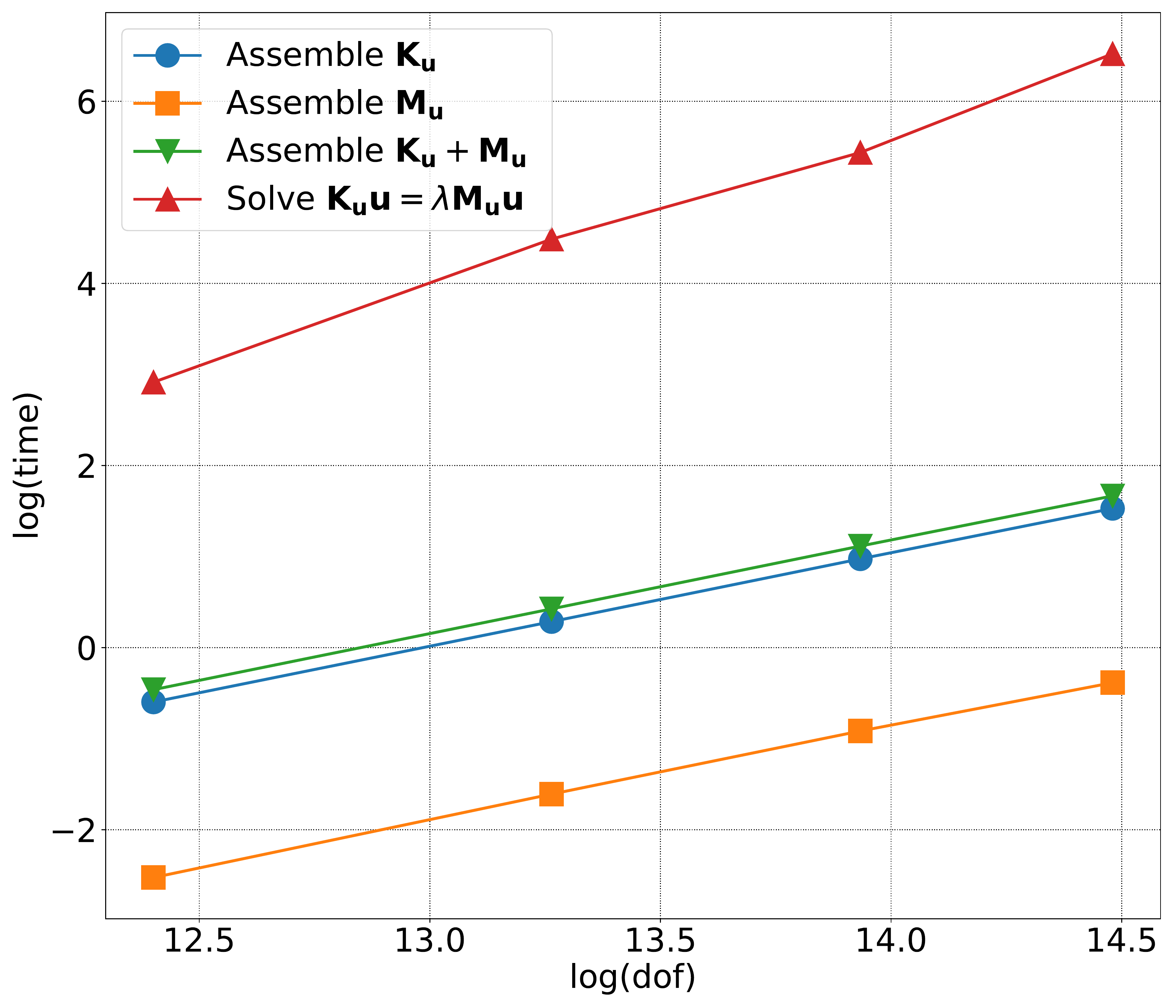}
	\end{minipage}
	\begin{minipage}{0.48\linewidth}\centering
		\includegraphics[scale=0.17]{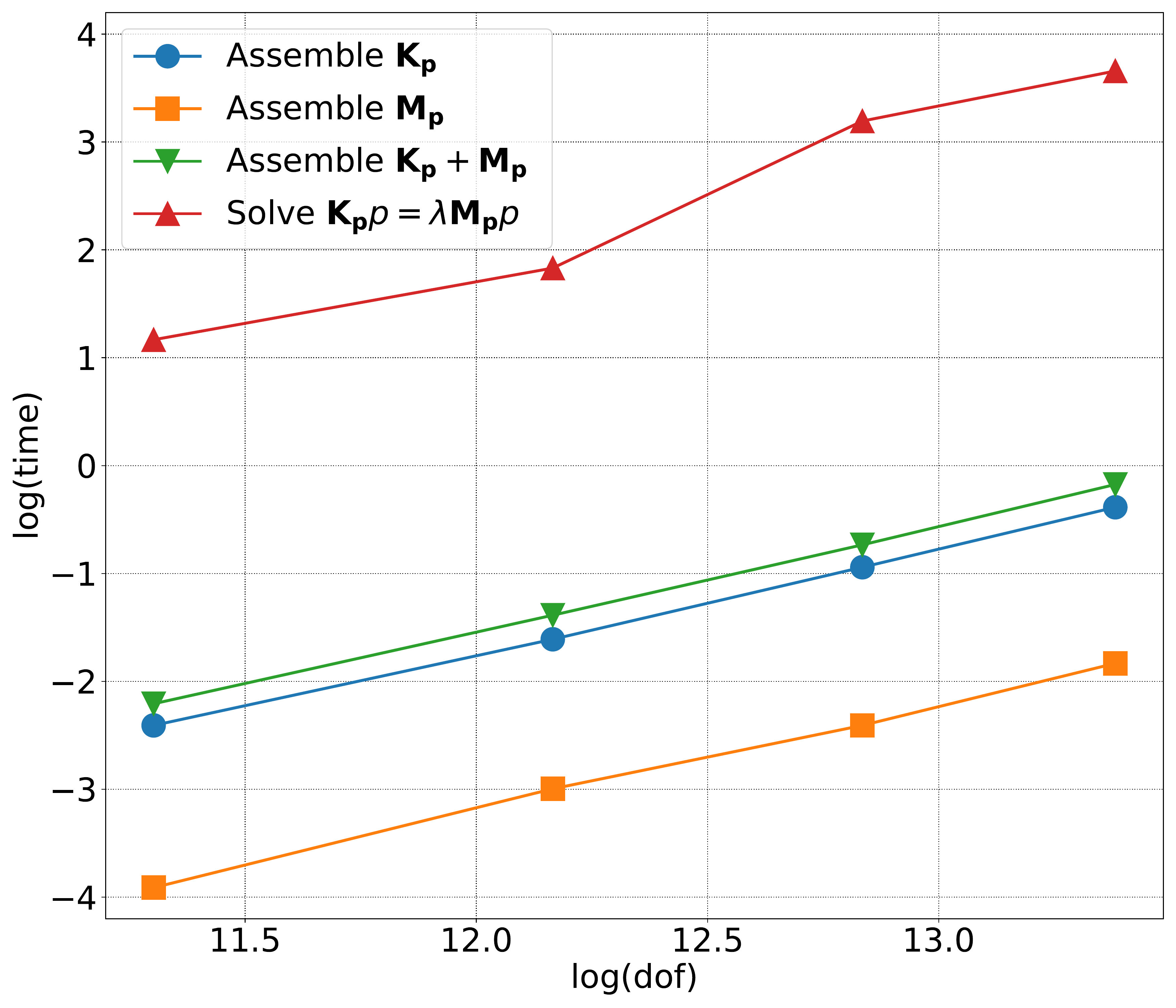}
	\end{minipage}
	\caption{Test 5. Comparison of time employed by the velocity and pressure schemes when assembling and solving the eigenproblems in a 3D unit cube with uniform refinements. Left: velocity formulation. Right: pure pressure formulation.}
	\label{fig:bench-tiempo}
\end{figure}
\begin{figure}[!h]
	\centering
	\begin{minipage}{0.48\linewidth}\centering
		\includegraphics[scale=0.17]{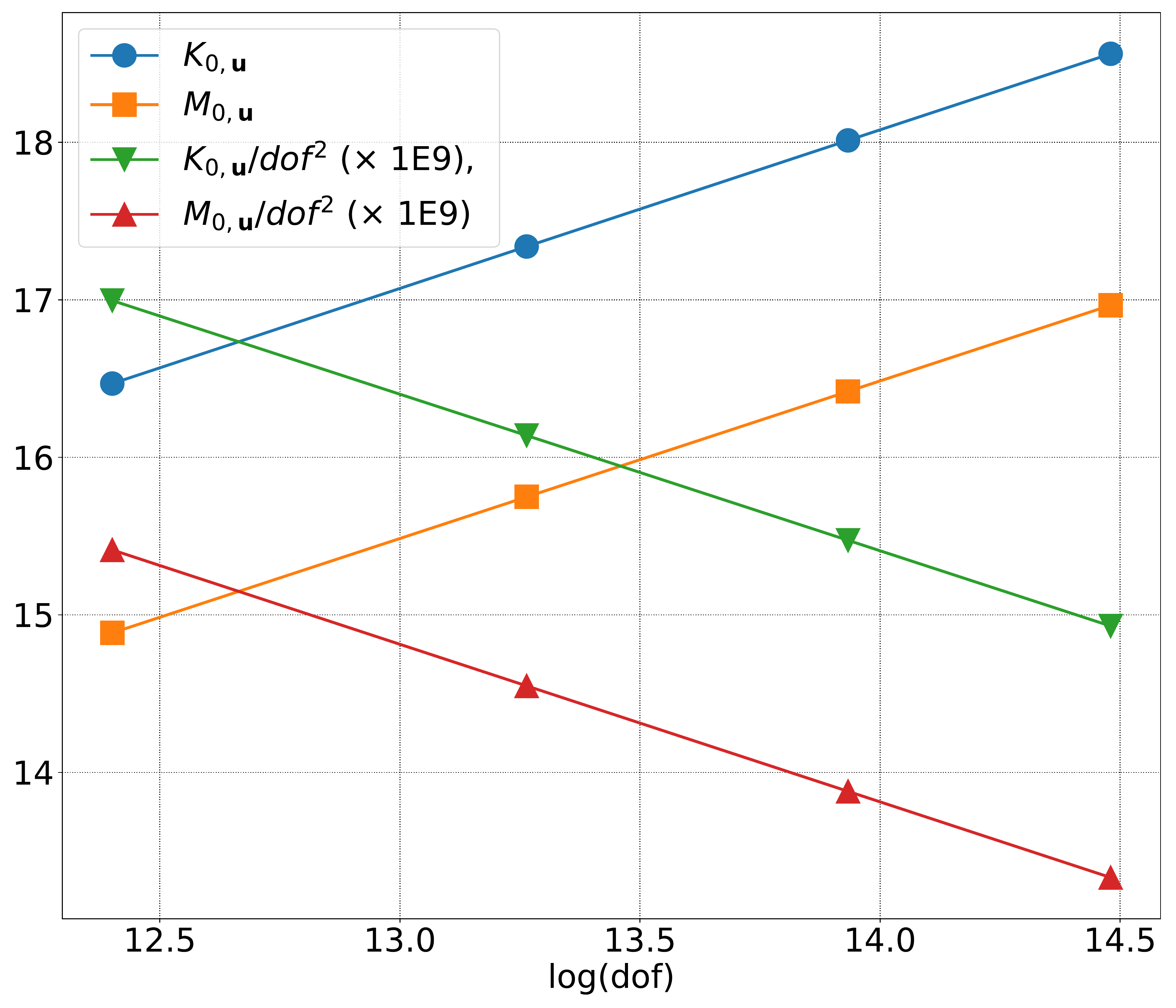}
	\end{minipage}
	\begin{minipage}{0.48\linewidth}\centering
		\includegraphics[scale=0.17]{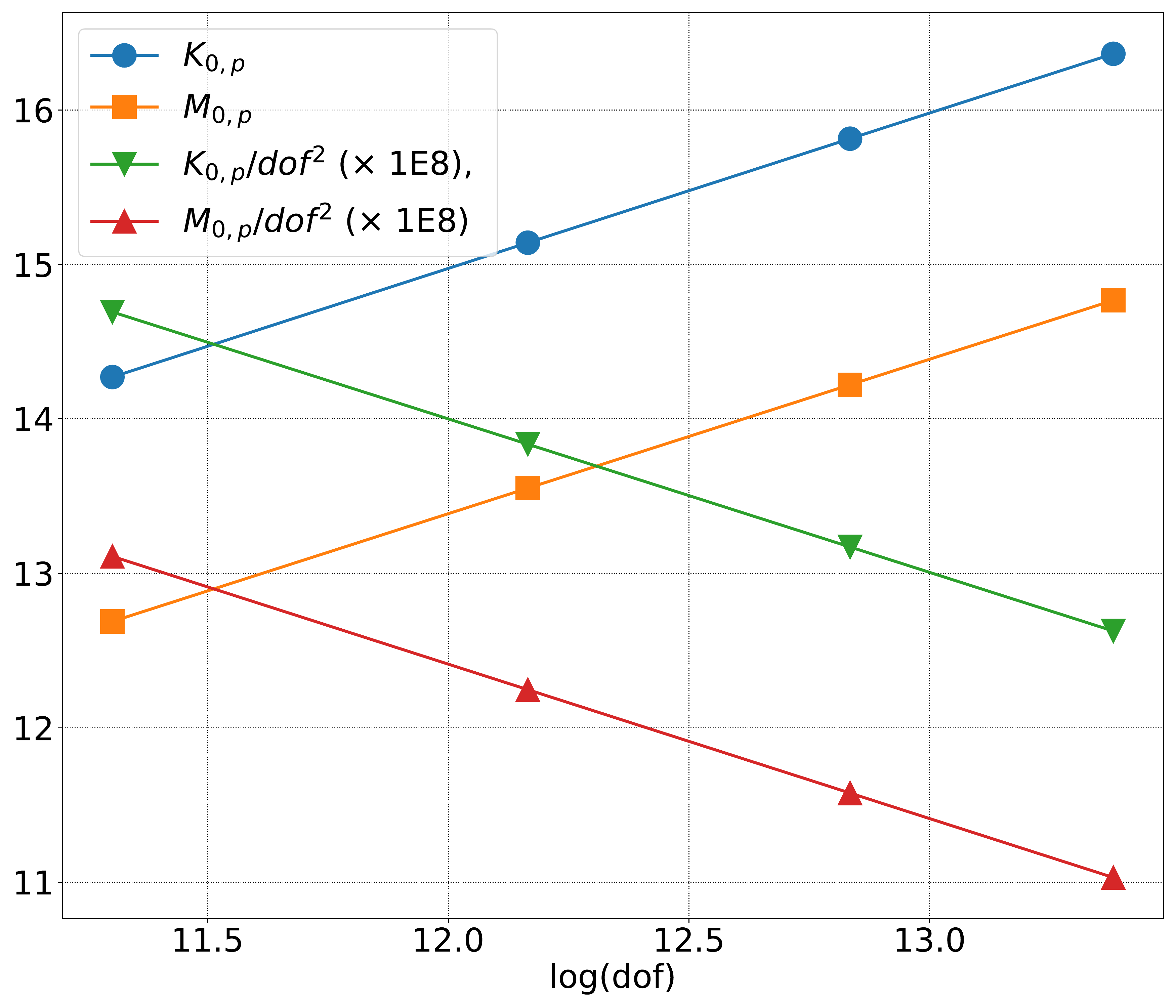}
	\end{minipage}
	\caption{Test 5. Comparison of the sparsity in the the velocity and pressure schemes on each mesh level for the 3D unit cube with uniform refinements. Left: velocity formulation. Right: pure pressure formulation.}
	\label{fig:bench-sparsity}
\end{figure}
\section{Conclusions}
We have presented DG methods to approximate the eigenvalues and eigenfunctions of the acoustic eigenvalue problem. The proposed  symmetric and nonsymmetric methods lead to an accurate computation and approximation of the corresponding eigenfunctions. The acoustic problem has as unknowns the displacement and pressure of the fluid. These DG methods were considered for two formulations: a formulation where the  displacement  is the only unknown and other where the pressure is the unknown. For both approaches, the remaining unknown is always recovered with a postprocess of the solution operator under study. A spurious analysis is presented in two dimensions, where we compare the effects of the stabilization parameter 
in the symmetric and nonsymmetric methods when the spectrum es computed, performing both the displacement and pressure formulations. The numerical tests also reveal that, for a correct choice of the aforementioned stabilization, the convergence order is the expected in two and three dimensions.
On the other hand, a benchmark test shows that the pressure formulation is less expensive in terms of computational costs compared with the displacement formulation. This conclusion is expectable since the pure pressure problem is only scalar. However, in both formulations, the eigenvalues are exactly the same, together with the computed convergence orders for both, the symmetric or nonsymmetric approaches.

\bibliographystyle{siamplain}
\bibliography{bib_LOQ}
\end{document}